\documentclass[a4paper,11pt]{article}
\usepackage[utf8]{inputenc}
\usepackage[T1]{fontenc}
\usepackage{graphicx}
\usepackage{amsmath,amsfonts,amssymb,amsthm}
\usepackage{mathtools}
\usepackage{mathrsfs}
\usepackage{bm}
\usepackage{xcolor}%
\usepackage{fullpage}

\definecolor{myred}{rgb}{0.77, 0.0, 0.1}
\definecolor{crimson}{rgb}{0.86, 0.08, 0.24}
\definecolor{awesome}{rgb}{1.0, 0.13, 0.32}
\definecolor{newgreen}{rgb}{0.0,0.6,0.0}

\usepackage{hyperref}
\hypersetup{
  colorlinks, %
  linkcolor=blue, %
  citecolor=blue, %
  linktoc=all %
}

\renewcommand{\leq}{\leqslant}
\renewcommand{\geq}{\geqslant}
\newcommand{\mleq}{\preccurlyeq}

\newcommand{\di}{\mathrm{d}}

\newcommand{\eps}{\varepsilon}

\newcommand{\argmin}{\mathop{\mathrm{arg}\,\mathrm{min}}}

\newcommand{\wt}{\widetilde}
\newcommand{\wh}{\widehat}
\newcommand{\ol}{\overline}
\newcommand{\ul}{\underline}
\newcommand{\pp}{\, : \;}

\DeclareFontFamily{U}{mathx}{}
\DeclareFontShape{U}{mathx}{m}{n}{<-> mathx10}{}
\DeclareSymbolFont{mathx}{U}{mathx}{m}{n}
\DeclareMathAccent{\widecheck}{0}{mathx}{"71}

\newcommand{\N}{\mathbb N}

\newcommand{\R}{\mathbb R}

\usepackage{xspace}
\newcommand{\ie}{{i.e.}\@\xspace} 
\newcommand{\eg}{e.g.\@\xspace}
\newcommand{\iid}{i.i.d.\@\xspace}

\newcommand{\set}[1]{\{ #1 \}}

\newcommand{\supp}{\mathrm{supp}}
\newcommand{\ceil}[1]{\lceil #1 \rceil}
\newcommand{\floor}[1]{\lfloor #1 \rfloor}

\newcommand{\abs}[1]{|#1|}
\newcommand{\norm}[1]{\|#1\|}

\newcommand{\E}{\mathbb E}
\renewcommand{\P}{\mathbb P}

\newcommand{\probas}{\mathcal{P}}
\newcommand{\ind}{\bm 1}
\newcommand{\indic}[1]{\ind ( #1 )}
\newcommand{\kl}{\mathrm{KL}}%
\newcommand{\kll}[2]{\kl(#1, #2)}%

\newcommand{\binomdist}{\mathsf{Bin}}
\newcommand{\poissondist}{\mathsf{Poisson}}
\newcommand{\expdist}{\mathsf{Exp}}

\newcommand{\F}{\mathcal{F}}

\newcommand{\Subsets}{\mathcal{S}}

\newcommand{\psd}{\probas_{s,d}}
\newcommand{\ad}{\mathrm{ad}}
\newcommand{\pad}{\wh P_n^{\ad}}
\newcommand{\padelta}{\wh P_{n,\delta}^{\ad}}
\newcommand{\adp}{\pad}
\newcommand{\adpdelta}{\padelta}
\newcommand{\altp}{\wh P_n^{\mathrm{HZC}}}
\newcommand{\nobs}{N_{\mathsf{obs}}}
\newcommand{\nmiss}{N_{\mathsf{miss}}}
\newcommand{\ndev}{N_{\mathsf{dev}}}
\newcommand{\nexp}{N_{\mathsf{exp}}}
\newcommand{\setsums}{\mathcal{S}}

\newtheorem{proposition}{Proposition}%
\newtheorem{theorem}{Theorem}
\newtheorem{lemma}{Lemma}
\newtheorem{corollary}{Corollary}
\newtheorem{fact}{Fact}

\theoremstyle{definition}
\newtheorem{definition}{Definition}

\theoremstyle{remark}
\newtheorem{remark}{Remark}

\title{Estimation of discrete distributions in relative entropy, and the deviations of the missing mass}
\author{Jaouad Mourtada\thanks{Department of Statistics, CREST/ENSAE Paris, Palaiseau, France}}
\date{\today}

\begin{document}
\maketitle

\begin{abstract}
  We study the problem of estimating a distribution over a finite alphabet from an i.i.d. sample, with accuracy measured in relative entropy (Kullback-Leibler divergence).
  While optimal bounds on the expected risk are known, high-probability guarantees remain less well-understood.
  First, we analyze the classical Laplace (add-one) estimator, obtaining matching upper and lower bounds on its performance and establishing its optimality among confidence-independent estimators.
  We then characterize the minimax-optimal high-probability risk and show that it is achieved by a simple confidence-dependent smoothing technique.
  Notably, the optimal non-asymptotic risk incurs an additional logarithmic factor compared to the ideal asymptotic rate.
  Next, motivated by regimes in which the alphabet size exceeds the sample size, we investigate methods that adapt to the sparsity of the underlying distribution.
  We introduce an estimator using data-dependent smoothing, for which we establish a high-probability risk bound depending on two effective sparsity parameters.
  As part of our analysis, we also derive a sharp high-probability upper bound on the missing mass.
\end{abstract}

\setcounter{tocdepth}{2}
\tableofcontents

\section{Introduction}
\label{sec:introduction}

\subsection{Problem setting
}
\label{sec:problem-setting}

Estimating a discrete probability distribution from a finite sample is a fundamental problem in
statistics, machine learning, and information theory.
In this work, we consider the following variant of this problem.
Let $P$ be an unknown probability distribution
on the finite set $\{ 1, \dots, d \}$ (identified with the vector $(p_1, \dots, p_d)$, where $p_j$ denotes the probability of the class $j$); %
given access to an \iid sample $X_1, \dots, X_n$ from $P$, find a distribution $\wh P_n = (\wh p_1, \dots, \wh p_d)$ such that the \emph{Kullback-Leibler divergence} or \emph{relative entropy}
\begin{equation}
  \label{eq:kl-p-phat}
  \kll{P}{\wh P_n}
  = \sum_{j=1}^d p_j \log \Big( \frac{p_j}{\wh p_j} \Big)
\end{equation}
is small,
with high probability over the random draw of the \iid sample $X_1, \dots, X_n$ from $P$.

The relative entropy~\eqref{eq:kl-p-phat} is a natural loss function for estimating distributions, which is commonly used in several fields:
in statistics~\cite{vandervaart1998asymptotic,vandegeer1999empirical}, %
due to its connection with maximum likelihood estimation;
in machine learning, as the excess risk for prediction under logarithmic (cross-entropy) loss~\cite{vapnik2000nature,bach2024learning};
in information theory, owing to its interpretation in terms of excess code-length in %
data compression~\cite{cover2006elements,gassiat2018universal,polyanskiy2023information};
and in natural language processing, through its link with the ``perplexity'' metric for evaluating language models~\cite{jurafsky2025speech}.

An important feature of the Kullback-Leibler divergence, compared to other common divergences between probability distributions such as the total variation and Hellinger distances, is that it penalizes significant underestimation of true frequencies.
To consider an extreme case, if the estimator $\wh P_n$ assigns a probability $\wh p_j = 0$ to a class $j$ with probability $p_j \neq 0$, then the Kullback-Leibler divergence $\kll{P}{\wh P_n}$ is infinite.
This
aligns with the needs of various applications: for instance, in a forecasting context where classes %
correspond to different outcomes, assigning a probability of $0$ to outcomes that are in fact possible
would constitute a severe underestimation of the underlying uncertainty.
In addition, in the context of
language models,
one may be interested in %
generating
new sentences %
that are not present in the training corpus; this requires assigning positive probabilities to sequence of words that have not been observed.

\paragraph{Empirical distribution.}

Perhaps the most natural estimator %
is the {empirical distribution} $\ol P_n = (N_j/n)_{1 \leq j \leq d}$, where for $j = 1, \dots, d$, we let
\begin{equation}
  \label{eq:def-class-count}
  N_j
  = \sum_{i=1}^n \indic{X_i = j}
\end{equation}
denote the
number of occurrences of the class $j$ in the sample $X_1, \dots, X_n$.
This estimator happens to coincide with the maximum likelihood estimator (MLE) over the class $\probas_d$ of all probability distributions on $\set{1, \dots, d}$.
As such, it enjoys rather strong optimality properties in the ``low-dimensional'' asymptotic regime, where the number $d$ of classes and distribution $P \in \probas_d$ are fixed, while the sample $n$ goes to infinity~\cite{lecam2000asymptotics}.

In particular, if all classes $j=1, \dots, d$ have nonzero probability, then
$\ol P_n$ converges to $P$ as $n \to \infty$ %
at a rate of $1/\sqrt{n}$ in distribution,
is asymptotically normal and efficient~\cite[\S5.2--5.6]{vandervaart1998asymptotic}.
As a result, $2n \cdot \kll{P}{\ol P_n}$ converges in distribution to a $\chi^2$ distribution with $d-1$ degrees of freedom.
Together with a standard tail bound on $\chi^2$ distributions, this implies the following guarantee: for any fixed $d \geq 2$, $P \in \probas_d$, and $\delta \in (0,1)$, one has
\begin{equation}
  \label{eq:asymptotic-mle}
  \limsup_{n \to + \infty} \P_P \Big( \kll{P}{\ol P_n} \geq \frac{d + 2 \log (1/\delta)}{n} \Big)
  \leq \delta
  \, .
\end{equation}
This guarantee features the optimal dependence on the dimension $d$, confidence level $1-\delta$ and sample size, which 
may serve as a benchmark for an ideal upper bound.

On the other hand, a significant limitation of the guarantee~\eqref{eq:asymptotic-mle} is that it is purely asymptotic, in that it holds in the limit of sufficiently large sample size with all other parameters being kept fixed.
This is at odds with the modern %
paradigm of
high-dimensional models, where the dimension $d$ may be large and possibly comparable to the sample size $n$.
Likewise, one may be interested in high-confidence bounds (that is, in small values in $\delta$),
as well as in guarantees that hold uniformly over all distributions $P \in \probas_d$.
All of these considerations highlight the limitations of purely asymptotic guarantees, and instead call for a quantitative, non-asymptotic analysis.

It should be emphasized that the lack of uniformity over the distribution $P$ of the point-wise asymptotic guarantee~\eqref{eq:asymptotic-mle} is not merely an artifact of its formulation, but instead reflects a fundamental limitation of the estimator $\ol P_n$ itself.
Indeed, the empirical distribution $\ol P_n$ assigns a probability of $0$ to classes that do not appear; such a configuration may occur in a finite sample,
especially in the presence of rare classes.
Hence, the MLE is generally inadequate for the purpose of density estimation in relative entropy, due to its propensity to underestimate uncertainty and to produce overly sharp probability estimates.

\paragraph{Laplace smoothing.}

In order to mitigate this shortcoming of the MLE, a natural approach is to ``smooth out'' or regularize the empirical distribution, by %
assigning some probability to all classes. %
Arguably the simplest and most classical method to achieve this is the \emph{add-one smoothing} technique, also known as \emph{Laplace rule of succession}~\cite{laplace1825essai}, which consists in adding $1$ to the count of each class.
The Laplace estimator is then given by $\wh P_n = (\wh p_1, \dots, \wh p_d)$, where
\begin{equation}
  \label{eq:def-laplace}
  \wh p_j
  = \frac{N_j + 1}{n + d}
  \qquad
  \text{for }
  j = 1, \dots, d
  \, .
\end{equation}
This method was first proposed (in the case $d=2$) by Laplace~\cite[p.~23]{laplace1825essai} in his treatise on probability.
Laplace deduced this estimator from what would now be called a Bayesian approach, of which it constitutes one of the earliest instances.
Indeed, 
it coincides with the Bayes predictive posterior distribution under a uniform prior on the probability simplex $\probas_d$.
This classical method has found use in various fields, including universal coding (see, \eg,~\cite[p.~435]{cover2006elements}) and natural language processing (\eg,~\cite[p.~46]{jurafsky2025speech}).
We also note in passing that
a closely related method
(adding $1/2$ to the count of each class) was proposed
by Krichevsky and Trofimov~\cite{krichevsky1981performance}.

The Laplace estimator $\wh P_n$ turns out to achieve an optimal bound in expectation%
~\cite{catoni1997mixture,mourtada2022logistic}: for any $P \in \probas_d$ and $n, d \geq 2$, one has
\begin{equation}
  \label{eq:bound-expectation-laplace}
  \E_P \big[ \kll{P}{\wh P_n} \big]
  \leq \log \Big( 1 + \frac{d-1}{n+1} \Big)
  \leq \frac{d}{n}
  \, .
\end{equation}
This matches the asymptotic rate of the MLE, but now non-asymptotically and uniformly over all distributions on $\set{1, \dots, d}$.
At the same time, the in-expectation bound~\eqref{eq:bound-expectation-laplace} falls short of constituting a non-asymptotic analogue of the asymptotic tail bound~\eqref{eq:asymptotic-mle}, as it
only provides limited information
on the tails of the estimation error $\kll{P}{\wh P_n}$.

\paragraph{Main questions.}

In this work, we %
investigate
the best possible high-probability guarantees for estimating discrete distributions in relative entropy, either through the Laplace estimator or other procedures.
Specifically, the
previous discussion naturally raises the following questions:
\begin{enumerate}
\item %
  Does there exist a constant $c > 0$ such that, for any $n \geq d \geq 2$ and $\delta \in (0, 1/2)$, there exists an estimator $\wh P_n$ for which
  \begin{equation}
    \label{eq:ideal-whp-bound}
    \sup_{P \in \probas_d} \P_P \bigg( \kll{P}{\wh P_n} \geq c \, \frac{d + \log (1/\delta)}{n} \bigg)
    \leq \delta
    \, ?
  \end{equation}
  If not, then what is the best possible uniform high-probability guarantee?
\item Does the Laplace estimator $\wh P_n$ achieve the ideal high-probability bound~\eqref{eq:ideal-whp-bound}?
  If not, then what is the best high-probability guarantee for the Laplace estimator?
\end{enumerate}

\subsection{Existing guarantees}
\label{sec:exist-high-prob}

The previous questions will be addressed in the following sections, but before this, we survey existing high-probability guarantees for estimation of discrete distributions in relative entropy.

\paragraph{High-probability guarantees for the Laplace estimator.}

As a starting point,
the in-expectation bound~\eqref{eq:bound-expectation-laplace} for the Laplace estimator $\wh P_n$ implies (by Markov's inequality) the following bound: for every distribution $P \in \probas_d$, with probability at least $1-\delta$ one has
\begin{equation}
  \label{eq:laplace-markov}
  \kll{P}{\wh P_n}
  \leq \frac{d}{n \delta}
  \, .
\end{equation}
However, this naïve bound is significantly worse than the ideal asymptotic bound~\eqref{eq:asymptotic-mle}.
For instance, it only shows that a bound of order $d/n$ holds with constant probability, rather than (asymptotically) with probability at least $1 -e^{-d}$ as in~\eqref{eq:ideal-whp-bound}.

A sequence of recent works has progressively tightened the bound~\eqref{eq:laplace-markov}.
First, Bhattacharyya, Gayen, Price, and Vinodchandran~\cite[Theorem~6.1]{bhattacharyya2021near} established a concentration inequality for $\kll{P}{\wh P_n}$, which implies the following bound: for every $P \in \probas_d$, with probability $1-\delta$,
\begin{equation}
  \label{eq:bhattacharyya-laplace}
  \kll{P}{\wh P_n}
  \lesssim \frac{d \log (n) \log (d/\delta)}{n}
  \, ,
\end{equation}
where we use the notation $A \lesssim B$ to mean that $A \leq c B$ for some universal constant $c$.
While this guarantee significantly improves over the bound~\eqref{eq:laplace-markov} for small values of $\delta$, it falls short of the asymptotic bound~\eqref{eq:asymptotic-mle} due to the fact that the dimension $d$ and deviation term $\log (1/\delta)$ are multiplied, rather than decoupled as in~\eqref{eq:ideal-whp-bound}.

A significantly improved %
bound was established by Han, Jana and Wu~\cite[Lemma~17]{han2023optimal}, who showed that, with probability $1-\delta$,
\begin{equation}
  \label{eq:han-laplace}
  \kll{P}{\wh P_n}
  \lesssim \frac{d + \sqrt{d} \log^{3} (1/\delta)}{n}
  \, .
\end{equation}
This implies an optimal bound of $d/n$ when $\log (1/\delta) \lesssim d^{1/6}$, but leads to (presumably) suboptimal guarantees in the regime $\log (1/\delta) \gg d^{1/6}$.

Finally, the best available high-probability guarantee on the Laplace estimator is due to Canonne, Sun and Suresh~\cite{canonne2023concentration}, who showed that for some absolute constant $c> 0$, with probability at least $1-\delta$,
\begin{equation}
  \label{eq:canonne-laplace}
  \kll{P}{\wh P_n}
  \leq \E_P \big[ \kll{P}{\wh P_n} \big] + c \, \frac{\sqrt{d} \log^{5/2} (d/\delta)}{n}
  \lesssim \frac{d + \sqrt{d} \log^{5/2} (1/\delta)}{n}  
\end{equation}
(where we used~\eqref{eq:bound-expectation-laplace} and that $\sqrt{d} \log^{5/2} d \lesssim d$).
In particular, this removes a $\sqrt{\log (1/\delta)}$ factor in the deviation term compared to~\eqref{eq:han-laplace}, leading to optimal guarantees in the larger range $\log (1/\delta) \lesssim d^{1/5}$.
Nevertheless, the bound still deteriorates in the regime $\log (1/\delta) \gg d^{1/5}$.
This being said, unlike other results discussed above, the guarantee from~\cite{canonne2023concentration} establishes concentration of the error $\kll{P}{\wh P_n}$ around its expectation, rather than merely a deviation bound.

\paragraph{Upper bound for an alternative estimator.}

Recently, van der Hoeven, Zhivotovskiy and Cesa-Bianchi~\cite{vanderhoeven2023high} proposed an alternative estimator $\altp$, based on a high-probability online-to-batch conversion scheme, which achieves the following high-probability guarantee: for any distribution $P$, with probability at least $1-\delta$ one has
\begin{equation}
  \label{eq:hoeven-bound}
  \kll{P}{\altp}
  \lesssim \frac{d + \log (n) \log (1/\delta)}{n}
  \, .
\end{equation}
In many regimes of interest, this constitutes the best known high-probability guarantee in the literature, for any estimator.
On the other hand, this result raises important questions.
First, the bound~\eqref{eq:hoeven-bound} features an additional $\log n$ factor compared to the asymptotic rate~\eqref{eq:asymptotic-mle}, which can be avoided in some regimes (\eg, in light of~\eqref{eq:canonne-laplace}), leaving open the question of the best possible statistical guarantees.
Second, the estimator $\altp$ is computationally involved: it requires integrating certain functions over the probability simplex, the cost of which appears to be super-linear in $\max(d, n)$ via a sampling approach.
This raises the question of whether or not the problem of high-probability estimation of discrete distributions exhibits a statistical-computational trade-off, that is, if a computational cost super-linear in $n$ is necessary to achieve a guarantee as strong as~\eqref{eq:hoeven-bound}.

Finally, after a first version of the present paper was disseminated, a concurrent work of van der Hoeven, Olkhovskaia and van Erven~\cite{vanderhoeven2025nearly} simplified the estimator from~\cite{vanderhoeven2023high} and refined its analysis, obtaining a guarantee of order $\set{d \log \log d + \log (d) \log (1/\delta)}/n$ which is near-optimal up to the $\log \log d$ factor.

\subsection{Paper %
  outline}
\label{sec:paper-outline}

This paper is organized as follows.
In Section~\ref{sec:laplace}, we describe the best possible high-probability guarantee on the Laplace estimator (Theorems~\ref{thm:upper-laplace} and~\ref{thm:lower-bound-conf-indep-tail}), and show in particular that this method is optimal among ``confidence-independent'' estimators.
In Section~\ref{sec:confidence-minimax}, we characterize the best possible uniform guarantees for any estimator (Theorem~\ref{thm:upper-bound-conf-dependent} and~\ref{thm:lower-bound-minimax}); the upper bound is achieved by a simple modification of the Laplace estimator using a confidence-dependent smoothing level.

In Section~\ref{sec:adaptation-support}, in order to handle situations where the total number of classes is very large, we study guarantees that depend on the ``effective sparsity'' of the distribution at hand.
We establish in particular a minimax lower bound for estimating sparse distributions that holds with high probability (Proposition~\ref{prop:lower-bound-minimax-sparse}), and then propose simple estimators using data-dependent smoothing that achieve high-probability guarantees (Theorem~\ref{thm:upper-bound-sparse}); these guarantees adapt to two natural ``effective sparsity'' parameters of the distribution.
In Section~\ref{sec:missing-mass}, we present a sharp high-probability bound on the missing and underestimated masses (Theorem~\ref{thm:deviation-missing-mass}), which is used in our analysis of the sparse case but may also be of independent interest.

The proof of our high-probability upper bounds for estimation are provided in Section~\ref{sec:proof-high-prob}, while Section~\ref{sec:proof-lower-bounds} contains the proofs of lower bounds for estimation.
Section~\ref{sec:proof-deviation-missing-mass} is devoted to the proof of Theorem~\ref{thm:deviation-missing-mass} on the missing mass, and Section~\ref{sec:proof-expected-sparse-loo} to the elementary proof of the in-expectation guarantee of Proposition~\ref{prop:expected-sparse-loo}.
Finally, Section~\ref{sec:technical-lemmata} gathers various technical lemmata.

\subsection{Related work}
\label{sec:related-work}

\paragraph{High-probability guarantees in relative entropy.}

As discussed in Section~\ref{sec:exist-high-prob}, our contribution belongs to a series of recent works~\cite{bhattacharyya2021near,han2023optimal,canonne2023concentration,vanderhoeven2023high}
on high-probability guarantees for estimation of discrete distributions in relative entropy.
The best known guarantee for the Laplace estimator is the upper bound~\eqref{eq:canonne-laplace} from~\cite{canonne2023concentration}, while the best guarantee (in many regimes) for any estimator is the upper bound~\eqref{eq:hoeven-bound} from~\cite{vanderhoeven2023high}.

\paragraph{Other aspects of estimation of discrete distributions.}

Naturally, estimation of discrete distributions is a basic problem, which has been investigated from various other perspectives in the literature.
First, one may consider different loss functions than the relative entropy; we refer to~\cite{kamath2015learning,canonne2020short} (and references therein) for an overview of existing guarantees under various losses.
Second, even in relative entropy,
the minimax-optimal in-expectation bound~\eqref{eq:bound-expectation-laplace} has been refined in various ways.
Braess and Sauer~\cite{braess2004bernstein}
characterize
asymptotically optimal numerical constants in the minimax expected relative entropy risk, in the regime where $d$ is fixed while $n \to \infty$.
In another direction, Orlitsky and Suresh~\cite{orlitsky2015competitive}
consider a more demanding competitive optimality criterion, in the spirit of the empirical Bayes paradigm~\cite{robbins1951compound,good1953population}.

\paragraph{Concentration properties of the empirical distribution.}

A relevant but distinct question concerns the concentration properties of the empirical distribution $\ol P_n$.
As discussed above, the relative entropy $\kll{P}{\ol P_n}$ does not enjoy distribution-free concentration properties, since $\ol P_n$ may assign zero probability to classes with positive true probability.
On the other hand, the opposite configuration cannot happen: a class $j$ with $p_j=0$ cannot appear in the sample, which qualitatively suggests that the \emph{reverse} relative entropy $\kll{\ol P_n}{P}$ may be well-behaved.
This is indeed the case: the theory of large deviations suggests that the reverse relative entropy $\kll{\ol P_n}{P}$ sharply encodes the concentration properties of the empirical distribution.
Specifically, a classical inequality~\cite[Theorem~11.2.1 p.~356]{cover2006elements} established by the so-called ``method of types''~\cite{csiszar1998method} states that, for any $n,d \geq 2$, $P \in \probas_d$ and $\delta \in (0, 1)$, one has
\begin{equation}
  \label{eq:deviation-types}
  \P_P \bigg( \kll{\ol P_n}{P} \geq \frac{d \log (n+1) + \log (1/\delta)}{n} \bigg)
  \leq \delta
  \, .
\end{equation}
While this bound is distribution-free, it features a presumably suboptimal $\log (n+1)$ factor.
This classical bound has been tightened in a series of recent works~\cite{mardia2020concentration,agrawal2020finite,guo2020chernoff,bhatt2023sharp,agrawal2022finite} on concentration of the reverse relative entropy.
In particular, it follows from~\cite[Corollary~1.7]{agrawal2022finite} (although this could also be deduced up to constants from the earlier work~\cite{agrawal2020finite}) that, for any $n, d \geq 2$, $P \in \probas_d$ and $\delta \in (0,1)$,
\begin{equation}
  \label{eq:deviation-agrawal}
  \P_P \bigg( \kll{\ol P_n}{P} \geq \frac{6 d + 6 \log (1/\delta)}{n} \bigg)
  \leq \delta
  \, .
\end{equation}
This deviation bound effectively settles the probabilistic question of optimal concentration of the empirical distribution.
In this work, we study the complementary statistical question
of optimal high-probability estimation guarantees in relative entropy $\kll{P}{\wh P_n}$.

\paragraph{Missing mass.}

As part of our analysis, we study the tail behavior of the ``missing mass''.
We refer to Section~\ref{sec:missing-mass} for a discussion of existing high-probability bounds on this quantity.

\subsection{Notation}
\label{sec:notation}

Throughout this work, we let $n \geq 1$ denote the sample size and $d \geq 2$ the number of classes.
If $A$ is a finite set, we denote by $|A|$ its cardinality.
We identify the set of probability distributions on $[d] = \set{1, \dots, d}$ with the set of probability vectors $\probas_d = \{ (p_1, \dots, p_d) \in \R_+^d : \sum_{j=1}^d p_j = 1 \}$, where for $1 \leq j \leq d$ we let $p_j$ denote the probability of the class $j$.
For $j \in \set{1, \dots, d}$, we let $\delta_j \in \probas_d$ denote the Dirac mass at $j$, identified with the $j$-th basis vector in $\R^d$.
Given two probability distributions $P = (p_1, \dots, p_d) \in \probas_d$ and $Q = (q_1, \dots, q_d)\in \probas_d$, we define the \emph{Kullback-Leibler divergence} or \emph{relative entropy} between $P$ and $Q$ by
\begin{equation*}
  \kll{P}{Q}
  = \sum_{j=1}^d p_j \log \Big( \frac{p_j}{q_j} \Big)
  \, ,
\end{equation*}
with the convention that $p \log (p/q)$ equals $0$ if $p = 0$, and $+ \infty$ if $p > 0$ but $q = 0$.
For $u, v \in \R^+$, we let $D (u, v) = u \log (\frac{u}{v}) - u + v$
with similar conventions.
Since $\sum_{j=1}^d p_j = \sum_{j=1}^d q_j = 1$, we have
\begin{equation}
  \label{eq:kl-D-function}
  \kll{P}{Q}
  = \sum_{j=1}^d D (p_j, q_j)
  \, .
\end{equation}
Additionally we define the function $h : \R^+ \to \R^+$ by $h(x) = x \log x - x + 1$ for $x > 0$ and $h (0) = 1$, so that $D (u,v) = v \cdot h (u/v)$ for $u,v \in \R^+$.

Given a distribution $P \in \probas_d$, the \emph{sample} $X_1, \dots, X_n$ is comprised of $n$ \iid random variables with distribution $P$.
We use the notations $\P_P$ and $\E_P$ to respectively denote probabilities and expectations when the %
distribution of $(X_1, \dots, X_n)$ is $P^{\otimes n}$.
For $j=1, \dots, d$, we defined the count of the class $j$ as its number of occurrences in the sample $X_1, \dots, X_n$, namely
\begin{equation}
  \label{eq:def-count}
  N_j
  = N_{j,n}
  = \sum_{i=1}^n \indic{X_i = j}
  \, .
\end{equation}
An \emph{estimator} is a map $\Phi : [d]^n \to \probas_d$, which we identify (following a standard convention) with the random variable
$\wh P_n = \Phi (X_1, \dots, X_n)$ taking values in $\probas_d$.

Finally, for any $\lambda \in \R^+$, we denote by $\poissondist(\lambda)$ the Poisson distribution with intensity $\lambda$, which assigns a probability of $e^{-\lambda} \lambda^k / k!$ to any non-negative integer $k \in \N$.

\section{Optimal
  guarantees for the Laplace estimator
}
\label{sec:laplace}

In this section, we consider the question of optimal high-probability guarantees for the classical Laplace (add-one) estimator, defined by~\eqref{eq:def-laplace}.
In Section~\ref{sec:upper-laplace}, we state our main upper bound, while in Section~\ref{sec:lower-bound-conf-indep} we provide a matching lower bound for a large class of estimators that includes the Laplace estimator.

\subsection{Upper bound for the Laplace estimator}
\label{sec:upper-laplace}

Our first main result is a finite-sample high-probability bound for the Laplace estimator.

\begin{theorem}
  \label{thm:upper-laplace}
  For any $n \geq 12, d \geq 2$ and $P \in \probas_d$, the Laplace estimator $\wh P_n$ defined by~\eqref{eq:def-laplace} achieves the following guarantee: for any $\delta \in (e^{-n/6}, e^{-2})$,
  \begin{equation}
    \label{eq:upper-laplace}
    \P_P \bigg( \kll{P}{\wh P_n}
    \geq %
    110000 \, \frac{ d + \log (1/\delta) \log \log (1/\delta)}{n}
    \bigg)
    \leq 4 \delta
    \, .
  \end{equation}
\end{theorem}

The proof of Theorem~\ref{thm:upper-laplace} is provided in Section~\ref{sec:proof-upper-laplace}.

We note in passing that the condition $\delta > e^{-n/6}$ in Theorem~\ref{thm:upper-laplace} is not restrictive, as it constitutes the nontrivial regime.
Indeed, when $n \geq d$ and $\delta = e^{-n/6}$, the upper bound~\eqref{eq:upper-laplace} is of order $\log n$.
But an upper bound $\kll{P}{\wh P_n} \leq \log (2n)$ actually holds deterministically
(thus for $\delta = 0$)
since for every $j = 1, \dots, d$ one has $\wh p_j \geq 1/(n+d) \geq 1/(2n)$, so that $p_j / \wh p_j \leq 2n$.

Theorem~\ref{thm:upper-laplace} improves the previously best known upper bound~\eqref{eq:canonne-laplace} on the Laplace estimator from~\cite{canonne2023concentration}, which is of order $\{ d+ \sqrt{d} \log^{5/2} (1/\delta) \}/n$.
Since $\delta > e^{-n/6}$ %
and thus $\log \log (1/\delta) \leq \log n$, it also improves the previously best known upper bound~\eqref{eq:hoeven-bound} %
of order $\{ d+ \log (n) \log (1/\delta) \}/n$ for this problem,
achieved by the (computationally involved) estimator from~\cite{vanderhoeven2023high}.
This shows in particular that such guarantees can be achieved in a computationally efficient manner, specifically
in time linear in $n$.

A curious feature of the upper bound~\eqref{eq:upper-laplace} is that it exhibits non-standard tails, in the form of the $\log (1/\delta) \log \log (1/\delta)/n$ deviation term.
This should be
contrasted with the more common quantiles of exponential and Poisson variables, respectively of order $\log (1/\delta)$ and $\frac{\log (1/\delta)}{\log \log (1/\delta)}$.
In particular, this tail bound is \emph{super-exponential}, which points to a technical
difficulty
in its proof: it cannot be established by the standard Chernoff method based on the moment generating function (m.g.f.).
Indeed, super-exponential tails lead to an infinite m.g.f.\footnote{Technically speaking, the m.g.f.~of the variable $\kll{P}{\wh P_n}$ is actually finite, due to the deterministic bound $\kll{P}{\wh P_n} \leq \log (2n)$.
  However, this (loose) bound depends on $n$, hence bounds based on the m.g.f.~would lead to deviation terms with an additional dependence on $n$ compared to~\eqref{eq:upper-laplace}.}, and conversely a finite m.g.f.~would lead to sub-exponential tails.
For this reason,
the proof does not use the m.g.f.~and instead proceeds by controlling raw moments ($L^p$ norms).

We refer to Section~\ref{sec:proof-high-prob} for a description of the main tools in the proof of Theorem~\ref{thm:upper-laplace}, which also serve for high-probability upper bounds stated in subsequent sections.
Roughly speaking, the key step in the analysis is to control the contribution to the error $\kll{P}{\wh P_n}$ of classes $j = 1, \dots, d$ for which the Laplace estimate significantly underestimates the true probability, as such classes may lead to large errors.

\subsection{Lower bound for confidence-independent estimators}
\label{sec:lower-bound-conf-indep}

It should be noted that the uniform non-asymptotic
high-probability bound of Theorem~\ref{thm:upper-laplace} for the Laplace estimator exceeds the asymptotic tail bound~\eqref{eq:asymptotic-mle} of the MLE (or Laplace estimator) by a factor of $\log \log (1/\delta)$ in the deviation term.
This raises the question of whether this extra factor is necessary, or whether it can be removed by a more precise analysis.

As it turns out, the extra $\log \log (1/\delta)$ factor is necessary, not only for the Laplace estimator but in fact for any ``confidence-independent'' estimator $\wh P_n = \Phi (X_1, \dots, X_n)$ that does not depend on the desired confidence level $1-\delta$.

\begin{theorem}
  \label{thm:lower-bound-conf-indep-tail}
  Let $n \geq d \geq 4000$ %
  and $\kappa \geq 1$.
  Let $\Phi : [d]^n \to \probas_d$ be an estimator such that, denoting $\wh P_n = \Phi (X_1, \dots, X_n)$ we have for any $P \in \probas_d$:
  \begin{equation}
    \label{eq:assumption-in-expectation-opt}
    \P_P \Big( \kll{P}{\wh P_n} \leq \frac{\kappa d}{n} \Big)
    > 0
    \, .
  \end{equation}
  Then, for any $\delta \in (e^{- n}, e^{-16 \kappa^2})$, there exists a distribution $P \in \probas_d$ such that
  \begin{equation}
    \label{eq:lower-bound-conf-indep}
    \P_P \bigg( \kll{P}{\wh P_n} \geq \frac{d + \log (1/\delta) \log \log (1/\delta)}{5000 \,n} \bigg)
    \geq \delta
    \, .
  \end{equation}
\end{theorem}

The proof of Theorem~\ref{thm:lower-bound-conf-indep-tail} (based on Lemma~\ref{lem:lower-bound-conf-indep-tail}) can be found in Section~\ref{sec:proof-theorem}.
The main idea of the proof is that, in order to achieve the guarantee~\eqref{eq:assumption-in-expectation-opt}, the estimator $\wh P_n$ cannot be ``too far'' from the empirical distribution $\ol P_n$.
But in this case, it may significantly underestimate the probability of some classes with small probability $\delta$, leading to the lower bound~\eqref{eq:lower-bound-conf-indep}.

A few comments may help clarify the meaning of Theorem~\ref{thm:lower-bound-conf-indep-tail}.
First, one should think of the parameter $\kappa \geq 1$ as an absolute constant.
The condition~\eqref{eq:assumption-in-expectation-opt} states that the estimator $\wh P_n$ achieves a bound of order $d/n$ with positive probability for any distribution $P$.
This holds in particular when the estimator $\wh P_n$ achieves an optimal in-expectation bound $\E_P [\kll{P}{\wh P_n}] \leq \kappa d/n$,
and more generally when the estimator achieves an optimal bound of $\kappa d/n$ in the regime of constant (bounded away from $0$ and $1$) confidence level.
The first condition applies with $\kappa = 1$ to the Laplace estimator, in light of its in-expectation bound~\eqref{eq:bound-expectation-laplace}.
We refer to such estimators, including those that are optimal in expectation,
as ``confidence-independent''.

The content of Theorem~\ref{thm:lower-bound-conf-indep-tail} is that such an estimator must necessarily incur the $\log \log (1/\delta)$ factor for small values of $\delta$.
In other words, the extra $\log \log (1/\delta)$ factor
in the high-confidence regime
is a necessary price to pay for optimality in the constant-confidence regime.

When compared with the upper bound of Theorem~\ref{thm:upper-laplace}, the lower bound of Theorem~\ref{thm:lower-bound-conf-indep-tail} implies
that the Laplace estimator is optimal in a minimax sense, over a large class of estimators.
This may be of interest
in itself given the simplicity of this procedure.

\section{Minimax-optimal guarantees for confidence-dependent estimators}
\label{sec:confidence-minimax}

An obvious
restriction in the lower bound of Theorem~\ref{thm:lower-bound-conf-indep-tail} is that it only applies to ``confidence-independent'' estimators---in particular, to those that achieve an optimal $d/n$ guarantee with constant probability.
This leaves open the possibility that, for a given $\delta \in (0,1/2)$,
improved guarantees with probability $1-\delta$ may be achieved by an estimator $\wh P_{n,\delta} = \Phi_\delta (X_1, \dots, X_n)$ that depends on $\delta$; that is, %
which is
tuned for the desired confidence level $1-\delta$, at the cost of being suboptimal at constant confidence levels.
We call such an estimator ``confidence-dependent''.

In this section, we investigate optimal high-probability guarantees for confidence-dependent estimators; Section~\ref{sec:upper-bound-conf} is dedicated to the upper bound,
and Section~\ref{sec:lower-bound-conf} to the lower bound.

\subsection{Upper bound
  via confidence-dependent smoothing
}
\label{sec:upper-bound-conf}

Since the gap between the asymptotically ideal tail bound~\eqref{eq:asymptotic-mle} and the non-asymptotic upper and lower bounds of Section~\ref{sec:laplace} consists in an extra $\log \log (1/\delta)$ factor in the deviation term, the question is whether this factor can be improved by a confidence-dependent estimator.
Theorem~\ref{thm:upper-bound-conf-dependent} below
answers this question in the affirmative:

\begin{theorem}
  \label{thm:upper-bound-conf-dependent}
  For any $n \geq 12, d \geq 2$ and $\delta \in (e^{-n/6}, e^{-2})$, define the estimator $\wh P_{n,\delta} = (\wh p_1, \dots, \wh p_d)$ by, for $j=1, \dots, d$,
  \begin{equation}
    \label{eq:def-conf-dependent}
    \wh p_j = \frac{N_j + \lambda_\delta}{n + \lambda_\delta d}
    \qquad \text{where} \qquad
    \lambda_\delta
    = \max \bigg\{ 1, \frac{\log (1/\delta)}{d} \bigg\}
    \, .
  \end{equation}
  Then, for any $P \in \probas_d$, we have
  \begin{equation}
    \label{eq:upper-bound-conf}
    \P_P \bigg( \kll{P}{\wh P_{n,\delta}}
    \geq 
    110000 \, \frac{ d + \log (d) \log (1/\delta)}{n}
    \bigg)
    \leq 4 \delta
    \, .
  \end{equation}
\end{theorem}

The estimator~\eqref{eq:def-conf-dependent} in Theorem~\ref{thm:upper-bound-conf-dependent}
can be seen as a confidence-dependent modification of the Laplace estimator.
Specifically, in the low and moderate-confidence regime where $\delta \geq e^{-d}$, the estimator $\wh P_{n,\delta}$ coincides with the Laplace estimator.
On the other hand, in the high-confidence regime where $\delta < e^{-d}$, it smooths the empirical distribution more strongly than the Laplace estimator, with a confidence-dependent level of smoothing---the higher the desired confidence level, the stronger the smoothing.

We 
now comment on the quantitative upper bound of Theorem~\ref{thm:upper-bound-conf-dependent}.
First, note that the condition $\delta > e^{-n/6}$ is not restrictive, since for $\delta = e^{-n/6}$ and $n \geq d$ the upper bound~\eqref{eq:upper-bound-conf} is of order $\log d$; but a deterministic upper bound of $\log d$ may be achieved by letting $\wh P_n = (1/d, \dots, 1/d)$ be the uniform distribution.
In fact, for $\delta \leq e^{-n/6}$ the estimator $\wh P_{n, \delta}$ also satisfies a deterministic bound $\kll{P}{\wh P_{n,\delta}} \leq \log (7d)$, as $\wh p_j \geq 1/(7d)$ for $j=1, \dots, d$.

Second, observe that the upper bound~\eqref{eq:upper-bound-conf}
provides an improvement over
the best possible guarantee for confidence-independent estimators, as characterized by Theorems~\ref{thm:upper-laplace} and~\ref{thm:lower-bound-conf-indep-tail}.
This amounts to saying that, regardless of $\delta \in (0,e^{-2})$ and $d \geq 2$, one has
\begin{equation}
  \label{eq:comparison-conf-dep-indep}
  \frac{d + \log (d) \log (1/\delta)}{n}
  \lesssim \frac{d + \log (1/\delta) \log \log (1/\delta)}{n}
  \, .
\end{equation}
To see why~\eqref{eq:comparison-conf-dep-indep} holds, consider the following two cases.
If $\log (1/\delta) \lesssim d/\log d$, then the left-hand side of~\eqref{eq:comparison-conf-dep-indep} is of order $d/n$ and the bound holds.
On the other hand, if $\log (1/\delta) \gtrsim d/\log d \gtrsim \sqrt{d}$, then
$\log d \lesssim \log \log (1/\delta)$ and the bound~\eqref{eq:comparison-conf-dep-indep} also holds.

Hence, Theorems~\ref{thm:lower-bound-conf-indep-tail} and~\ref{thm:upper-bound-conf-dependent} together imply an advantage of confidence-dependent estimators over confidence-independent ones.
We note that such an advantage has been previously observed for
a different problem, namely mean estimation under heavy-tailed noise~\cite{devroye2016subgaussian,catoni2012challenging}.
It is notable that it also manifests itself %
in
the basic problem
of estimation of discrete distributions, in the absence of %
robustness constraints.

We refer to Section~\ref{sec:proof-upper-conf-dependent} for the proof of Theorem~\ref{thm:upper-bound-conf-dependent}, which
shares a common structure with
that of Theorem~\ref{thm:upper-laplace},
although some details differ.
Again, the core of the analysis is to control the contribution to the relative entropy of classes whose frequency is underestimated, %
which is technically achieved through
sharp moment estimates on the corresponding terms.

\subsection{Lower bound for confidence-dependent estimators
}
\label{sec:lower-bound-conf}

While the upper bound of Theorem~\ref{thm:upper-bound-conf-dependent} for the estimator $\wh P_{n,\delta}$ circumvents the lower bound of Theorem~\ref{thm:lower-bound-conf-indep-tail} for confidence-independent estimators, it still exceeds the ideal asymptotic tail bound~\eqref{eq:asymptotic-mle} by a factor of $\log d$ in the deviation term.
Theorem~\ref{thm:lower-bound-minimax} below shows that this extra $\log d$ factor cannot be avoided, even for confidence-dependent estimators:

\begin{theorem}
  \label{thm:lower-bound-minimax}
  Let $n \geq d \geq 5000$ and $\delta \in (e^{-n}, e^{-1})$.
  For any estimator $\Phi = \Phi_\delta : [d]^n \to \probas_d$, there exists a distribution $P \in \probas_d$ such that, letting
  $\wh P_n = \wh P_{n, \delta} = \Phi_\delta (X_1, \dots, X_n)$ we have
  \begin{equation}
    \label{eq:lower-bound-minimax}
    \P_P \bigg( \kll{P}{\wh P_{n}} \geq \frac{d + \log (d) \log (1/\delta)}{5000 \,n} \bigg)
    \geq \delta
    \, .
  \end{equation}
\end{theorem}

Together, Theorems~\ref{thm:upper-bound-conf-dependent} and~\ref{thm:lower-bound-minimax} characterize, up to universal constant factors, the minimax high-probability risk (in other words, the ``sample complexity'') for estimation of discrete distributions in relative entropy.
An interesting feature of this lower bound is that it exceeds the asymptotic rate~\eqref{eq:asymptotic-mle} by a $\log d$ term;
this establishes a separation between %
asymptotic guarantees and uniform %
non-asymptotic guarantees.

The proof of Theorem~\ref{thm:lower-bound-minimax} relies on the following lemma, which is proved in Section~\ref{sec:proof-lemma-}.

\begin{lemma}
  \label{lem:lower-minimax-tail}
  Let $n \geq d \geq 2$ and $\delta \in (e^{-n}, e^{-1})$.
  There exists a set $\F = \F_{n, d, \delta} \subset \probas_d$ of $d$ distributions with support size at most $2$ such that the following holds.
  For any estimator $\wh P_n = \Phi (X_1, \dots, X_n)$, there exists a distribution $P \in \F$ such that
  \begin{equation}
    \label{eq:lower-minimax-tail}
    \P_P \bigg( \kll{P}{\wh P_n} \geq \frac{\log (d) \log (1/\delta)}{14 \,n} \bigg)
    \geq \delta
    \, .
  \end{equation}
\end{lemma}

The idea behind Lemma~\ref{lem:lower-minimax-tail} is quite simple, and can be summarized as follows: let $P$ be either a Dirac mass at $1$, namely $P = \delta_1 = (1, 0, \dots, 0)$, or a mixture of the form $(1 - \frac{\log (1/\delta)}{n}) \delta_1 + \frac{\log (1/\delta)}{n} \delta_j$ for some $j = 2, \dots, d$.
Then, regardless of which of these distributions $P$ is, with probability at least of order $\delta$, only the first class is observed ($X_1 = \dots = X_n = 1$).
In this case, it is impossible to tell which distribution $P$ is.
In order to avoid incurring a large error in the first case where $P = \delta_1$, one must assign a large probability to the first class, and thus a low probability to all remaining classes.
However, this entails a large error in the second case.
Specifically, the extra $\log (d)$ factor comes from the fact that in the second case where $P$ puts some mass on a class $j \in \set{2, \dots, d}$,
no information on $j$ is available if only the first class is observed.
Hence, all the remaining mass must be shared among the $d-1$ remaining classes $j=2, \dots, d$, effectively dividing by $d-1$ the per-class probability and thus inflating the relative entropy.

We note in passing that Lemma~\ref{lem:lower-minimax-tail} has broader implications to the theory of aggregation and density estimation, beyond the present context of estimation of discrete distributions.
Indeed, it is known from~\cite{yang1999information,catoni2004statistical} (building on an idea of Barron~\cite{barron1987bayes}), that for any finite model/class $\F$ of distributions, given an \iid sample of size $n$ from an unknown distribution $P \in \F$, there exists an estimator $\wh P_n$ such that $\E_P [\kll{P}{\wh P_n}] \leq \log (|\F|)/n$ for all $P \in \probas$.
This naturally raises the question of whether a corresponding ideal high-probability guarantee, of the form $\P_P (\kll{P}{\wh P_n} \geq C\, \{ \log |\F| + \log (1/\delta) \}/n) \leq \delta$ for some absolute constant $C$, can also be achieved, possibly by another estimator.
Lemma~\ref{lem:lower-minimax-tail} shows that this is not the case: since $|\F_{n,d,\delta}| = d$, the best tail bound one may hope for is of order $\log (|\F|) \log (1/\delta)/n$.

\section{Adaptation to the effective support size}
\label{sec:adaptation-support}

The results of Section~\ref{sec:confidence-minimax} settle the question of the
minimax-optimal high-probability guarantees.
In addition, the results of Section~\ref{sec:laplace} show that the classical Laplace estimator is rather close to being optimal, in the same sense.

While these results attest to the soundness of the Laplace rule---and of its confidence-dependent modification---, one should keep in mind that in several applications such as natural language processing, its use has been supplanted by that of more advanced methods, such as Kneser-Ney smoothing~\cite{kneser1995improved,chen1999empirical}.

This apparent contradiction between theory and practice comes from the fact that we have so far focused on minimax guarantees that hold uniformly over all distributions.
By itself, this uniformity is a strength, as it requires no restrictive assumptions on the true distribution.
In addition, it is not obviously detrimental, given that the limiting risk~\eqref{eq:asymptotic-mle} of the MLE does not depend on the distribution $P \in \probas_d$, as long as the latter assigns positive probability to all classes.

Nevertheless, this last restriction points towards a possible improvement: if only $s < d$ classes have positive probability, then by~\eqref{eq:asymptotic-mle}, the limiting high-probability risk of the MLE scales as $\{ s + \log (1/\delta) \}/n$, which can be much smaller than $d/n$ if $s \ll d$.
This suggests that, more generally, the complexity of distribution estimation should be governed by some notion of support size, such as the number of positive or (for a fixed sample size) large enough probabilities.

Indeed, while the minimax rate of $d/n$ cannot be improved for worst-case distributions that are approximately uniform over $\set{1, \dots, d}$,
distributions that arise in practice often exhibit a non-uniform structure,
with a small number of frequent classes and a large number of less frequent classes.
Under such a configuration, it may be possible to estimate the distribution $P$ even when the sample size $n$ is smaller than the total number $d$ of classes.

In this section, we consider high-probability upper and lower bounds for estimation of discrete distributions, which depend on suitable notions of ``support size'' of the distribution.
Section~\ref{sec:minimax-lower-sparse} contains minimax lower bounds for estimation of sparse distributions, while Section~\ref{sec:adaptive-estimators} is devoted to high-probability upper bounds for
suitable adaptive estimators.

\subsection{Minimax lower bounds for sparse distributions}
\label{sec:minimax-lower-sparse}

We start by establishing minimax lower bounds on the best possible estimation guarantee for ``sparse'' distributions.
Since we are interested in lower bounds, we consider a small class of sparse distributions $P$, namely distributions with support size at most $s \leq d$.
Naturally, such lower bounds transfer to larger classes, that is, to less stringent notions of sparsity.

For any $P \in \probas_d$, we let $\supp (P) = \{ 1 \leq j \leq d : p_j > 0 \}$ denote the support of $P$.
In addition, for $1 \leq s \leq d$ we define the class $\probas_{s, d} = \{ P \in \probas_d : \abs{\supp (P)} \leq s \}$ of probability distributions on $\set{1, \dots, d}$ supported on at most $s$ elements.
We call such distributions \emph{$s$-sparse}.

Our main lower bound for the class of $s$-sparse distributions is the following:

\begin{proposition}
  \label{prop:lower-bound-minimax-sparse}
  Let $n, d \geq 2$.
  For any $1 \leq s \leq \min (n, d/55%
  )$ and any estimator $\wh P_n = \Phi (X_1, \dots, X_n)$, 
  there exists a distribution $P \in \probas_{s,d}$ such that
  \begin{equation}
    \label{eq:sparse-lower-high-probability}
    \P_P \bigg( \kll{P}{\wh P_n} \geq \frac{s \log (e d/s)}{300\, n} \bigg)
    \geq 1 - 3 \exp \Big( - \frac{s}{35} \Big) %
    \, .
  \end{equation}
\end{proposition}

We refer to Section~\ref{sec:proof-lower-minimax-sparse} for the proof of Proposition~\ref{prop:lower-bound-minimax-sparse}.
In particular, letting $s = \floor{d/55}$ in Proposition~\ref{prop:lower-bound-minimax-sparse} shows the following: for any $n \geq d \geq 110$ and any estimator $\wh P_n = \Phi (X_1, \dots, X_n)$, there exists a distribution $P \in \probas_d$ such that
\begin{equation}
  \label{eq:dense-lower-high-prob}
  \P_P \bigg( \kll{P}{\wh P_n} \geq \frac{d}{4600\, n} \bigg)
  \geq 1 - 3 \exp \Big( - \frac{d}{3000} \Big) %
  \, .
\end{equation}

Two aspects of the lower bound of Proposition~\ref{prop:lower-bound-minimax-sparse}
merit
further discussion.

First, the minimax lower bound is of order $s \log (e d/s) / n$, which exceeds by a $\log (e d/s)$ factor the rate of estimation of a distribution $P$ with \emph{known} support of size $s$.
This extra factor, which is standard in the context of estimation under sparsity, can be seen as the price of estimation of the support of $P$: indeed, one has $s \log (e d/s) \asymp \log \binom{d}{s}$, where $\binom{d}{s}$ is the number of possible supports of size $s$.
We note that
the rate of $s \log (e d/s)/n$ coincides with the minimax rate of estimation of a sparse vector
under Gaussian noise
(e.g.,~\cite[p.~156]{wainwright2019high}).

However, despite the similarity in rates, there are qualitative differences between the Gaussian setting and the multinomial setting that we consider.
Indeed, it follows from the results of Agrawal~\cite{agrawal2022finite} (see Lemma~\ref{lem:hellinger-empirical} below) that when $P \in \probas_{s,d}$, the empirical distribution achieves an upper bound in squared Hellinger distance of order $s/n$ with probability at least $1 - e^{-s}$.
This rate no longer features the extra $\log (ed/s)$ factor; by contrast, the extra $\log (e d/s)$ factor does appear in the Gaussian model, even when the error is measured in squared Hellinger distance.
Roughly speaking, this difference stems from the fact that in the multinomial model, unlike in the Gaussian model, the amount of ``noise'' in a coordinate $j=1, \dots, d$ decays when the magnitude of the corresponding coefficient parameter $p_j$ goes to $0$.

In particular,
the optimal rate of estimation of sparse discrete distributions in Kullback-Leibler divergence exceeds the optimal rate under squared Hellinger divergence by %
a $\log (e d/s)$ factor.
This gap comes from the fact that the empirical distribution %
only identifies a subset of the support of the true distribution,
and the price of
missing part of the support of $P$
is higher in Kullback-Leibler divergence than in squared Hellinger distance.

A second important feature of the lower bound of Proposition~\ref{prop:lower-bound-minimax-sparse} (and of its consequence~\eqref{eq:dense-lower-high-prob} in the non-sparse case) is that it holds \emph{with high probability}.
Such a statement is stronger than minimax lower bounds that are usually found in the literature, which hold either in expectation or with constant probability.
For estimation of non-sparse discrete distributions, lower bounds in expectation are proven in~\cite{han2015minimax,kamath2015learning},
while~\cite{chhor2024generalized} obtains a lower bound with constant probability.
To appreciate the difference, consider the $d/n$ lower bound in the non-sparse case.
Then, a $d/n$ lower bound in expectation is in principle compatible with the following behavior: a risk of $1$ with probability $d/n$, and a risk of $0 \ll d/n$ with high probability $1-d/n$.
In contrast, a stronger lower bound with constant probability rules out such a behavior: it asserts that a lower bound of order $d/n$ must hold with probability bounded away from $0$, say $10\%$ or $80\%$.
However, even this lower bound does not rule out the possibility that the estimator achieves an error significantly smaller than $d/n$ (say, of $0$) with nontrivial probability---respectively, with probability $90\%$ or $20\%$.
A high-probability lower bound such as~\eqref{eq:dense-lower-high-prob} excludes such a behavior: it asserts that a lower bound of order $d/n$ holds with
overwhelming probability,
effectively ruling out the possibility that the estimator ``gets lucky''.
In addition, 
the probability with which the lower bound~\eqref{eq:dense-lower-high-prob} holds,
which behaves as $1-e^{-d/c}$,
is best possible, as shown by the
convergence of the properly scaled risk of the MLE to a $\chi_{d-1}^2$ distribution as $n \to \infty$.

We note that the classical method for proving lower bounds in statistical estimation (put forward by Ibragimov and Has'minskii~\cite{ibragimov1981estimation}, see~\cite[Chapter~2]{tsybakov2009nonparametric} and~\cite[Chapter~15]{wainwright2019high}), which consists in a reduction from estimation to testing, followed by a lower bound for testing though Fano's inequality, actually provides a high-probability lower bound.
However, the probability estimate that the commonly used version of this inequality (see \eg \cite[Proposition~5.12 p.~502]{wainwright2019high}) leads to is weaker than~\eqref{eq:dense-lower-high-prob}.
Specifically, it leads to a lower bound that holds with probability $1- c/d$ for some constant $c$, in contrast with the lower bound~\eqref{eq:dense-lower-high-prob} with probability $1-e^{-d/c}$.
As it turns out, in the non-sparse case, one could also obtain a probability of $1-e^{-d/c}$ by instead using the sharp version of Fano's inequality (\eg,~\cite[Lemma~2.10]{tsybakov2009nonparametric}) and further simplifying the resulting bound.
However, the reduction from estimation to testing is not straightforward in the sparse case, essentially because the Kullback-Leibler divergence does not behave like a distance over sparse distributions with distinct supports.

The proof of Proposition~\ref{prop:lower-bound-minimax-sparse}, %
which can be found in Section~\ref{sec:proof-lower-minimax-sparse},
is not based 
on a reduction to testing, but instead on the probabilistic method---in other words, on a Bayesian lower bound.
Specifically, we consider suitable distributions with randomly chosen support of size $s$, and show that on average over the random draw of such a distribution, the probability that any estimator achieves a risk at least of order $s \log (ed/s)/n$ is overwhelmingly close to $1$.

We believe that minimax lower bounds that hold with probability exponentially close to $1$ (such as Proposition~\ref{prop:lower-bound-minimax-sparse}) have broad relevance in statistical estimation beyond the present setting of discrete distributions, as they provide very precise information.
Although such lower bounds appear to be scarce in the literature, a recent example is~\cite[Theorem~2.1]{rajaraman2024statistical} which is established via the $\chi^2$ mutual information using tools from~\cite{chen2016bayes}.

In addition to the high-probability lower bound of Proposition~\ref{prop:lower-bound-minimax-sparse}, we now provide a \emph{low-probability} lower bound.
Low-probability lower bounds (such as Theorems~\ref{thm:lower-bound-conf-indep-tail} and~\ref{thm:lower-bound-minimax} in previous sections) are more common in the literature;
their interest stems from the fact that they are precisely the converse of high-probability upper bounds, and can thus attest to the optimality of such upper bounds.
Specifically, combining Proposition~\ref{prop:lower-bound-minimax-sparse} with Lemma~\ref{lem:lower-minimax-tail} above---which applies to our present setting as it involves $2$-sparse distributions---leads to the following:

\begin{corollary}
  \label{cor:minimax-sparse-lower}
  Let $n, d \geq 110$, $s \in \{ 2, \dots, \min(n, d/55) \}$ and $\delta \in (e^{-n}, e^{-2})$.
  For any estimator $\Phi = \Phi_{s,\delta} : [d]^n \to \probas_d$, there exists a distribution $P \in \probas_{s,d}$ such that, letting $\wh P_n = \Phi (X_1, \dots, X_n)$, we have
  \begin{equation}
    \label{eq:minimax-sparse-lower}
    \P_P \bigg( \kll{P}{\wh P_n} \geq \frac{s \log (e d/s) + \log (d) \log (1/\delta)}{320\, n} \bigg)
    \geq \delta
    \, .
  \end{equation}
\end{corollary}

We refer to Section~\ref{sec:proof-lower-minimax-sparse-tail} for the proof of this result.
Corollary~\ref{cor:minimax-sparse-lower} recovers the minimax lower bound of Theorem~\ref{thm:lower-bound-minimax} up to constants, by setting $s = \floor{d/55}$.

\subsection{Upper bounds for adaptive estimators}
\label{sec:adaptive-estimators}

Having obtained in Corollary~\ref{cor:minimax-sparse-lower} a lower bound on the best possible high-probability guarantee, we now turn to %
upper bounds.
Of particular interest are estimators that achieve such guarantees without prior knowledge of the support size $s$ of $P$; we call such estimators \emph{adaptive}.

\subsubsection{Effective sparsity parameters}

Before stating such guarantees, it is worth discussing what an ideal upper bound might look like.
Of course, a natural objective would be to obtain minimax high-probability upper bounds under sparsity constraints that match the lower bound of Corollary~\ref{cor:minimax-sparse-lower}.
However, it should be noted that the notion of sparsity (support size) we considered in the previous section is rather restrictive: it excludes distributions with full support, but for which only a small number of classes have significant probability.
Intuitively, classes with positive but very small probability should have a limited effect on the estimation error, at least for moderate sample sizes.

These considerations call for identifying notions of ``effective support size'' or ``effective sparsity'' that capture the hardness of estimation more accurately than the number of
nonzero entries of the distribution.
One should expect such notions of effective support size to depend on the sample size $n$, given that the large-sample asymptotic behavior~\eqref{eq:asymptotic-mle} of the MLE is governed by the number of nonzero entries. %

Perhaps the most natural notion of ``effective support size of $P$ under sample size $n$'' is the typical number of distinct classes that would appear on an \iid sample of size $n$ from $P$.
As we will explain briefly, this leads to the following definition:

\begin{definition}
  \label{def:strong-effective-support}
  For any distribution $P = (p_1, \dots, p_d) \in \probas_d$ and
  $n \geq 1$,
  the \emph{effective support size} of $P$ at sample size $n$ is the quantity $s_n (P) \in [1,\min(n, d)]$ defined by
  \begin{equation}
    \label{eq:def-effective-support}
    s_n (P)
    = \sum_{j=1}^d \min (np_j, 1)
    = \big| \big\{ 1 \leq j \leq d : p_j \geq 1/n \big\} \big| + \sum_{j \pp p_j < 1/n} n p_j
    \, .
  \end{equation}
\end{definition}

It is clear from the first expression that $s_n (P)$ increases with $n$ while $s_n (P)/n$ decreases with $n$, that $s_1 (P) = 1$, that $s_n (P) \leq s$ if $P \in \probas_{s,d}$, and that $s_n (P)$ converges as $n \to \infty$ to $s_{\infty} (P) = | \set{1 \leq j \leq d : p_j > 0} |$.
Now, denote by
$D_n$ the number of distinct classes in the sample $X_1, \dots, X_n$, namely
\begin{equation}
  \label{eq:def-distinct-classes}
  D_n
  = \sum_{j = 1}^d \indic{N_j \geq 1}
  = \big| \big\{ X_i : 1 \leq i \leq n \big\} \big|
  \, .
\end{equation}
The following fact relates $\E_P [D_n]$ to $s_n (P)$.

\begin{fact}
  \label{fac:expected-support}
  For any $P \in \probas_d$ and $n \geq 1$, one has
  $(1-e^{-1}) s_n (P) \leq \E_P [D_n] \leq s_n (P)$.
\end{fact}

\begin{proof}
  The upper bound comes from the fact that $\E_P [D_n] = \sum_{j=1}^d \P (N_j = 1)$ and that
  $\P (N_j = 1) = \P (\bigcup_{i=1}^n \{X_i = j\}) \leq \min (n p_j, 1)$ by a union bound.
  For the lower bound, we start with the expression $\E_P [D_n] = \sum_{j=1}^d \{ 1 - (1-p_j)^n \}$.
  Now, if $p_j \geq 1/n$, then $(1-p_j)^n \leq (1-1/n)^n \leq e^{-1}$, hence $1 - (1-p_j)^n \geq 1-e^{-1}$.
  On the other hand, if $p_j \leq 1/n$, %
  then $(1-p_j)^n \leq e^{-np_j} \leq 1 - (1-e^{-1}) np_j$ by convexity of $\exp$, so that $1 - (1-p_j)^n \geq (1-e^{-1}) np_j$.
\end{proof}

One reason why the quantity $s_n (P)$ is a natural ``effective sparsity'' parameter is that it controls the error of the empirical distribution in squared Hellinger distance, as shown by Lemma~\ref{lem:hellinger-empirical} below.
To see how it arises here, recall that for an estimator $\wh P_n = (\wh p_1, \dots, \wh P_d)$, one has
\begin{equation*}
  \kll{P}{\wh P_n}
  = \sum_{j=1}^d D (p_j, \wh p_j)
  \, .
\end{equation*}
Now, consider the situation
in which for each class $j = 1, \dots, d$, we are given two options:
\begin{itemize}
\item estimate $p_j$ based on the binomial $N_j$, for instance using the empirical frequency %
  or the binary Laplace estimate $(N_j+1)/(n+2)$.
  By~\eqref{eq:bound-expectation-laplace}, the latter gives $\E [D (p_j, \wh p_j)] \lesssim 1/n$;
\item alternatively, %
  resort to an ``oracle'' that returns an approximation $\wt p_j$ of $p_j$
  guaranteed to be of the same order of magnitude: $p_j/2 \leq \wt p_j \leq 2p_j$.
  This ensures that $D (p_j, \wt p_j) \lesssim p_j$.
\end{itemize}
Using the best of the two options for each coordinate %
leads to an error in relative entropy of order at most $\sum_{j=1}^d \min (p_j, 1/n) = s_n (P)/n$.
However, one cannot expect a bound of $s_n (P)/n$ to be achievable without the oracle, as this would violate the lower bound of Corollary~\ref{cor:minimax-sparse-lower}.
This is because a class $j$ with small probability $p_j \ll 1/n$
typically
does not appear in the sample; in this case, one has no information about the order of magnitude of $p_j$ except that $p_j \lesssim 1/n$.
This obstruction corresponds to the difficulty of estimating the support of $P$ discussed in Section~\ref{sec:minimax-lower-sparse}, which leads to an extra $\log (e d/s)$ factor.

In light of this discussion, a natural objective would be to aim for a high-probability upper bound matching the lower bound~\eqref{eq:minimax-sparse-lower}, but with the support size $s$ possibly replaced by the effective support size $s_n (P)$.

An in-expectation version of such a result was established by Falahatgar, Ohannessian, Orlitsky and Pichapati~\cite[Theorem~1]{falahatgar2017power}, who showed that the ``absolute discounting'' estimator $\wh P_n$ satisfies the following guarantee: for any $P \in \probas_d$, denoting $s_n = s_n (P)$ one has
\begin{equation}
  \label{eq:falahatgar-expected}
  \E_P [ \kll{P}{\wh P_n} ]
  \leq c \cdot \frac{s_n \log (e d/s_n)}{n}
\end{equation}
for some constant $c$ that only depends on the tuning parameter of the estimator.

Before establishing a high-probability extension of this result,
we first show that an improvement of the bound~\eqref{eq:falahatgar-expected} is possible even in expectation.
To see why, consider a distribution $P \in \probas_d$ with support size $%
s \ll d$, and consider the asymptotic regime where $n \to \infty$.
In this case, it follows from~\eqref{eq:asymptotic-mle} that as $n \to \infty$, the MLE $\ol P_n$ achieves a risk of order $s/n$ with high probability.
On the other hand, again as $n \to \infty$, one has $s_n \to s$, hence the bound~\eqref{eq:falahatgar-expected} scales as $s \log (e d/s)/n$, which exceeds the risk of the MLE by a $\log (ed/s)$ factor.

To understand this discrepancy, recall that the $\log (ed/s)$ factor comes from the fact that the support of $P$ is unknown.
Meanwhile, $s_n (P)$ measures the typical number of distinct classes that appear in the sample.
But since classes that appear in the sample are known, they do not contribute to the uncertainty about the support of $P$.
In contrast, the only classes that contribute to the uncertainty on the support are those that are missing from the sample\footnote{To be accurate, the situation is slightly more subtle: classes that do appear, but with an empirical frequency significantly smaller than their theoretical one, also contribute to the inflation of the relative entropy.}.

Hence, one should expect that the effective support size $s_n (P)$ does not suffice to describe the achievable error rate in relative entropy.
Instead, the logarithmic term owing to the lack of knowledge of the support should be tied to a different sparsity parameter, which accounts for the contribution of classes that are missing from the sample.
We now define such a parameter:

\begin{definition}
  \label{def:effective-missing-support}
  For any distribution $P = (p_1, \dots, p_d) \in \probas_d$ and real number $n \geq 1$, the \emph{effective missing support size} of $P$ at sample size $n$ is the quantity $s_n^\circ (P) \in (0, \min (n, d)]$ defined by
  \begin{equation}
    \label{eq:def-effective-missing-support}
    s_n^\circ (P)
    = \sum_{j=1}^d \min \big( e^{1- n p_j}, n p_j \big)
    = \sum_{j \pp p_j \geq 1/n} e^{1 - n p_j} + \sum_{j \pp p_j < 1/n} n p_j
    \, .
  \end{equation}
\end{definition}

The
reason why $s_n^\circ (P)$ accounts for the contribution of missing classes is that $s_n^\circ (P)/n$ is closely related to the expected ``missing mass'' (total probability of classes that do not appear in the sample), as shown in Lemma~\ref{lem:expected-missing-mass} below.
We defer to Section~\ref{sec:missing-mass} for more discussion on the behavior of the missing mass.
In addition, again by Lemma~\ref{lem:expected-missing-mass}, $s_n^\circ (P)$ is also closely related to the expected number of new classes that would appear on an independent sample $X_{n+1}, \dots, X_{2n}$  of size $n$ (but not in $X_1, \dots, X_n$), namely $\E_P [D_{2n} - D_{n}]$; this justifies the terminology ``effective missing support size'' of Definition~\ref{def:effective-missing-support}.
This suggests that $s_n^\circ (P)$ (roughly the number of classes that first appear after $\asymp n$ observations) can be seen as a ``local'' counterpart to the ``global'' effective sparsity parameter $s_n (P)$ (number of classes that first appear after $\lesssim n$ observations).

It is clear from the second expression in~\eqref{eq:def-effective-missing-support} that $s_n^\circ (P) \leq s_n (P)$.
In addition, from the first expression, the quantity $s_n^\circ (P)/n$ decreases with $n$.
On the other hand, contrary to $s_n (P)$, the quantity $s_n^\circ (P)$ does not increase with $n$; in fact, one has $s_n^\circ (P) \to 0$ as $n \to \infty$ for any $P \in \probas_d$.

To appreciate the difference between the two effective support size parameters,
one may use the following
rough %
approximations: 
$s_n (P)$ %
usually 
amounts to the number of classes $j$ such that $p_j \gtrsim 1/n$, while $s_n^\circ (P)$ often roughly corresponds to the number of classes $j$ with $p_j \asymp 1/n$.

\subsubsection{%
  Upper bound in expectation for sparse distributions}

As discussed above, the contribution to the estimation error of classes that appear in the sample (with an empirical frequency of the same order as their true frequency) should ideally scale as $s_n (P)/n$, without additional logarithmic factors.
On the other hand, the contribution to the estimation error of classes that do not appear in the sample (or appear with an empirical frequency significantly smaller than their true frequency) is governed by the parameter $s_n^\circ  = s_n^\circ (P)$.
To anticipate their contribution, consider the situation where $s_n^\circ$ indeed scales as the number of classes $j$ such that $p_j \asymp 1/n$.
Such classes have a constant probability of not appearing in the sample, in which case their identity is unknown.
Thus, there are typically about $s_n^\circ$ missing classes that have a probability of order $1/n$.
However, assuming that the total number of missing classes is of order $d$ (which occurs for instance in the high-dimensional regime $d \geq 2n$), the total mass of missing classes, of order $s_n^\circ/n$, must be shared among roughly $d$ classes.
This means each of the roughly $s_n^\circ$ classes $j$ with $p_j \asymp 1/n$ is assigned a probability $\wh p_j \asymp (s_n^\circ (P)/n)/d \lesssim 1/n$, leading to a total contribution to the estimation error of order $s_n^\circ \times \frac{1}{n} \log \big( \frac{1/n}{s_n^\circ / (d n)} \big) \asymp \frac{s_n^\circ}{n} \log (d/s_n^{\circ})$.

The next result (proved in Section~\ref{sec:proof-expected-sparse-loo}) shows that
a matching upper bound
can indeed be achieved by a suitable adaptive
estimator:

\begin{proposition}
  \label{prop:expected-sparse-loo}
  Let $\adp = (\wh p_1, \dots, \wh p_d)$ denote the estimator defined by, for $j=1, \dots, d$,
  \begin{equation}
    \label{eq:def-adaptive-estimator}
    \wh p_j
    = \frac{N_j + \wh \lambda}{n + \wh \lambda d}
    \qquad \text{with} \qquad
    \wh \lambda = \frac{D_n}{d} \, ,
  \end{equation}
  where $D_n = \sum_{j=1}^d \indic{N_j \geq 1}$  is the number of distinct classes among $X_1, \dots, X_n$.
  Then, for any $n, d \geq 2$ and distribution $P \in \probas_d$, letting $s_n = s_n (P)$ and $s_{n/2}^\circ = s_{n/2}^\circ (P)$ one has
  \begin{equation}
    \label{eq:expected-sparse-loo}
    \E_P \big[ \kll{P}{\adp} \big]
    \leq \frac{2.4 s_n + 2 s_{n/2}^\circ \log ( {e d}/{s_{n/2}^\circ} )}{n+1}
    \, .
  \end{equation}
\end{proposition}

The estimator $\adp$ defined by~\eqref{eq:def-adaptive-estimator} can be viewed as an adaptive modification of the Laplace estimator~\eqref{eq:def-laplace}, where the regularization parameter $\wh \lambda$ is chosen in a data-dependent manner, so as to adapt to the ``shape'' of the distribution $P$.
In the case where $n  \gtrsim d$ and the distribution is ``dense'', namely puts significant probability $p_j \gtrsim 1/d$ to all classes, then typically $D_n$ is of order $d$ and thus $\wh \lambda \asymp 1$, hence the estimator behaves similarly to the Laplace estimator.
On the other hand, when the distribution is highly sparse, then typically $D_n \ll d$ and thus the regularization parameter $\wh \lambda \ll 1$ is much smaller than for the Laplace estimator---it may be as low as $1/d$, in the extreme case where only one class appears in the sample.
The specific choice of tuning parameter $\wh \lambda = D_n / d$ is motivated by the risk decomposition in Lemma~\ref{lem:decomp-risk} below.

To the best of our knowledge, the estimator $\adp$ defined by~\eqref{eq:def-adaptive-estimator} is new.
However, it is related to existing estimators from the literature.
First, it bears some relation with absolute discounting~\cite{ney1994structuring,kneser1995improved}, a core component in the Kneser-Ney smoothing method which has long been favored in natural language processing~\cite{kneser1995improved,chen1999empirical} and whose performance has been analyzed in~\cite{ohannessian2012rare,falahatgar2017power}; see also~\cite{teh2006hierarchical} for a justification of this method in the case of polynomially decaying class frequencies.
Specifically, absolute discounting consists in removing a constant $\eta \in (0, 1)$ to the count $N_j \geq 1$ of all classes $j=1, \dots, d$ that appear in the sample, and sharing the freed mass among
missing classes.
When $D_n \leq d/2$, absolute discounting and the estimator $\adp$ both assign a probability $\wh p_j \asymp D_n / (d n)$ to missing classes $j$; on the other hand, the corrections of these two estimators for classes $j$ with $N_j \gg 1$ differ.
Second, in an online setting where observations accrue one by one, Hutter~\cite{hutter2013sparse} proposed a sequential prediction method using a regularization $\wh \lambda_t \asymp D_t / \{ d \log [e t/D_t] \}$ after $t$ observations, and shows that it satisfies regret guarantees in sparse situations.
This method is related to the estimator $\adp$, the main difference being a logarithmic factor in the regularization parameter.
We also note that guarantees in the sequential setting necessarily feature an additional $\log n$ factor compared to the fixed-sample size setting, due to the contribution of small sample sizes.

Since $s_{n/2}^\circ \leq s_{n/2} \leq s_n$, Proposition~\ref{prop:expected-sparse-loo} matches the upper bound~\eqref{eq:falahatgar-expected} from~\cite{falahatgar2017power} for the absolute discounting estimator.
In addition, it improves this bound when $s_{n/2}^\circ \ll s_n$.

\begin{remark}[Comparison between $s_n$ and $s_n^\circ$]
  \label{rem:example-distributions}
  In order to quantitatively appreciate similarities and differences between the two effective sample sizes, it helps to consider the following stylized situations.
  Let $P = (p_1, \dots, p_d)$, and assume (up to relabeling classes) that $p_1 \geq \dots \geq p_d$.
  \begin{enumerate}
  \item[(a)] \emph{Polynomial decay}: first, assume that $p_j \asymp j^{-\alpha}$ for some constant $\alpha > 1$, and that $d \gg
    n^{1/\alpha}$.
    In this case, one has $s_n \asymp s_n^\circ \asymp n^{1/\alpha}$ up to constants that may depend on $\alpha$.
    Hence, the two effective sample size parameters are equivalent, and so are the bounds~\eqref{eq:falahatgar-expected} and~\eqref{eq:expected-sparse-loo}.
  \item[(b)] \emph{Geometric decay}: assume now that $c_1 e^{-C_1 j} \leq p_j \leq C_2 e^{-c_2 j}$ for every $j=1, \dots, d$ for some constants $c_1, c_2, C_1, C_2 > 0$.
    If $d \geq C \log n$ for some large enough constant $C$ (depending on the previous constants $c_k, C_k$), then $s_n \asymp \log n$ while $s_n^\circ \asymp 1$.
    In particular, if $d$ scales polynomially in $n$, then the bound~\eqref{eq:falahatgar-expected} scales as $\log^2 (n) /n$, while the bound~\eqref{eq:expected-sparse-loo} scales as $\log (n)/n$.
  \item[(c)] \emph{Sparse distributions}: finally, assume that $P$ is supported on $s \leq d$ classes, and that the frequencies of these classes are lower-bounded; namely, $p_j \geq c/s$ for $j = 1, \dots, s$ for some $c \in (0,1)$, while $p_j = 0$ for $s < j \leq d$.
    Then, for $n \geq 2 s \log (e s)/c$, one has $s_n = s$ while $s_n^\circ \leq s_{n/2}^\circ \leq e s e^{- c n/(2s)} \leq 1$, hence $s_n^\circ \ll s_n$ if $s \gg 1$.
    Thus if $d^{\eps} \leq s \leq d^{1-\eps}$ for some $\eps \in (0,1)$, ignoring the dependence on $c, \eps$, the bound~\eqref{eq:falahatgar-expected} scales as $s \log (d)/n$, while the bound~\eqref{eq:expected-sparse-loo} scales as $s/n$, removing a $\log d$ factor.
  \end{enumerate}
\end{remark}

We note that the bound of Proposition~\ref{prop:expected-sparse-loo} holds in expectation, while the focus of the present work is on high-probability guarantees.
High-probability extensions %
will be proved in what follows, but we have nonetheless included Proposition~\ref{prop:expected-sparse-loo} for two reasons.
First, this bound features small numerical constants.
Second and more importantly, the proof of Proposition~\ref{prop:expected-sparse-loo} is self-contained and significantly simpler than that of the high-probability bounds.
Specifically, this proof relies on a combination of exchangeability and leave-one-out arguments.
Such arguments already appear in the proof of the bound~\eqref{eq:bound-expectation-laplace} for the Laplace estimator (see~\cite{catoni1997mixture,mourtada2022logistic}), although the proof of Proposition~\ref{prop:expected-sparse-loo} requires somewhat
more careful counting and conditioning.

\subsubsection{High-probability upper bounds for sparse distributions}

We now turn to the strongest positive guarantees in this work, namely high-probability upper bounds for the estimator $\pad$ and its confidence-dependent version $\padelta$, defined below.
These results can be seen as sparsity-adaptive extensions of Theorems~\ref{thm:upper-laplace} and~\ref{thm:upper-bound-conf-dependent}, respectively.

\begin{theorem}
  \label{thm:upper-bound-sparse}
  Let $n \geq 12, d \geq 3$, $\delta \in (e^{-n/6}, e^{-2})$ and $P \in \probas_d$, and denote $s_n = s_n (P)$ and $s_n^\circ = s_n^\circ (P)$.
  The add-$\wh \lambda$ estimator $\pad = (\wh p_1, \dots, \wh p_d)$
  given by
  \begin{equation}
    \label{eq:def-add-hat-lambda}
    \wh p_j = \frac{N_j + \wh \lambda}{n + \wh \lambda d}
  \end{equation}
  with $\wh \lambda = D_n/d$, where $\wh D_n = \sum_{j=1}^d \indic{N_j \geq 1}$ is the number of distinct classes that appear in the sample $X_1, \dots, X_n$, satisfies the following guarantee: with probability at least $1-14 \delta$ under $P$, 
  \begin{equation}
    \label{eq:upper-bound-sparse}
    \kll{P}{\pad} \leq
    121000 \, \frac{s_n + s_{n/112}^\circ \log (e d/s_{n}) + \max \{ \log d, \log \log (1/\delta) \} \log (1/\delta)}{n}
    \, .
  \end{equation}
  In addition, %
  letting $\wh \lambda_\delta = \max \{ D_n, \log (1/\delta) \}/d$, the add-$\wh \lambda_\delta$ estimator $\padelta$ satisfies the following guarantee:
  with probability at least $1-14 \delta$ under $P$, 
  \begin{equation}
    \label{eq:upper-bound-sparse-confident}
    \kll{P}{\padelta} \leq
    121000 \, \frac{s_n + s_{n/112}^\circ \log ( e d / s_n ) + \log (d) \log (1/\delta)}{n}
    \, .
  \end{equation}
\end{theorem}

The proof of Theorem~\ref{thm:upper-bound-sparse} is provided in Section~\ref{sec:proof-upper-sparse-whp}.
This proof relies on a combination of lemmata already used in the analysis of the Laplace estimator, together with additional results on the ``underestimated mass'' described in Section~\ref{sec:missing-mass} below.
We now comment on Theorem~\ref{thm:upper-bound-sparse}.

First, since $s_{n/112}^\circ \leq s_{n/112} \leq s_n \leq d$, using that $s \mapsto s \log (e d/s)$ is increasing on $[0, d]$ and recalling inequality~\eqref{eq:comparison-conf-dep-indep},
Theorem~\ref{thm:upper-bound-sparse} shows that the confidence-independent estimator $\adp$ achieves the best possible uniform high-probability guarantee over $\probas_d$ among confidence-independent estimators, while the confidence-dependent estimator $\adpdelta$ achieves the minimax high-probability guarantee over $\probas_d$.
These results match the lower bounds of previous sections and the guarantees for the Laplace estimator and its confident-dependent modification.

Second, for any $s \in \{2, \dots, d\}$ and $P \in \probas_{s,d}$ one has $s_{n/112}^\circ \leq s_n \leq s$; hence, Theorem~\ref{thm:upper-bound-sparse} implies that the estimator $\adpdelta$ achieves an upper bound of order $\{ s \log (e d/s) + \log (d) \log(1/\delta) \}/n$ over $\probas_{s,d}$.
This matches the minimax lower bound over $\probas_{s,d}$ of Corollary~\ref{cor:minimax-sparse-lower}, and thus completes the characterization of the minimax high-probability rate over this class.
In addition, the estimator $\adpdelta$ achieves this rate simultaneously for all sparsity levels $s \in \{ 2, \dots, d \}$.

Finally, the guarantees of Theorem~\ref{thm:upper-bound-sparse} mainly dependent on the intrinsic and distribution-dependent parameters $s = s_n (P)$ and $s^\circ = s_{n/112}^\circ (P)$, with only a (necessary) logarithmic dependence on the total number $d$ of classes.
They may therefore be viewed as essentially ``non-parametric''.
Note that since $s^\circ \leq s$, the term $s^\circ \log (e d/s)$ in~\eqref{eq:upper-bound-sparse} and~\eqref{eq:upper-bound-sparse-confident} may be bounded by $s^\circ \log (e d/s^\circ)$; this matches the corresponding term in the in-expectation bound of Proposition~\ref{prop:expected-sparse-loo}, up to constant factors in the error bound and sample size.
In fact, the complexity term in Theorem~\ref{thm:upper-bound-sparse} may appear to be of a smaller order of magnitude than that of Proposition~\ref{prop:expected-sparse-loo} when $s^\circ \ll s$.
This is not the case, since one also has
$s + s^\circ \log (e d/s^\circ) \lesssim s + s^\circ \log (e d/s)$
(indeed, either $s^\circ \log (e d/s^\circ) \leq s$ and the bound holds, or otherwise $\frac{e d/s^\circ}{\log (e d/s^\circ)} < ed/s$ and thus $\log (e d/s^\circ) \lesssim \log (e d/s)$ hence the bound also holds).

\paragraph{Consequences for the sample complexity.}

While we have stated our results in terms of error rates, one may also formulate them in terms of the \emph{sample complexity}, which is the sample size $n$ required to achieve an error in relative entropy of $\eps$ with probability at least $1-\delta$.

For instance, the results of Section~\ref{sec:confidence-minimax} (Theorems~\ref{thm:upper-bound-conf-dependent} and~\ref{thm:lower-bound-minimax}) imply that for $d \geq 5000$, $0 < \delta < e^{-2}$ and $0 < \eps < 1 $, a sample size of
\begin{equation*}
  n \gtrsim
  \frac{d + \log (d) \log (1/\delta)}{\eps}
\end{equation*}
is both necessary and sufficient for the existence of an estimator $\wh P_n = \Phi (X_1, \dots, X_n)$ such that $\P_P (\kll{P}{\wh P_n} \leq \eps) \geq 1-\delta$ for every $P \in \probas_d$.

It is also instructive to formulate the distribution-dependent upper bound~\eqref{eq:upper-bound-sparse-confident} for the adaptive and confidence-dependent estimator $\adpdelta$ in this way.
For any dimension $d \geq 2$, accuracy $\eps \in (0, 1)$, failure probability $\delta \in (0, e^{-2})$ and distribution $P \in \probas_d$, define the following ``critical sample sizes'':
\begin{align}
  \label{eq:critical-sample-observed}
  \nobs (P, \eps)
  &= \inf \bigg\{ n \geq 1 : \frac{s_n (P)}{n} \leq \eps \bigg\}
    \, ,
  \\
  \nmiss (P, d, \eps)
  &= \inf \bigg\{ n \geq 1
    : \frac{s_n^\circ (P)}{n}
    \leq \frac{\eps}{\log ( e d / \eps n)}
    \bigg\}
    \, ,
    \label{eq:critical-sample-missing} \\
  \ndev (d, \eps, \delta)
  &= \frac{\log (d) \log (1/\delta)}{\eps}
    \, .
    \label{eq:critical-sample-dev}
\end{align}
Note that for $n \geq \nobs (P, \eps)$ (resp.~$n \geq \nmiss (P, d, \eps)$), one has $s_n (P)/n \leq \eps$ (resp.~$s_n^\circ (P)/n \leq \eps/\log (e d/\eps n)$) since $s_n (P)/n$ (resp.~$s_n^\circ (P)/n$ and $\log (e d/\eps n)$) is non-increasing in $n$, as noted above.
The sample sizes~\eqref{eq:critical-sample-observed} and~\eqref{eq:critical-sample-dev} serve to control the first and third term in the error bound~\eqref{eq:upper-bound-sparse-confident}.
In addition, up to constant factors, the second sample size allows one to control the second term in the bound~\eqref{eq:upper-bound-sparse-confident}.
Indeed, bounding $s_n \geq s_{n/112}^\circ$, the second term is at most of order $ s_{n/112}^\circ \log (e d/s_{n/112}^\circ) / n$.
In addition, up to replacing $n$ by $112 n$, for this term to be bounded by $C \eps$ for some absolute constant $C > 1$, it suffices that (recalling that $s_n^\circ \leq d$)
\begin{equation*}
  \frac{s_n^\circ \log (e d/s_n^\circ)}{n}
  \lesssim \eps
  \quad \text{i.e.} \quad
  \frac{d/s_n^\circ}{\log (e d/s_n^\circ)}
  \gtrsim \frac{d}{\eps n}
  \quad \text{i.e.} \quad
  \frac{d}{s_n^\circ}
  \gtrsim \frac{d}{\eps n} \log \Big( \frac{e d}{\eps n} \Big)
  \, ,
\end{equation*}
which amounts to the condition in~\eqref{eq:critical-sample-missing}.

The upper bound~\eqref{eq:upper-bound-sparse-confident} from Theorem~\ref{thm:upper-bound-sparse}, together with the previous discussion, implies the following:
for some absolute constants
$c_1, c_2 \geq 1$ (one may take $c_1 = 112$), %
if
\begin{equation}
  \label{eq:critical-sample-size-total}
  n
  \geq c_1 \max \big\{ \nobs (P, \eps), \nmiss (P, d, \eps), \ndev (d, \eps, \delta) \big\}
  \, ,
\end{equation}
then the estimator $\adpdelta$ satisfies
\begin{equation}
  \label{eq:high-probability-bound}
  \P_P \big( \kll{P}{\adpdelta} \geq c_2 \, \eps \big)
  \leq \delta
  \, .
\end{equation}

\section{High-probability bound on the missing mass}
\label{sec:missing-mass}

In this section, we present a result that plays an important role in the proof of Theorem~\ref{thm:upper-bound-sparse}, namely in the high-probability analysis of adaptive estimators, but which may also be of independent interest.

Specifically, the following quantities appear naturally in our analysis and in other contexts:

\begin{definition}[Missing and underestimated masses]
  \label{def:missing-underestimated-mass}
  Given a distribution $P$ and an \iid sample $X_1, \dots, X_n$ from $P$,
  The \emph{missing mass} $M_n = \sum_{j=1}^d p_j \indic{N_j = 0}$ is the total mass under $P$ of
  classes that do not appear in the sample.
  
  We also define the \emph{underestimated mass} $U_n = \sum_{j=1}^d p_j \indic{N_j \leq np_j/4}$ as the total mass of classes whose empirical frequency underestimates their true probability by at least a factor of $4$.
\end{definition}

The missing and underestimated masses depend on both the sample $X_1, \dots, X_n$ and on the true distribution; as such, they are not ``observable'' from the data.

It is clear from the definition that $M_n \leq U_n$.
As it happens, the quantity that plays a role in our analysis is the underestimated mass $U_n$, and thus our main goal is to provide a high-probability upper bound on $U_n$.
On the other hand, the missing mass $M_n$ is a classical quantity, which has been studied in Statistics since the work of Good~\cite{good1953population}, and our bound on $U_n$ will also yield a new high-probability upper bound on $M_n$.

Specifically, our aim is to address the following question:
\begin{quote}
  For a given $\eps > 0$ and $\delta \in (0, e^{-1})$ and a distribution $P \in \probas_d$, how large must the sample size $n$ be to ensure that the underestimated mass $U_n$ (and thus the missing mass $M_n$) is smaller than $\eps$ with probability at least $1-\delta$?
\end{quote}

Before stating our main result, we first clarify that the behavior of the expected missing mass is essentially governed by the parameter $s_n^\circ (P)$ (Definition~\ref{def:effective-missing-support}), a fact we alluded to in Section~\ref{sec:adaptation-support}.

\begin{lemma}
  \label{lem:expected-missing-mass}
  For $P \in \probas_d$, 
  let $s_n^\bullet (P) = \sum_{j=1}^d (np_j) e^{-np_j}$ for $n \in (0, +\infty)$, and recall the definitions of
  $s_n^\circ (P)$ and $M_n$ (Definitions~\ref{def:effective-missing-support} and~\ref{def:missing-underestimated-mass}).
  For any integer $n \geq 1$, the following holds:
  \begin{enumerate}
  \item $\E_P [M_n] = \sum_{j=1}^d p_j (1-p_j)^n$;
  \item $s_{2n}^\bullet (P)/(2n) - e^{-0.3 n} \leq \E_P [M_n] \leq s_n^\bullet (P)/n$;
  \item $e^{-1} s_n^\circ (P) \leq s_n^{\bullet} (P) \leq 2 s_{n/2}^\circ (P)$.
  \item For $n \geq 3$, letting $s_n^\diamond (P) = \E_P [D_{2n} - D_n]$, one has $s_{2n}^\circ (P)/12 - e^{-n} \leq s_n^\diamond (P) \leq s_n^\circ (P)$.
  \end{enumerate}
\end{lemma}

In short, the quantities $n \, \E_P [M_n], s_n^\bullet (P), s_n^\circ (P)$
are essentially equivalent,
up to constant factors in their values and in the sample size $n$, and possibly additive exponentially small terms.

\begin{proof}[Proof of Lemma~\ref{lem:expected-missing-mass}]
  For the first identity, write $\E_P [M_n] = \sum_{j=1}^d p_j \P_P(N_j = 0) = \sum_{j=1}^d p_j (1-p_j)^n$.
  The upper bound $\E_P [M_n] \leq s_n^\bullet (P)/n$ comes from the fact that $(1-p_j)^n \leq e^{-np_j}$.
  
  For the lower bound, given $\lambda \in \R^+$ let $N (\lambda) \sim \poissondist (\lambda)$, and set $N_j^{(\lambda)} = \sum_{1 \leq i \leq N(\lambda)} \indic{X_i = j}$ and $M^{(\lambda)} = M_{N (\lambda)} = \sum_{j=1}^d p_j \indic{N_j^{(\lambda)} = 0}$.
  It is a classical fact (\eg, this follows from~\cite[Theorem~5.6 p.~100]{mitzenmacher2017probability} and the convolution property of Poisson distribution) that $N_j^{(\lambda)} \sim \poissondist (\lambda p_j)$.
  Thus $\E [M^{(\lambda)}] = \sum_{j=1}^d p_j \P (N_j^{(\lambda)} = 0) = \sum_{j=1}^d p_j e^{-\lambda p_j} = s_\lambda^{\bullet} (P)/\lambda$.
  Clearly, if $N (2n) \geq n$ then $M^{(2n)} \leq M_n$.
  On the other hand, if $N (2n) < n$ we may write $M_n \geq 0 \geq M^{(2n)} - 1$.
  Hence, 
  \begin{equation*}
    n \, \E_P [M_n]
    \geq n\, \E_P [M^{(2n)} - \indic{N(2n) < n}]
    = n \times \frac{s_{2n}^\bullet (P)}{2n} - n \P (N(2n) < n)
    \, ,
  \end{equation*}
  and the desired claim follows from the bound $\P (N(2n) < n) \leq \exp (-D(n, 2n))
  \leq e^{-(1-\log 2)n} %
  \leq e^{-0.3 n}$ by Lemma~\ref{lem:poisson-tail}.

  For the third point, we apply the bounds $(np_j) e^{-np_j} \leq np_j$ and $(np_j) e^{-np_j} = 2 (np_j/2) e^{-np_j} \leq 2 e^{-1+np_j/2} e^{-np_j} \leq e^{1-n p_j/2}$ to get the upper bound.
  To get the lower bound, we use that $np_j \geq 1$ when $p_j \geq 1/n$, and $e^{-np_j} \geq e^{-1}$ when $p_j < 1/n$.

  Finally, for the fourth point, write
  \begin{equation*}
    s_n^\diamond (P)
    = \E_P [D_{2n} - D_n]
    = \sum_{j=1}^d \big\{ (1-p_j)^n - (1 - p_j)^{2n} \big\}
    = \sum_{j=1}^d (1-p_j)^n \big\{ 1 - (1 - p_j)^{n} \big\}
    \, .
  \end{equation*}
  The upper bound $s_n^\diamond (P) \leq s_n^\circ (P)$ follows from the bounds $(1-p_j)^n \leq e^{-np_j}$ and $1 - (1-p_j)^n \leq \min (1, np_j)$.
  For the lower bound, consider first indices $j$ for which $np_j \leq 1$.
  In this case, one has $1-(1-p_j)^n \geq (1-e^{-1}) n p_j$ (see the proof of Fact~\ref{fac:expected-support}) while $(1-p_j)^n \geq (1-1/n)^n \geq (2/3)^3$, thus $(1-p_j)^n \{ 1 - (1 - p_j)^{n} \} \geq (2/3)^3 (1-e^{-1}) n p_j \geq n p_j/6 \geq \min (2 n p_j, e^{1-2np_j}) / 12$.
  When $np_j > 1$, one has $1 - (1-p_j)^n > 1 - (1-1/n)^n \geq 1- e^{-1}$, while $(1-p_j)^n \geq e^{-2np_j} - e^{-n} \indic{p_j >3/4}$ as $1-p > e^{-2p}$ for $p \in [0, 3/4]$; thus
  $    (1-p_j)^n\{ 1 - (1 - p_j)^{n} \} \geq (1-e^{-1}) e^{-2np_j} - e^{-n} \indic{p_j > 3/4} \geq {\min (2np_j, e^{1-2np_j})}/{12} - e^{-n} \indic{p_j > 3/4}$.
  Combining the previous inequalities %
  and using that there is at most one index $j$ such that $p_j > 3/4$ concludes the proof.
\end{proof}

In particular, the inequalities of Lemma~\ref{lem:expected-missing-mass} imply that:
\begin{equation}
  \label{eq:link-missing-s-n-zero}
  \frac{s_{2n}^\circ (P)}{2 e n} - e^{-0.3 n}
  \leq \E_P [M_n]
  \leq \frac{2 s_{n/2}^\circ (P)}{n}
  \, .
\end{equation}
(In addition, the $e^{-0.3 n}$ term can often be ignored: for instance, it is dominated by the first term for large $n$ whenever there is a class $j$ with $p_j \leq 0.1$, which occurs whenever at least $10$ classes have nonzero probability.)
Hence, ignoring the $e^{-0.3 n}$ term, the
behavior of the typical value $\E_P [M_n] $ of the missing mass is essentially equivalent to that of $s_n^\circ (P)/n$.

Theorem~\ref{thm:deviation-missing-mass} below provides a deviation upper bound on the underestimated and missing masses, involving the same complexity parameter as the in-expectation estimates~\eqref{eq:link-missing-s-n-zero}:

\begin{theorem}
  \label{thm:deviation-missing-mass}
  For any $n, d \geq 2$, $\delta\in (0,e^{-1})$ and any $P \in \probas_d$, with probability at least $1- 8\delta$ under $P$ one has
  \begin{equation}
    \label{eq:deviation-missing-mass}
    M_n
    \leq U_n
    \leq 
    \frac{336 \, s_{n/112}^\circ (P) + 2500 e \log (1/\delta)}{n}
    \, .
  \end{equation}
\end{theorem}

The proof of Theorem~\ref{thm:deviation-missing-mass} %
is provided in Section~\ref{sec:proof-deviation-missing-mass}.
Roughly speaking, the proof %
relies on
a careful control of the contribution to $U_n$ of classes $p_j$ of each order of magnitude, followed by a combination of these %
per-scale contributions.

As we argue now, Theorem~\ref{thm:deviation-missing-mass} constitutes an
almost
optimal high-probability upper bound on the missing mass, up to constant factors in the mass and sample size.
Indeed, combining~\eqref{eq:link-missing-s-n-zero} and Theorem~\ref{thm:deviation-missing-mass} shows that
for every $\delta \in (0, e^{-1})$, with probability at least $1-\delta$,
\begin{equation}
  \label{eq:missing-mass-bernstein}
  M_n
  \leq U_n
  \leq 3 e\, \E_P [M_{n/224}] %
  + 2600 e \, \frac{\log (1/\delta)}{n}
  \, .
\end{equation}
(Specifically, from~\eqref{eq:link-missing-s-n-zero} we obtain an additive term in $3 e^{1-0.15 n} \leq 20e\log (1/\delta)/n$.)
In addition, as will be discussed below, the deviation term in $\log (1/\delta)/n$ is
necessary, except in situations where the probabilities in $P$ exhibit a significant gap; and in this case, the bound of Theorem~\ref{thm:deviation-missing-mass} admits a simple strengthening that addresses its sub-optimality.

\paragraph{Critical sample size.}

Thanks to Theorem~\ref{thm:deviation-missing-mass}, we can address the question of the critical sample size for which the missing mass is small with high probability.
Specifically, for $P \in \probas_d$ and $\eps \in (0, 1)$,
define
\begin{align}
  \label{eq:critical-missing-not}
  N^\circ (P, \eps)
  &= \inf \bigg\{ n \geq 1 : \frac{s_n^\circ (P)}{n} \leq \eps \bigg\}
    \,
    .
\end{align}
Note that $s_n^\circ (P)/n \leq \eps$ for $n \geq N^\circ (P, \eps)$ as $s_n^\circ(P)/n$ is non-increasing in $n$.
Then, it follows from Theorem~\ref{thm:deviation-missing-mass}
that
for any $\delta \in (0, e^{-1})$, if
\begin{equation}
  \label{eq:critical-missing-whp}
  n
  \geq
  2500 \max \Big\{ N^\circ (P, \eps), \frac{\log (1/\delta)}{\eps} \Big\}
  \, ,
\end{equation}
then $\P_P (M_n \geq 6 %
\eps) \leq \delta$.
Below, we show that the estimate~\eqref{eq:critical-missing-whp} is optimal in many situations, and that a simple tightening is optimal in the general case.
We finally compare our condition with those derived from previous upper bounds on the missing mass in the literature.

\paragraph{Necessity of the complexity parameter.}

We first argue that the condition $n \gtrsim N^\circ (P, \eps)$ is almost necessary for the missing mass to be small in expectation.
Indeed, the latter condition amounts to
\begin{equation}
  \label{eq:critical-expected-missing}
  n
  \geq \nexp (P, \eps)
  = \inf \Big\{ n \geq 1 : \E_P [M_n] \leq \eps \Big\}
  \, .
\end{equation}
But by inequality~\eqref{eq:link-missing-s-n-zero}, one has for some absolute constant $c > 1$:
\begin{equation}
  \label{eq:equiv-ncirc-nexp}
  N^\circ (P, c \, \eps)
  \leq c \, \nexp (P, \eps) + c \log \Big( \frac{1}{\eps} \Big)
  \, ,
\end{equation}
hence whenever $\log (1/\eps) \ll N^\circ (P, \eps)$ (which as discussed above occurs in most situations of interest)
the condition $n \gtrsim N^\circ (P, \eps)$ is necessary.
Besides, whenever the second term $\log (1/\delta)/\eps \geq \log(1/\eps)$ in~\eqref{eq:critical-missing-whp} is necessary, so is $\log (1/\eps)$ and thus $N^\circ (P, \eps)$.

In fact, it follows from inequality~\eqref{eq:missing-mass-bernstein} that $\P_P (U_n \geq 4 e \eps) \leq \delta$ whenever
\begin{equation}
  \label{eq:condition-optimal-exp}
  n \geq
  2600 \max \bigg\{ \nexp (P, \eps), \frac{\log (1/\delta)}{\eps} \bigg\}
  \, .
\end{equation}

\paragraph{%
  Deviation term and refinement.%
}

We now address the question of whether the condition $n \gtrsim \log (1/\delta)/\eps$ is necessary.
As it turns out, this is often but not always the case; however,
a simple strengthening of this result
does provide a necessary condition.
In order to state this refinement, define the set
of all sums of class probabilities in $P$:
\begin{equation}
  \label{eq:set-sums-P}
  \setsums (P)
  = \bigg\{ p_J = \sum_{j \in J} p_j : J \subset [d] \bigg\}
  \subset [0, 1]
  \, .
\end{equation}
It is clear that $U_n, M_n \in \setsums (P)$.
Hence, in order to ensure that
$M_n < \eps$,
it suffices that
$M_n < \ol \eps$,
where we let
\begin{equation}
  \label{eq:upper-epsilon}
  \ol \eps
  = \ol \eps (P, \eps)
  = \inf \big( \setsums (P) \cap [\eps, 1] \big)
  \, .
\end{equation}
(Equivalently, if $M_n \leq t$ then $M_n \leq \ul b (P, t) = \sup \big( \setsums (P) \cap [0, t] \big)$, which
strengthens
the bound of Theorem~\ref{thm:deviation-missing-mass}.)
Hence,
condition~\eqref{eq:critical-missing-whp} ensures that
$\P_P (M_n \geq 6 \eps) \leq \delta$ whenever
\begin{equation}
  \label{eq:critical-missing-strengthened}
  n
  \geq 15000 \max \Big\{ N^\circ (P, \eps), \frac{\log (1/\delta)}{\ol \eps (P, \eps)
  } \Big\}
  \, .
\end{equation}
Indeed, if $\ol \eps \leq 6 \eps$, apply~\eqref{eq:critical-missing-whp} and bound $1/\eps \leq 6/\ol \eps$.
On the other hand, if $\ol \eps > 6 \eps$, apply~\eqref{eq:critical-missing-whp} to $\ol \eps/6 > \eps$, so that with probability $1-\delta$ one has $M_n < 6 \times \ol \eps/6  = \ol \eps$, hence $M_n < \eps$, and use that $N^\circ (P, \ol \eps/6) \leq N^\circ (P, \eps)$.

We now show that the condition $n \gtrsim \log (1/\delta)/\ol \eps$ is necessary to ensure this high-probability bound, whenever $\ol \eps$ is bounded away from $1$.
For concreteness, assume that $\ol \eps \leq 3/4$, so that $1-\ol \eps > e^{-2 \ol \eps}$.
Let $J \subset [d]$ such that $p_J = \sum_{j \in J} p_j = \ol \eps$, and let $N_J = \sum_{j \in J} N_j = \sum_{i=1}^n \indic{X_i \in J}$.
Clearly,
if $N_J = 0$ then $M_n \geq p_J = \ol \eps \geq \eps$.
Hence, $\P_P (M_n \geq \eps)
\geq \P_P (N_J = 0) = (1 - p_J)^n > e^{-2\ol \eps n} \geq \delta$ whenever $n \leq \log (1/\delta)/(2\ol \eps)$.

Summarizing the previous discussion, we
obtain
up to universal constants the critical sample size after which the missing mass is small, in expectation and with high probability (the sufficient condition being deduced from~\eqref{eq:condition-optimal-exp} in the same way that~\eqref{eq:critical-missing-strengthened} is deduced from~\eqref{eq:critical-missing-whp}).

\begin{corollary}
  \label{cor:sample-complexity-missing}
  Let $\eps \in (0, 3/4)$, $\delta \in (0, e^{-1})$, $n \geq 1$ and
  $P \in \probas_d$.
  \begin{enumerate}
  \item When $\ol \eps (P, \eps) \leq 3/4$, if $\E_P [M_n] \leq \eps$ and $\P_P (M_n \geq \eps) \leq \delta$, then
    \begin{equation*}
      n \geq \frac{1}{2} \max \bigg\{ \nexp (P, \eps), \frac{\log (1/\delta)}{\ol \eps(P, \eps)} \bigg\}
      \, .
    \end{equation*}
  \item Conversely, if
    \begin{equation}
      \label{eq:sample-complexity-opt}
      n
      \geq 11000 e \max \bigg\{ \nexp (P, \eps), \frac{\log (1/\delta)}{\ol \eps (P, \eps)} \bigg\}
      \, ,
    \end{equation}
    then $\E_P [M_n] \leq \eps$ and $\P_P (M_n \geq 4 e \eps) \leq \delta$.
  \end{enumerate}  
\end{corollary}

Since the optimal deviation term scales as $\log (1/\delta)/\ol \eps$, the deviation term $\log(1/\delta)/\eps$ is optimal
unless $\ol \eps \gg \eps$.
Distributions for which this occurs can be characterized as follows:

\begin{fact}
  \label{fac:gap-eps}
  Let $\eps \in (0, 1/2)$ and $\ol \eps \geq 2 \eps$.
  Let $P = (p_1, \dots, p_d) \in \probas_d$,
  with $p_1 \geq \dots \geq p_d$ without loss of generality.
  The following properties are equivalent:
  \begin{enumerate}
  \item[\textup(i\textup)] $\ol \eps (P, \eps) \geq \ol \eps$;
  \item[\textup(ii\textup)] there exists $j^* \in \{ 1, \dots, d \}$ such that $p_{j^*} \geq \ol \eps$ while $\sum_{j^* < j \leq d} p_j < \eps$.    
  \end{enumerate}
  In this case, one has
  $\ol\eps (P, \eps) = p_{j^*}$.
\end{fact}

\begin{proof}
  We first prove the implication $(i) \Rightarrow (ii)$.
  Let $j^* = \max \{ 1 \leq j \leq d : p_j \geq \eps \} \in \{ 0, \dots, d \}$, with the convention that $j^* = 0$ if the set is empty.
  
  We first show that $\sum_{j^* < j \leq d} p_j < \eps$.
  If $j^* = d$, the sum equals $0$ and the property holds; we thus assume that $j^* \leq d-1$.
  In this case, let $j'$ denote the largest integer in $\{ j^*+1, \dots, d\}$ such that $\sum_{j^* < j \leq j'} p_j < \eps$ (by definition of $j^*$, this inequality holds for $j' = j^*+1$, hence $j' \geq j^*+1$).
  We need to show that $j' = d$.
  We proceed by contradiction and assume that $j' \leq d-1$.
  In this case, one has $\sum_{j^* < j \leq j'} p_j < \eps$ while $\sum_{j^* < j \leq j'+1} p_j \geq \eps$.
  But since $\sum_{j^* < j \leq j'+1} p_j \in \setsums (P)$, this also implies that $\sum_{j^* < j \leq j'+1} p_j \geq \ol \eps \geq 2\eps$, hence $p_{j'+1} = \sum_{j^* < j \leq j'+1} p_j - \sum_{j^* < j \leq j'} p_j
  > \eps$.
  But this contradicts the fact that $p_{j'+1} \leq p_{j^*+1} < \eps$, proving the claim of this paragraph.

  Now since $\sum_{0 < j^* \leq d} p_j = 1$, the inequality $\sum_{j^* < j \leq d} p_j < \eps < 1$ implies that $j^* \geq 1$.
  By definition of $j^*$, one has $p_{j^*} \geq \eps$; but since $p_{j^*} \in \setsums(P)$, this implies that $p_{j^*} \geq \ol \eps (P, \eps) \geq \ol\eps$.

  We now conclude with the implication $(ii)\Rightarrow (i)$.
  Let $J \subset [d]$.
  If $J \subset \{ j^*+1, \dots, d \}$, then $\sum_{j \in J} p_j \leq \sum_{j^*<j\leq d} p_j < \eps$.
  Otherwise, $J$ contains an element $j' \in \{1, \dots, j^*\}$, hence $\sum_{j \in J} p_j \geq p_{j'} \geq p_{j^*} \geq \ol\eps$.
  Since the inequality $\sum_{j \in J} p_j \geq p_{j^*}$ is an equality for $J = \{j^*\}$, one has $\ol \eps (P, \eps) = p_{j^*} \geq \ol \eps$.
\end{proof}

It follows from Fact~\ref{fac:gap-eps} that if $\ol\eps (P, \eps) \gg \eps$, then there exists $j^* \in \{ 1, \dots, d \}$ such that $p_{j^*} \gg \eps$ while $\sum_{j^* < j \leq d} p_j < \eps$.
Then, either $j^* = d$ and $p_d\gg \eps$, or $j^* \leq d-1$ and $\sum_{j=j^*+1}^d p_j < \eps \ll p_{j^*}$, so in particular $p_{j^*+1} \ll p_{j^*}$.

We note that the latter property occurs neither for polynomial decay of probabilities (as defined in Remark~\ref{rem:example-distributions}) nor for exponential decay, for which $p_{j+1} \gtrsim p_j$ for every $j<d$.
However, this property does occur in the case of sparse distributions (third example of Remark~\ref{rem:example-distributions}), for which the refinement in terms of $\ol \eps (P, \eps)$ brings an improvement.

\paragraph{Previous results.}

The behavior of the missing mass and the question of its estimation have been studied in a rich literature, which can be traced to the work of Good~\cite{good1953population}
introducing the Good-Turing estimate.
In what follows, we focus our discussion on existing high-probability bounds on the missing mass, and %
on
conditions they imply on the sample size for the missing mass to be small.
We additionally refer to the discussion in~\cite{benhamou2017concentration} (and references therein)
for background and references on the distinct question of estimation of the missing mass.

First, a deviation bound due to
McAllester and Schapire~\cite{mcallester2000convergence}, with constants later improved by McAllester and Ortiz~\cite{mcallester2003concentration} 
(see also Berend and Kontorovitch~\cite{berend2013concentration} for an alternative approach), shows that the missing mass exhibits sub-Gaussian tails.
Specifically, the following distribution-free deviation bound holds~\cite[Theorem~16]{mcallester2003concentration}: for any $d \geq 2$, distribution $P \in \probas_d$ and $\delta \in (0, 1)$, one has
\begin{equation}
  \label{eq:mcallester-ortiz}
  \P_P \bigg( M_n \geq \E_P [M_n] + \sqrt{\frac{\log (1/\delta)}{n}} \bigg)
  \leq \delta
  \, .
\end{equation}
This implies that for any $\eps \in (0, 1)$ and $\delta \in (0, e^{-1})$, if
\begin{equation}
  \label{eq:sample-cond-mcallester-ortiz}
  n
  \geq \max \bigg\{ \nexp (P, \eps), \frac{\log (1/\delta)}{\eps^2} \bigg\}
  \, ,
\end{equation}
then $\P_P (M_n \geq 2 \eps) \leq \delta$.
While significant and nontrivial, 
condition~\eqref{eq:sample-cond-mcallester-ortiz} exhibits a suboptimal dependence on $\eps$ compared to~\eqref{eq:condition-optimal-exp}, although it involves smaller numerical constants.

To the best of our knowledge, the previous best deviation bound on the missing mass is due to Ben-Hamou, Boucheron and Ohannessian~\cite{benhamou2017concentration}, tightening previous multiplicative concentration bounds by Ohannessian and Dahleh~\cite{ohannessian2012rare}.
Specifically, the following tail bound follows from~\cite[Theorem~3.9]{benhamou2017concentration}:
\begin{align}
  \label{eq:benhamou}
  \P_P \bigg( M_n \geq \E_P [M_n]
  &+ \frac{\sqrt{2 d_n^+ \log (1/\delta)}}{n} + \frac{\log (1/\delta)}{n} \bigg)
    \leq \delta \, , \\
  \text{where} \qquad
  d_n^+
  = d_n^+ (P)
  &= \sum_{j=1}^d \big\{ 1 - (n p_j + 1) e^{-np_j} \big\}
    \asymp \sum_{j=1}^d \min \big[ 1, (n p_j)^2 \big]
    \, .
\end{align}
(Note that $d_n^+ (P) \asymp s_n (P)$ whenever
$\sum_{j \pp p_j < 1/n} (np_j) \lesssim \sum_{j \pp p_j \geq 1/n} 1$.)
Since $d_n^+ \leq \sum_{j=1}^d (1 - e^{-np_j}) \leq \sum_{j=1}^d (n p_j) = n$, this bound recovers~\eqref{eq:mcallester-ortiz} up to constants in the regime of interest $\delta \in (e^{-n}, 1)$.
Furthermore, as shown in~\cite{benhamou2017concentration}, inequality~\eqref{eq:benhamou} strictly improves~\eqref{eq:mcallester-ortiz} in most situations, and in fact provides sharp results in several interesting cases.

However, there are situations in which the second term in~\eqref{eq:benhamou} dominates and leads to a suboptimal bound.
For instance, consider the case where $p_j \asymp 1/d$ for $1 \leq j \leq d/2$, while $p_j \asymp 2^{- (j-d/2)} / d$ for $d/2 < j \leq d$, and the regime of constant probability $\delta \in (0, 1)$.
For $d \log d \ll n \ll 2^d$, the typical missing mass is at most of order $\E_P [M_n] \asymp 1/n$.
However, in this regime one has $d_n^+ \asymp d$, hence the bound~\eqref{eq:benhamou} is of order $\sqrt{d}/n$, which is suboptimal by a $\sqrt{d}$ factor.
This $\sqrt{d}$ factor can be removed by resorting to Theorem~\ref{thm:deviation-missing-mass}.

\section{Proof of high-probability upper bounds}
\label{sec:proof-high-prob}

In this section, we provide the proof of high-probability upper bounds for estimation in this paper, namely Theorems~\ref{thm:upper-laplace},~\ref{thm:upper-bound-conf-dependent} and~\ref{thm:upper-bound-sparse} (the proof of the last result relies in part on Theorem~\ref{thm:deviation-missing-mass} on the underestimated mass, which is proved in Section~\ref{sec:proof-deviation-missing-mass}).

Specifically, we start with %
lemmata
that are used in the proof of all high-probability upper bounds, before concluding with the proof of each specific result.

\subsection{Risk decomposition}
\label{sec:risk-decomposition}

All estimators we consider are ``add-$\lambda$'' smoothing rules, where $\lambda$ may
be confidence-dependent and/or data-dependent.
We therefore start our analysis with the following deterministic risk decomposition for add-$\lambda$ estimators.
Below, we let $\ol p_j = N_j/n$ be empirical frequency of class $j=1, \dots, d$.

\begin{lemma}
  \label{lem:decomp-risk}
  Consider the %
  distribution
  $\wh P_n = (\wh p_1, \dots, \wh p_d) %
  $
  given by
  \begin{equation}
    \label{eq:estimator-add-constant}
    \wh p_j
    = \frac{N_j + \lambda}{n + \lambda d}
    \, , \quad j = 1, \dots, d
    \, ,
  \end{equation}
  for some $\lambda \in (0, n/d]$ that may depend on $X_1, \dots, X_n$.
  Then, we have
  \begin{equation}
    \label{eq:decomp-kl-upper}
    \kll{P}{\wh P_n}
    \leq 6 \sum_{j=1}^d \Big(\sqrt{\ol p_j} - \sqrt{p_j}\Big)^2 + \frac{7\lambda d}{n} %
    + \sum_{j \pp p_j \geq 4 \lambda/n} p_j \log \Big( \frac{2 n p_j}{\lambda} \Big) \bm 1 \Big( N_j \leq \frac{n p_j}{4} \Big)    
    \, .
  \end{equation}
\end{lemma}

The decomposition~\eqref{eq:decomp-kl-upper}
features three terms.

The first term, which does not depend on the estimator (that is, on $\lambda$), corresponds to the squared Hellinger distance between the empirical distribution $\ol P_n = (\ol p_1, \dots, \ol p_d)$ and the true distribution $P$.
It constitutes a natural ``hard limit'' for our estimation guarantees in relative entropy.
This term is controlled in Section~\ref{sec:upper-bound-hellinger}.

The second term accounts for the ``bias'' due to the use of regularization, which diminishes the probability assigned to high-frequency classes.
This term increases with the smoothing parameter $\lambda$.
Its control is immediate when $\lambda$ is data-independent, and relatively straightforward when $\lambda$ is data-dependent.

The third and final term accounts for the contribution of classes whose frequency is significantly underestimated, which inflates the relative entropy over ideal asymptotic rates.
The effect of underestimation of true frequencies is mitigated by the use of smoothing,  
and indeed this term decreases with the smoothing parameter $\lambda$.
The control of this term is at the core of the analysis, and is carried out in Section~\ref{sec:contr-contr-under}.

\begin{proof}[Proof of Lemma~\ref{lem:decomp-risk}]
  By Lemma~\ref{lem:kl-hellinger-bounded}, for any $p, q \in \R^+$, if $q \geq p/8$ then $D (p, q) \leq \phi(8) (\sqrt{p} - \sqrt{q})^2 \leq 3 (\sqrt{p} - \sqrt{q})^2$.
  On the other hand, if $q \leq p/8$ then $D (p, q) = p \log (p/q) - p + q \leq p \log (p/q)$.
  Hence, for any $p, q \in \R^+$, we have
  \begin{equation}
    \label{eq:kl-hellinger-plus-res}
    D (p, q)
    \leq 3 (\sqrt{p} - \sqrt{q})^2 + p \log \Big( \frac{p}{q} \Big) \indic{q \leq p/8}
    \, .
  \end{equation}
  It follows from this inequality that
  \begin{equation}
    \label{eq:proof-decomp-kl-1}
    \kll{P}{\wh P_n}
    = \sum_{j=1}^d D (p_j, \wh p_j)
    \leq 3 \sum_{j=1}^d \left( \sqrt{\wh p_j} - \sqrt{p_j} \right)^2 + \sum_{j=1}^d p_j \log \Big( \frac{p_j}{\wh p_j} \Big) \ind \Big( \wh p_j \leq \frac{p_j}{8} \Big)
    \, .
  \end{equation}

  We now bound the two terms of the right-hand side of~\eqref{eq:proof-decomp-kl-1}, starting with the first one.
  First, since $\wh p_j = (N_j + \lambda)/(n+ \lambda d) = (\ol p_j + \lambda/n)/(1 + \lambda d/n)$, we have
  $\ol p_j / (1 + \lambda d/n) \leq \wh p_j \leq \ol p_j + \lambda/n$.
  We therefore have
  \begin{align}
    &\Big(\sqrt{\wh p_j} - \sqrt{p_j} \Big)^2
    \leq 2 \Big( \sqrt{\wh p_j} - \sqrt{\ol p_j} \Big)^2 + 2 \Big( \sqrt{\ol p_j} - \sqrt{p_j} \Big)^2 \nonumber \\
    &\leq 2 \max \Big\{ \Big( \sqrt{\ol p_j} - \sqrt{\ol p_j} \Big( 1 + \frac{\lambda d}{n} \Big)^{-1/2} \Big)^{2}, \Big( \sqrt{\ol p_j + \frac{\lambda}{n}} - \sqrt{\ol p_j} \Big)^2 \Big\} + 2 \Big( \sqrt{\ol p_j} - \sqrt{p_j} \Big)^2 \nonumber \\
    &\leq 2 \Big( 1 -  \Big( 1 + \frac{\lambda d}{n} \Big)^{-1/2} \Big)^{2} \ol p_j + 2 \Big( \sqrt{\ol p_j + \frac{\lambda}{n}} - \sqrt{\ol p_j} \Big)^2 + 2 \Big( \sqrt{\ol p_j} - \sqrt{p_j} \Big)^2
      \label{eq:decomp-helling-regularized}
      \, .
  \end{align}
  Using that, for any $u \in \R^+$, one has
  $1 - (1+u)^{-1/2} = u / (1+ u + \sqrt{1+u}) \leq u / (2 \sqrt{u} + \sqrt{u}) = \sqrt{u}/3$, %
  we can bound
  the first term in~\eqref{eq:decomp-helling-regularized} as
  \begin{equation*}
    \Big( 1 -  \Big( 1 + \frac{\lambda d}{n} \Big)^{-1/2} \Big)^{2} \ol p_j
    \leq \bigg( \frac{\sqrt{\lambda d/n}}{3} \bigg)^2 \ol p_j
    \leq \frac{\lambda d}{9 n} \cdot \ol p_j
    \, .
  \end{equation*}
  Likewise, we can bound the second term as follows:
  \begin{equation*}
    \Big( \sqrt{\ol p_j + \lambda/n} - \sqrt{\ol p_j} \Big)^2
    = \bigg( \frac{(\ol p_j + \lambda/n) - \ol p_j}{\sqrt{\ol p_j + \lambda/n} + \sqrt{\ol p_j}} \bigg)^2
    \leq \bigg( \frac{\lambda/n}{\sqrt{\lambda/n}} \bigg)^2
    = \frac{\lambda}{n}
    \, .
  \end{equation*}
  Plugging these two bounds into~\eqref{eq:decomp-helling-regularized}, summing over $j=1, \dots, d$ and using that $\sum_{j=1}^d \ol p_j = 1$, we deduce that
  \begin{align}
    \label{eq:proof-regularized-hellinger}
    \sum_{j=1}^d \Big(\sqrt{\wh p_j} - \sqrt{p_j} \Big)^2
    &\leq \frac{2 \lambda d}{9 n} \cdot \sum_{j=1}^d \ol p_j + \frac{2 \lambda d}{n} + 2 \sum_{j=1}^d \Big( \sqrt{\ol p_j} - \sqrt{p_j} \Big)^2 \nonumber \\
    &= \frac{20 \lambda d}{9 n} + 2 \sum_{j=1}^d \Big( \sqrt{\ol p_j} - \sqrt{p_j} \Big)^2
    \, .
  \end{align}

  We now turn to bounding the second term in the decomposition~\eqref{eq:proof-decomp-kl-1}.
  First, since $\lambda \leq n/d$, we have $n + \lambda d \leq 2 n$, and thus
  $\wh p_j \geq \max \{ N_j, \lambda \}/(2n)$.
  This implies that
  \begin{align}
    \label{eq:proof-decomp-bound-res-term}
    \sum_{j=1}^d p_j \log \Big( \frac{p_j}{\wh p_j} \Big) \ind \Big( \wh p_j \leq \frac{p_j}{8} \Big)    
    &\leq \sum_{j=1}^d p_j \log \Big( \frac{p_j}{\lambda/(2n)} \Big) \ind \Big( \frac{\max \{ N_j, \lambda \}}{2n} \leq \frac{p_j}{8} \Big) \nonumber \\
    &= \sum_{j \pp p_j \geq 4 \lambda/n} p_j \log \Big( \frac{2 n p_j}{\lambda} \Big) \ind \Big( N_j \leq \frac{n p_j}{4} \Big)
      \, .
  \end{align}
  Plugging upper bounds~\eqref{eq:proof-regularized-hellinger} and~\eqref{eq:proof-decomp-bound-res-term} into the decomposition~\eqref{eq:decomp-helling-regularized}
  concludes the proof.
\end{proof}

\subsection{Upper bound in Hellinger distance}
\label{sec:upper-bound-hellinger}

In this section, we proceed with the control of the first term in the decomposition of Lemma~\ref{lem:decomp-risk}.
Specifically, Lemma~\ref{lem:hellinger-empirical} below provides a high-probability bound
on the squared Hellinger distance between the empirical distribution $\ol P_n$ and the true distribution $P$.
This bound follows from the analysis of the reverse-relative entropy from~\cite{agrawal2022finite} (for large probabilities) combined with a simple binomial deviation bound (for small probabilities).

\begin{lemma}
  \label{lem:hellinger-empirical}
  Let $n,d \geq 2$ and $\delta \in (0,1)$.
  For any $P \in \probas_d$, letting
  $s_n (P) = \sum_{j=1}^d \min (1, n p_j)$
  and denoting by $\ol P_n = (\ol p_1, \dots, \ol p_d)$ the empirical distribution of $X_1, \dots, X_n$, one has
  \begin{equation}
    \label{eq:hellinger-empirical}
    \P_P \bigg( \sum_{j=1}^d \Big( \sqrt{\ol p_j} - \sqrt{p_j} \Big)^2
    \geq \frac{4 s_n (P) + 7 \log (1/\delta)}{n} \bigg)
    \leq 2 \delta
    \, .
  \end{equation}
\end{lemma}

We note in passing that this bound is of the right order of magnitude, as one may check that $\E_P [\sum_{j=1}^d (\sqrt{\ol p_j} - \sqrt{p_j})^2] \asymp s_n (P) /n$ whenever $\max_{j} p_j$ is bounded away from $1$.

\begin{proof}[Proof of Lemma~\ref{lem:hellinger-empirical}]
  Let $J^+ = \{ 1 \leq j \leq d : p_j \geq 1/n \}$ and
  $J^- = \set{1, \dots, d} \setminus J^+$.
  Also, let $s^+ = |J^+|$ and $s^- = n \sum_{j \in J^-} p_j$, so that $s_n (P) = s^+ + s^-$.
  From the inequalities $(\sqrt{p} - \sqrt{q})^2 \leq D (p, q)$ (Lemma~\ref{lem:kl-hellinger-bounded}) and $(\sqrt{p} - \sqrt{q})^2 \leq p + q$, it follows that
  \begin{equation}
    \label{eq:proof-hellinger-small-large}
    \sum_{j=1}^d \Big( \sqrt{\ol p_j} - \sqrt{p_j} \Big)^2
    \leq \sum_{j \in J^+} D \big( \ol p_j, p_j \big) + \sum_{j \in J^-} \big( p_j + \ol p_j \big)
    \, .
  \end{equation}

  \paragraph{First term.}
  We start by controlling the first term in~\eqref{eq:proof-hellinger-small-large}, following~\cite{agrawal2022finite}.
  Specifically, let $H^+ = \sum_{j \in J^+} D ( \ol p_j, p_j )$.
  By proceeding as in~\cite[Corollary~1.7]{agrawal2022finite}, except that in \cite[Proposition~2.4]{agrawal2022finite} we sum only over indices $j \in J^+$ (rather than over $1 \leq j \leq d$), we obtain the following inequality:
  for any $t \in (0, n/2)$,
  \begin{equation*}
    \log \E_P \big[ %
    e^{t (H^+ - \E_P [H^+])}
    \big]
    \leq \frac{4 s^+ t^2/n^2}{1-2t/n}
    \, .
  \end{equation*}
  In other words, $H^+ - \E_P [H^+]$ is sub-gamma~\cite[\S2.4]{boucheron2013concentration} with variance factor $8 s^+ /n^2$ and shape parameter $2/n$.
  Hence, by~\cite[p.~29]{boucheron2013concentration}, one has for every $\delta \in (0, 1)$,
  \begin{equation}
    \label{eq:proof-subgamma-kl}
    \P \bigg( H^+ - \E_P [H^+] \geq \frac{4 \sqrt{s^+ \log (1/\delta)}}{n} + \frac{2 \log (1/\delta)}{n} \bigg)
    \leq \delta
    \, .
  \end{equation}
  In addition, recalling the inequality $h(t) \leq (t-1)^2$ for any $t \in \R^+$ (Lemma~\ref{lem:entropy-function}), we have for all $p,q \in \R^+, q >0$:
  \begin{equation*}
    D (p, q)
    = q h \Big( \frac{p}{q} \Big)
    \leq q \Big( \frac{p}{q}-1 \Big)^2
    = \frac{(p-q)^2}{q}
    \, .
  \end{equation*}
  This implies that
  \begin{equation*}
    \E_P [H^+]
    = \sum_{j \in J^+} \E_P [D (\ol p_j, p_j)]
    \leq \sum_{j \in J^+} \frac{\E_P [(\ol p_j - p_j)^2]}{p_j}
    = \sum_{j \in J^+} \frac{p_j (1-p_j)/n}{p_j}
    \leq \frac{s^+}{n}
    \, .
  \end{equation*}
  Plugging the inequality into~\eqref{eq:proof-subgamma-kl}, we deduce that with probability at least $1-\delta$, we have
  \begin{equation}
    \label{eq:proof-bound-H+}
    H^+
    < \frac{s^+ + 4 \sqrt{s^+ \log (1/\delta)} + 2 \log (1/\delta)}{n}
    \leq \frac{(1 + 2 \sqrt{2}) s^+ + (2 + 2 \sqrt{2}) \log (1/\delta)}{n}
    \, .
  \end{equation}

  \paragraph{Second term.}
  We now turn to the second term in~\eqref{eq:proof-hellinger-small-large}.
  Note that
  \begin{equation*}
    n \sum_{j \in J^-} \ol p_j =
    \sum_{j \in J^-} N_j
    = \sum_{i=1}^n \sum_{j \in J^-} \indic{X_i = j}
    = \sum_{i=1}^n \indic{X_i \in J^-}
    \, ,
  \end{equation*}
  which follows a binomial distribution with parameters $n$ and $\sum_{j \in J^-} p_j = s^-/n$.
  Bernstein's inequality~\cite[Theorem~2.10 p.~37]{boucheron2013concentration} then yields, with probability at least $1-\delta$,
  \begin{equation}
    \label{eq:proof-bound-H-}
    \sum_{j \in J^-} \big( p_j + \ol p_j \big)
    < \frac{s^-}{n} + \frac{s^- + \sqrt{2 s^- \log (1/\delta)} + \log (1/\delta)}{n}
    < \frac{3 s^- + 2 \log (1/\delta)}{n}
    \, .
  \end{equation}

  \paragraph{Conclusion.}
  For $\delta \in (0, 1/2)$, with probability at least $1-2\delta$ both~\eqref{eq:proof-bound-H+} and~\eqref{eq:proof-bound-H-} hold.
  In this case, inequality~\eqref{eq:proof-hellinger-small-large} together with the fact that $s^+ + s^- = s_n (P)$ and the bound $2 \sqrt{2} \leq 3$ imply that 
  \begin{equation*}
    \sum_{j=1}^d \Big( \sqrt{\ol p_j} - \sqrt{p_j} \Big)^2
    < \frac{4 s_n (P) + 7 \log (1/\delta)}{n}
    \, .
  \end{equation*}
  This concludes the proof.
\end{proof}

\subsection{Control of the contribution of underestimated frequencies}
\label{sec:contr-contr-under}

The control of the second term, namely the bias due to the use of regularization, is immediate when the parameter $\lambda$ is not data-dependent;
we also postpone its analysis in the context of data-dependent regularization to the proof of Theorem~\ref{thm:upper-bound-sparse} below.

We now turn to the control of the key term in the decomposition of Lemma~\ref{lem:decomp-risk}, namely the third term accounting for the contribution of classes whose true frequency is significantly underestimated in the sample.
Specifically, in this section we establish the following control on the residual, which is arguably the core of the analysis:

\begin{lemma}
  \label{lem:tail-residual}
  Let $P \in \probas_d$.
  For any $\lambda \geq 1$, let $d_{\lambda} = | \set{1 \leq j \leq d : p_j \geq 4 \lambda/n} |$ and
  \begin{equation*}
    R_\lambda
    = \sum_{j \pp p_j \geq 4 \lambda/n} p_j \log \Big( \frac{2 n p_j}{\lambda} \Big) \bm 1 \Big( N_j \leq \frac{n p_j}{4} \Big)
    \, .
  \end{equation*}
  For any $\delta \in (e^{-n/6}, e^{-2})$,
  one has  
  \begin{equation}
    \label{eq:tail-residual-laplace}
    \P_P \bigg( R_1
    \geq \frac{62000 \, d_1 + 106000 \log (1/\delta) \log \log (1/\delta)}{n}
    \bigg)
    \leq 2 \delta
    \, .
  \end{equation}
  In addition, for any $\delta \in (e^{-n/6}, e^{-d}]$, if $\lambda = \log (1/\delta)/d$, then 
  \begin{equation}
    \label{eq:tail-residual-conf}
    \P_P \bigg( R_\lambda
    \geq %
    \frac{74000 \log (d) \log (1/\delta)}{n}
    \bigg)
    \leq 2 \delta
    \, .
  \end{equation}
\end{lemma}

Roughly speaking, inequality~\eqref{eq:tail-residual-laplace} will be used in the analysis of confidence-independent estimators such as the Laplace estimator (or of confidence-dependent estimators in the low confidence regime), while~\eqref{eq:tail-residual-conf} will be used in the analysis of the confidence-dependent estimators in the high-confidence regime.
Intuitively, the first term in~\eqref{eq:tail-residual-laplace} comes from the contribution of all classes, while the second term may come from single class $j$ with $p_j \asymp \log (1/\delta)/n$.

Observe that the residual $R_\lambda$ of Lemma~\ref{lem:tail-residual} is a sum of dependent random variables, since the counts $(N_j)_{1 \leq j \leq d}$ are dependent.
Up to a standard technique of Poisson sampling, one may reduce its control to that of a ``Poissonized'' sum involving independent summands, stated next:

\begin{lemma}
  \label{lem:tail-residual-poisson}
  Let $\wt N_1, \dots, \wt N_d$ be independent random variables, with $\wt N_j \sim \poissondist (\lambda_j/2)$ for $j = 1, \dots, d$.
  For any $\lambda \geq 1$, let $d_{\lambda} = | \set{1 \leq j \leq d : \lambda_j \geq 4 \lambda} |$ and
  \begin{equation*}
    \wt R_\lambda
    = \sum_{j \pp \lambda_j \geq 4 \lambda} \lambda_j \log \Big( \frac{2 \lambda_j}{\lambda} \Big) \bm 1 \Big( \wt N_j \leq \frac{\lambda_j}{4} \Big)
    \, .
  \end{equation*}
  For any $\delta \in (0, e^{-2})$,
  one has  
  \begin{equation}
    \label{eq:tail-residual-laplace-poisson}
    \P \Big( \wt R_1
    \geq 62000 \, d_1 + 106000 \log (1/\delta) \log \log (1/\delta)
    \Big)
    \leq \delta
    \, .
  \end{equation}
  In addition, for any $\delta \in (0, e^{-d}]$, if $\lambda = \log (1/\delta)/d$, then 
  \begin{equation}
    \label{eq:tail-residual-conf-poisson}
    \P \Big( \wt R_\lambda
    \geq %
    74000 \log (d)
    \log (1/\delta) 
    \Big)
    \leq \delta
    \, .
  \end{equation}
\end{lemma}

\begin{proof}[Proof of Lemma~\ref{lem:tail-residual} from Lemma~\ref{lem:tail-residual-poisson}]
  We resort to the technique of \emph{Poisson sampling}.
  Specifically, let $(X_i)_{i \geq n+1}$ be an \iid sequence of random variables with distribution $P$, independent from $X_1, \dots, X_n$.
In addition, let $N \sim \poissondist (n/2)$ be independent from the sequence $(X_i)_{i \geq 1}$.
For $j = 1, \dots, d$, define
\begin{equation}
  \label{eq:poisson-count}
  \wt N_j
  = \sum_{1 \leq i \leq N} \indic{X_i = j}
  \, .
\end{equation}
It is a classical fact about Poisson random variables (which follows, \eg, from the combination of~\cite[Theorem~5.6 p.~100]{mitzenmacher2017probability} and~\cite[Exercise~2.1.11 p.~55]{durrett2010probability}) that $\wt N_1, \dots, \wt N_d$ are independent random variables, with $\wt N_j \sim \poissondist (n p_j/2)$ for $j=1, \dots, d$.
Consider now the event $E = \{ N \leq n \}$; by the Poisson deviation bound (Lemma~\ref{lem:poisson-tail}), one has
\begin{equation}
  \label{eq:proba-poisson-sample}
  \P (E)
  = 1 - \P (N > n)
  \geq 1 - e^{- D (n, n/2)} %
  \geq 1 - e^{-n/6}
  \geq 1-\delta
  \, .
\end{equation}
In addition, under $E$ one has $\wt N_j \leq N_j$ for $j = 1, \dots, d$, hence letting $\lambda_j = n p_j$, %
\begin{equation}
  \label{eq:domination-poisson-residual}
  R_\lambda =
  \sum_{j \pp p_j \geq 4 \lambda/n} p_j \log \Big( \frac{2 n p_j}{\lambda} \Big) \bm 1 \Big( N_j \leq \frac{n p_j}{4} \Big)  
  \leq \frac{1}{n} \sum_{j \pp \lambda_j \geq 4 \lambda} \lambda_j \log \Big( \frac{2 \lambda_j}{\lambda} \Big) \bm 1 \Big( \wt N_j \leq \frac{\lambda_j}{4} \Big)
  \, .
\end{equation}
The right-hand side of~\eqref{eq:domination-poisson-residual} is controlled in Lemma~\ref{lem:tail-residual-poisson}, which concludes the proof.
\end{proof}

We now turn to the proof of Lemma~\ref{lem:tail-residual-poisson}.
It should be noted that %
$\wt R_\lambda$
is a nonnegative weighted sum of independent Bernoulli variables, with varying coefficient and parameters.
Perhaps a natural approach to control such a sum is to apply Bennett's inequality~\cite[Theorem~2.9 p.~35]{boucheron2013concentration}.
Unfortunately, this would lead to a highly suboptimal tail bound.
Roughly speaking, the reason why Bennett's inequality fails to capture the right tail behavior of this sum is that it is highly inhomogeneous, in the sense that the coefficients of the independent Bernoulli variables may be of different orders of magnitude.
In order to handle this structure, we
instead evaluate the upper envelope of the tails of the individual summands, and then
resort to a sharp estimate from Lata{\l}a~\cite{latala1997estimation} on moments of sums of independent random variables.

We start with the control on the tails of the individual summands that comprise $\wt R_\lambda$:

\begin{lemma}
  \label{lem:upper-envelope}
  For any $\lambda \geq 1$, let $W^{(\lambda)}$ be a random variable such that $\P (W^{(\lambda)} = 0) = 1- e^{-2\lambda/7}$, and $\P (W^{(\lambda)} \geq t \log (t/\lambda)) = e^{- t/14}$ for any $t\geq 4 \lambda$.
  Then, for any $j = 1, \dots, d$ such that $\lambda_j \geq 4 \lambda$, the random variable
  \begin{equation*}
    V_j^{(\lambda)}
    = \lambda_j \log \Big( \frac{\lambda_j}{\lambda} \Big) \ind \Big( \wt N_j \leq \frac{\lambda_j}{4} \Big)
  \end{equation*}
  is stochastically dominated by $W^{(\lambda)}$, in the sense that $\P (V_j^{(\lambda)} \geq w) \leq \P (W^{(\lambda)} \geq w)$ for any $w \in \R$.
\end{lemma}

In particular, for $\lambda = 1$, Lemma~\ref{lem:upper-envelope} asserts that for any $\delta < e^{-2/7}$,
the quantile of order $1-\delta$ of $V_j^{(1)}$ is smaller than $14 \log (1/\delta) \log (14 \log (1/\delta)) \asymp \log (1/\delta) \log \log (1/\delta)$, \emph{regardless of the value of $\lambda_j \geq 4$}.
By independence of the variables $V_j^{(\lambda)}$ and by Lemma~\ref{lem:stoch-domin}, we deduce from Lemma~\ref{lem:upper-envelope} that the random variable $\wt R_\lambda$ is stochastically dominated by $\sum_{j=1}^{d_\lambda} W_j^{(\lambda)}$, where $W_1^{(\lambda)}, W_2^{(\lambda)}, \dots$ are \iid random variables with the same distribution as $W^{(\lambda)}$.

\begin{proof}[Proof of Lemma~\ref{lem:upper-envelope}]
  First, since $\wt N_j \sim \poissondist (\lambda_j/2)$, by the Poisson deviation bound (Lemma~\ref{lem:poisson-tail}) we have
  \begin{equation}
    \label{eq:proof-envelope-proba-0}
    \P \Big( \wt N_j \leq \frac{\lambda_j}{4} \Big)
    \leq \exp \Big( - D \Big( \frac{\lambda_j}{4}, \frac{\lambda_j}{2} \Big) \Big)
    = \exp \big( - (1- \log 2) \lambda_j/4 \big)
    \leq e^{- \lambda_j/14}
    \, .
  \end{equation}
  We need to show that $\P ( V_j^{(\lambda)} \geq w) \leq \P (W^{(\lambda)} \geq w)$ for any $w \in \R$.
  For $w \leq 0$, both probabilities are equal to $1$.
  For $0 < w \leq 4 \log(4) \lambda$, this is a consequence of~\eqref{eq:proof-envelope-proba-0} as $\P (V_j^{(\lambda)} \geq w) \leq \P (\wt N_j \leq \lambda_j/4) \leq e^{-\lambda_j/14} \leq e^{-4 \lambda/14} = e^{-2\lambda/7} = \P (W^{(\lambda)} \geq w)$.

  For $w > 4 \log (4) \lambda$, using that the map $t \mapsto t \log (t/\lambda)$ is an increasing bijection from $(4\lambda, + \infty)$ to $(4 \log (4) \lambda, + \infty)$, we may write $w = t \log (t/\lambda)$ for some $t> 4\lambda$.
  There are now two cases.
  First, if $\lambda_j \geq t$, then by~\eqref{eq:proof-envelope-proba-0} $\P (V_j^{(\lambda)} \geq w) \leq \P (\wt N_j \leq \lambda_j/4) \leq e^{-\lambda_j/14} \leq e^{-t/14} = \P (W^{(\lambda)} \geq w)$.
  On the other hand, if $\lambda_j < t$, then $V_j^{(\lambda)} \leq \lambda_j \log (\lambda_j/\lambda) < t \log (t/\lambda)$, thus $\P (V_j^{(\lambda)} \geq w) = 0$ and the inequality also holds.
\end{proof}

We now aim to obtain a high-probability bound on $\sum_{j=1}^{d_\lambda} W_j^{(\lambda)}$, which is a sum of \iid nonnegative random variables.
Unfortunately, since $W^{(\lambda)}$ has super-exponential tails, its moment generating functions is infinite at any positive value, thus one cannot rely on the Chernoff approach to obtain such a tail bound.
We will instead work with its moments, using a moment estimate of Lata{\l}a~\cite{latala1997estimation}.
Specifically, Lemma~\ref{lem:latala} below is a consequence of (part of) \cite[Corollary~1]{latala1997estimation} obtained by tracking the numerical constants in this result.
We recall that, for any $p \in \R^+$ such that $p \geq 1$ and any real random variable $Z$, we let $\norm{Z}_p = \E [|Z|^p]^{1/p} \in \R^+ \cup \{ +\infty\}$.

\begin{lemma}[\cite{latala1997estimation}, Corollary~1]
  \label{lem:latala}
  Let $Z, Z_{1}, \dots, Z_{m}$ be \iid nonnegative random variables.
  Then, for any $p \in [1, + \infty)$,
  \begin{equation*}
    \bigg\|\sum_{i=1}^{m}Z_{i} \bigg\|_{p} 
    \leq 2e^{2} \sup%
    \Big\{ \frac{p}{s}\Big(\frac{m}{p}\Big)^{1/s}  \|Z\|_{s} 
    : \max \Big( 1, \frac{p}{m} \Big) \leq s \leq p
    \Big\} \, . 
  \end{equation*}
\end{lemma}

In order to apply Lemma~\ref{lem:latala}, we need to control the $L^p$ norm of $W^{(\lambda)}$.
This is achieved in the following lemma.

\begin{lemma}
  \label{lem:moment-w-lambda}
  For any $p \in [1, + \infty)$, one has
  \begin{equation}
    \label{eq:moment-w-lambda}
    \big\| W^{(\lambda)} \big\|_p
    \leq 215 p \log \Big( \max \Big\{ e, \frac{50 p}{\lambda} \Big\} \Big)
    \, .
  \end{equation}
\end{lemma}

\begin{proof}[Proof of Lemma~\ref{lem:moment-w-lambda}]
  We have
  \begin{align}
    \label{eq:proof-moment-w}
    \big\| W^{(\lambda)} \big\|_p^p
    &= \E \big[ (W^{(\lambda)})^p \big]
      = \int_{0}^\infty \P \big( (W^{(\lambda)})^p \geq u \big) \di u \nonumber \\
    &= \int_0^{(4 \log (4) \lambda)^p} e^{-2\lambda/7} \di u + \int_{4\log(4) \lambda}^{\infty} \P \big( W^{(\lambda)}\geq w \big) p w^{p-1} \di w  \nonumber \\
    &= (4 \log (4) \lambda)^p e^{-2\lambda/7} + \int_{4 \lambda}^{\infty} \P \Big( W^{(\lambda)} \geq t \log \Big( \frac{t}{\lambda} \Big) \Big) p \Big( t \log \Big( \frac{t}{\lambda} \Big) \Big)^{p-1} \log \Big( \frac{e t}{\lambda} \Big) \di t \nonumber \\
    &\leq (4 \log (4) \lambda)^p e^{-2\lambda/7} + \frac{p}{2\lambda} \int_{4 \lambda}^{\infty} e^{-t/14} \Big( t \log \Big( \frac{t}{\lambda} \Big) \Big)^{p} \di t
      \, ,
  \end{align}
  where we used that $\log (e t/\lambda) \leq 2 \log (t/\lambda) \leq t \log (t/\lambda)/(2\lambda)$.
  Now, for any $t \geq 4 \lambda$, let
  \begin{equation*}
    \phi_\lambda (t)
    = \log \Big\{ e^{-t/28} \Big( t \log \Big( \frac{t}{\lambda} \Big) \Big)^p \Big\}
    = p \log (t) + p \log \log \Big( \frac{t}{\lambda} \Big) - \frac{t}{28}
    \, .
  \end{equation*}
  Since
  \begin{equation*}
    \phi'_\lambda (t)
    = \frac{p}{t} + \frac{p}{t \log (t/\lambda)} - \frac{1}{28}
    \leq \Big( 1 + \frac{1}{\log 4} \Big) \frac{p}{t} - \frac{1}{28}
    \, ,
  \end{equation*}
  we deduce that $\phi'_\lambda (t) < 0$ for any $t \geq 50 p$, thus $\phi_\lambda$ decreases over $[\max\{50 p, 4\lambda\}, + \infty)$.

  \paragraph{Small $p$.}
  We first consider the case where $50 p \leq 4 \lambda$.
  In this case, we get
  \begin{equation*}
    \sup_{t \geq 4 \lambda} \Big\{ e^{-t/28} \Big( t \log \Big( \frac{t}{\lambda} \Big) \Big)^p \Big\}
    = e^{-\lambda/7} \big( 4 \log (4) \lambda \big)^p
    \, ,
  \end{equation*}
  and therefore
  \begin{equation*}
    \frac{p}{2\lambda} \int_{4 \lambda}^{\infty} e^{-t/14} \Big( t \log \Big( \frac{t}{\lambda} \Big) \Big)^{p} \di t
    \leq \frac{1}{25} e^{-\lambda/7} \big( 4 \log (4) \lambda \big)^p \int_{4 \lambda}^{\infty} e^{-t/28} \di t
    = \frac{28}{25} \big( 4 \log (4) \lambda \big)^p e^{-2\lambda/7}
    \, .
  \end{equation*}
  Plugging this inequality into~\eqref{eq:proof-moment-w}, we deduce that
  \begin{equation*}
    \big\| W^{(\lambda)} \big\|_p^p
    \leq 2.2 \big( 4 \log(4) \lambda \big)^p e^{-2\lambda/7}
    \, .
  \end{equation*}
  We now apply the bound $(e x/p)^p \leq e^x$ for $x \in \R^+$ to $x = \lambda/7$, which gives
  \begin{equation*}
    \big\| W^{(\lambda)} \big\|_p^p
    \leq 2.2 \big( 4 \log (4) \big)^p \Big( \frac{7 p}{e} \Big)^p e^{-\lambda/7} %
    \leq 2.2 (15 p)^p e^{-12.5 p/7}
    \leq 2.2 (2.4 p)^p
    \, ,
  \end{equation*}
  so that $\norm{W^{(\lambda)}}_p \leq (2.2)^{1/p} \cdot 2.4 p \leq 6p$.

  \paragraph{Large $p$.}

  We now turn to the case where $4 \lambda \leq 50p$.
  In this case, we get
  \begin{align*}
    &\sup_{t \geq 4 \lambda} \Big\{ e^{-t/28} \Big( t \log \Big( \frac{t}{\lambda} \Big) \Big)^p \Big\}
    = \sup_{4 \lambda \leq t \leq 50 p} \Big\{ e^{-t/28} \Big( t \log \Big( \frac{t}{\lambda} \Big) \Big)^p \Big\} \\
    &\leq \sup_{t \in \R^+} \Big\{ e^{-t/28} t^p \Big\} \Big( \log \Big( \frac{50 p}{\lambda} \Big) \Big)^p
      = e^{-p} (28 p)^p \Big( \log \Big( \frac{50 p}{\lambda} \Big) \Big)^p
      \, .
  \end{align*}
  This implies that
  \begin{align*}
    \frac{p}{2\lambda} \int_{4 \lambda}^{\infty} e^{-t/14} \Big( t \log \Big( \frac{t}{\lambda} \Big) \Big)^{p} \di t
    &\leq \frac{p}{2} \Big( 28 e^{-1}p \log \Big( \frac{50 p}{\lambda} \Big) \Big)^p \int_{4 \lambda}^{\infty} e^{-t/28} \di t \\
    &\leq 14 p \Big( 28 e^{-1} p \log \Big( \frac{50 p}{\lambda} \Big) \Big)^p
      \, .
  \end{align*}
  Plugging this inequality into~\eqref{eq:proof-moment-w} and using the bound $( 4 \log(4) \lambda )^p e^{-2\lambda/7}  \leq (2.4 p)^p$ shown above, we get
  \begin{align*}
    \big\| W^{(\lambda)} \big\|_p
    &\leq \Big[ (2.4 p)^p + 14 p \Big( 28 e^{-1} p \log \Big( \frac{50 p}{\lambda} \Big) \Big)^p \Big]^{1/p} \\
    &\leq 2.4 p + 14 p^{1/p} \cdot 28 e^{-1} p \log \Big( \frac{50 p}{\lambda} \Big) \\
    &\leq 2.4 p + 210 p \log \Big( \frac{50 p}{\lambda} \Big)
      \leq 215 \log \Big( \frac{50p}{\lambda} \Big)
    \, .
  \end{align*}
  Hence, the desired bound holds for all $p \geq 1$.
\end{proof}

Combining the moment estimate of Lemma~\ref{lem:moment-w-lambda} with Lemma~\ref{lem:latala}, we obtain the following control on the moments of the term that dominates residuals when $\lambda = 1$.

\begin{lemma}
  \label{lem:moment-sum-w-laplace}
  For any $p \geq 1$, one has
  \begin{equation}
    \label{eq:moment-sum-w-laplace}
    \bigg\| \sum_{j=1}^{d_1} W_j^{(1)} \bigg\|_p
    \leq 15000 d_1 + 4600 p \log (50 p)
    \, .
  \end{equation}
\end{lemma}

\begin{proof}[Proof of Lemma~\ref{lem:moment-sum-w-laplace}]
  We start with the case where $p \geq d_1$.
  In this case, plugging Lemma~\ref{lem:moment-w-lambda} into Lemma~\ref{lem:latala} and bounding $d_1/p \leq 1$ gives:
  \begin{align*}
    \bigg\| \sum_{j=1}^{d_1} W_j^{(1)} \bigg\|_p
    &\leq 2 e^2 \sup \Big\{ \frac{p}{s}\Big(\frac{d_1}{p}\Big)^{1/s}  %
    215 s \log \Big( \max \Big\{ e, {50 s} \Big\} \Big)
      : \max \Big( 1, \frac{p}{d_1} \Big) \leq s \leq p \Big\} \\
    &\leq 2 e^2 \sup_{1 \leq s \leq p} \Big\{ 215 p \log (50 s) \Big\}
      = 430 e^2 p \log \big( 50 p \big)
      \, .
  \end{align*}
  We now turn to the case where $p \leq d_1$.
  In this case, Lemmas~\ref{lem:latala} and~\ref{lem:moment-w-lambda} imply that
  \begin{equation*}
    \bigg\| \sum_{j=1}^{d_1} W_j^{(1)} \bigg\|_p
    \leq 430 e^2 \sup_{1 \leq s \leq p} \Big\{ p \Big( \frac{d_1}{p} \Big)^{1/s} \log (50 s) \Big\}
    \, .
  \end{equation*}
  We bound the supremum over $1 \leq s \leq p$ as follows.
  If $1 \leq s \leq 2$, then (using that $d_1/p \geq 1$)
  \begin{equation*}
    p \Big( \frac{d_1}{p} \Big)^{1/s} \log (50 s)
    \leq p \Big( \frac{d_1}{p} \Big)^{1/s} \log (100)
    \leq p \Big( \frac{d_1}{p} \Big) \log (100)
    = \log (100) d_1
    \, .
  \end{equation*}
  If $s \geq 2$, we consider two cases.
  If $s \leq \sqrt{d_1/p}$, then
  \begin{equation*}
    p \Big( \frac{d_1}{p} \Big)^{1/s} \log (50 s)
    \leq p \Big( \frac{d_1}{p} \Big)^{1/2} \log \bigg( 50 \sqrt{\frac{d_1}{p}} \bigg)
    = d_1 \frac{\log(50) + \log (\sqrt{d_1/p})}{\sqrt{d_1/p}}
    \leq ( \log 50 + e^{-1} ) d_1
    \, .
  \end{equation*}
  Finally, if $\sqrt{d_1/p} \leq s \leq p$, then
  \begin{equation*}
    p \Big( \frac{d_1}{p} \Big)^{1/s} \log (50 s)
    \leq p s^{1/s} \log (50 p)
    \leq e^{1/e} p \log (50 p)
    \, .
  \end{equation*}
  Putting things together, for any value of $p$ one has
  \begin{equation*}
    \bigg\| \sum_{j=1}^{d_1} W_j^{(1)} \bigg\|_p
    \leq 430 e^2 \max \big\{ \log (100) d_1, e^{1/e} p \log (50 p) \big\}
    \, ,
  \end{equation*}
  which proves~\eqref{eq:moment-sum-w-laplace} after bounding the maximum by a sum and simplifying constants.
\end{proof}

Next, we obtain similarly a control of suitable large moments of the sum.

\begin{lemma}
  \label{lem:moment-sum-w-conf}
  Assume that $p \geq d$ and that $\lambda = p/d$.
  Then
  \begin{equation}
    \label{eq:moment-sum-w-conf}
    \bigg\| \sum_{j=1}^{d_\lambda} W_j^{(\lambda)} \bigg\|_p
    \leq 3200 \log (50 d) p
    \, .
  \end{equation}
\end{lemma}

\begin{proof}[Proof of Lemma~\ref{lem:moment-sum-w-conf}]
  Since $p \geq d$, we have in particular that $d_\lambda / p \leq 1$.
  Plugging Lemma~\ref{lem:moment-w-lambda} into Lemma~\ref{lem:latala} therefore gives:
  \begin{align*}
    \bigg\| \sum_{j=1}^{d_\lambda} W_j^{(\lambda)} \bigg\|_p
    &\leq 2 e^2 \sup \Big\{ \frac{p}{s}\Big(\frac{d_\lambda}{p}\Big)^{1/s}  %
    215 s \log \Big( \max \Big\{ e, \frac{50 s}{\lambda} \Big\} \Big)
      : \max \Big( 1, \frac{p}{d_\lambda} \Big) \leq s \leq p \Big\} \\
    &\leq 2 e^2 \sup_{1 \leq s \leq p} \Big\{ 215 p \log \Big( \max \Big\{ e, \frac{50 s}{\lambda} \Big\} \Big) \Big\} \\
    &= 430 e^2 p \log \max \Big\{ e, \frac{50 p}{\lambda} \Big\}
      = 430 e^2 p \log \big( 50 d \big)
      \, .
  \end{align*}
  The conclusion follows by bounding $430 e^2 \leq 3200$.
\end{proof}

With these results in place, we can conclude the proof of Lemma~\ref{lem:tail-residual-poisson}---and thus, of Lemma~\ref{lem:tail-residual}.

\begin{proof}[Proof of Lemma~\ref{lem:tail-residual-poisson}]
  Denote by $\mleq$ the stochastic domination relation between real random variables (Definition~\ref{def:stoch-domin}).
  As noted above, by independence of $\wt N_1, \dots, \wt N_d$ it follows from Lemmas~\ref{lem:upper-envelope} and~\ref{lem:stoch-domin} that
  \begin{equation}
    \label{eq:stoch-domin-residual-sum}
    \wt R_\lambda
    \leq \frac{3}{2} \sum_{j \pp \lambda_j \geq 4 \lambda} \lambda_j \log \Big( \frac{\lambda_j}{\lambda} \Big) \bm 1 \Big( \wt N_j \leq \frac{\lambda_j}{4} \Big)
    \mleq \frac{3}{2} \sum_{j=1}^{d_\lambda} W_j^{(\lambda)}
    \, .
  \end{equation}
  (Above, we used that $\log (2 \lambda_j/\lambda) \leq \frac{3}{2} \log (\lambda_j/\lambda)$ for $\lambda_j/\lambda \geq 4$.)
  It therefore suffices to establish tail bounds on the right-hand side of~\eqref{eq:stoch-domin-residual-sum}.
  In both cases $\lambda = 1$ and $\lambda = \log (1/\delta)/d$,
  we deduce such tail bounds from the moment bounds of Lemmas~\ref{lem:moment-sum-w-laplace} and~\ref{lem:moment-sum-w-conf}, respectively,
  together with the following inequality:
  for any real random variable $Z$ and $p \geq 1$, 
  \begin{equation*}
    \P (|Z| \geq e \norm{Z}_p)
    = \P (|Z|^p \geq e^p \, \E [|Z|^p])
    \leq e^{-p}
    \, ,
  \end{equation*}
  applied to $p = \log (1/\delta)$.
  Inequalities~\eqref{eq:tail-residual-laplace-poisson} and~\eqref{eq:tail-residual-conf-poisson} are obtained by further bounding constants, using in particular that $\log (50) \leq \log_2(50) \min \{ \log(1/\delta), \log d \}$ as $\delta < e^{-2}$ and $d \geq 2$.
\end{proof}

In the following sections, we proceed with the proofs of Theorem~\ref{thm:upper-laplace},~\ref{thm:upper-bound-conf-dependent} and~\ref{thm:upper-bound-sparse}---the
first two directly obtained by combining the results above.

\subsection{Proof of Theorem~\ref{thm:upper-laplace}}
\label{sec:proof-upper-laplace}

  We apply the decomposition of Lemma~\ref{lem:decomp-risk} with $\lambda = 1$.
  The first term in this decomposition is bounded through the Hellinger bound of Lemma~\ref{lem:hellinger-empirical}, together with the inequality $s_n (P) \leq d$.
  The second term is equal to $7 d/n$.
  Finally, the third term is bounded through the bound~\eqref{eq:tail-residual-laplace} on $R_1$ from Lemma~\ref{lem:tail-residual}.
  Putting things together and using a union bound, we obtain, for any $\delta \in (e^{-n/6}, e^{-2})$,
  \begin{equation*}
    \P_P \bigg( \kll{P}{\wh P_n}
    \geq 6 \times \frac{4 d + 7 \log (1/\delta)}{n} + \frac{7 d}{n} + \frac{62000 d + 106000 \log (1/\delta) \log \log (1/\delta)}{n}
    \bigg)
    \leq 4 \delta
    \, .
  \end{equation*}
  Further bounding constants gives the bound of Theorem~\ref{thm:upper-laplace}.

\subsection{Proof of Theorem~\ref{thm:upper-bound-conf-dependent}}
\label{sec:proof-upper-conf-dependent}

  We consider two cases.
  If $\log (1/\delta) \leq d$, then $\lambda_\delta = 1$ and $\wh P_{n, \delta}$ coincides with the Laplace estimator.
  Then, Theorem~\ref{thm:upper-laplace} ensures that, with probability at least $4 \delta$,
  \begin{equation}
    \label{eq:confidence-smaller-laplace}
    \kll{P}{\wh P_{n,\delta}}
    \leq 110000 \frac{d + \log (1/\delta) \log \log (1/\delta)}{n}
    \leq 110000 \frac{d + \log (1/\delta) \log d}{n}
    \, .
  \end{equation}
  On the other hand, if $\log (1/\delta) > d$, namely $\delta \in (e^{-n/6}, e^{-d})$, then $\lambda_\delta = \log (1/\delta)/d > 1$.
  We then proceed similarly to the proof of Theorem~\ref{thm:upper-laplace}, with only two changes: the second term in the decomposition of Lemma~\ref{lem:decomp-risk} now equals $7\lambda_\delta d/n = 7\log (1/\delta)/n$, while the third term is now controlled using the bound~\eqref{eq:tail-residual-conf} on $R_{\lambda_\delta}$ from Lemma~\ref{lem:tail-residual}.
  This gives, with probability at least $1-4\delta$,
  \begin{equation*}
    \kll{P}{\wh P_{n, \delta}}
    < 6 \times \frac{4 d + 7 \log (1/\delta)}{n} + \frac{7 \log (1/\delta)}{n} + \frac{74000 \log (d) \log (1/\delta)}{n}
    \, ,
  \end{equation*}
  which also implies the desired tail bound.

\subsection{Proof of Theorem~\ref{thm:upper-bound-sparse}}
\label{sec:proof-upper-sparse-whp}

We are now in position to complete the proof of Theorem~\ref{thm:upper-bound-sparse}, up to Theorem~\ref{thm:deviation-missing-mass} which we prove in Section~\ref{sec:proof-deviation-missing-mass} below.

  In what follows, we fix $P \in \probas_d$ and let $s_n = s_n(P)$ and $s_n^\circ = s_n^\circ (P)$.
  For now, let $\wt \lambda$ be either $\wh \lambda$ or $\wh \lambda_\delta$, and let $\wt P$ be the add-$\wt \lambda$ estimator.
  First, the decomposition of Lemma~\ref{lem:decomp-risk} writes:
  \begin{equation}
    \label{eq:proof-sparse-decomp}
    \kll{P}{\wt P}
    \leq 6 \sum_{j=1}^d \Big(\sqrt{\ol p_j} - \sqrt{p_j}\Big)^2 + \frac{7 \wt \lambda d}{n} %
    + \sum_{j \pp p_j \geq 4 \wt \lambda/n} p_j \log \Big( \frac{2 n p_j}{\wt \lambda} \Big) \bm 1 \Big( N_j \leq \frac{n p_j}{4} \Big)
    \, .
  \end{equation}

  \paragraph{First term.}
  The first term in the decomposition~\eqref{eq:proof-sparse-decomp} is bounded through Lemma~\ref{lem:hellinger-empirical}: with probability at least $1-2\delta$,
  \begin{equation}
    \label{eq:proof-sparse-first-term}
    6 \sum_{j=1}^d \Big( \sqrt{\ol p_j} - \sqrt{p_j} \Big)^2
    < \frac{24 s_n + 42 \log (1/\delta)}{n}
    \, .
  \end{equation}

  \paragraph{Second term.}
  For the second term, note that $\wt \lambda \leq \wh \lambda_\delta \leq \max \set{D_n, \log (1/\delta)}/d$, and that by inequality~\eqref{eq:bernstein-occupancy} from Lemma~\ref{lem:deviation-occupancy}, with probability at least $1-\delta$,
  \begin{equation*}
    D_n
    \leq 2 \E_P [D_n] + 2 \log (1/\delta)
    \, .
  \end{equation*}
  Combining these inequalities and recalling that $\E_P [D_n] \leq s_n$ (Fact~\ref{fac:expected-support}), we get: with probability $1-\delta$,
  \begin{equation}
    \label{eq:proof-sparse-second-term}
    \frac{7 \wt \lambda d}{n}
    \leq \frac{14 s_n + 14 \log (1/\delta)}{n}
    \, .
  \end{equation}

  \paragraph{Third term.}

  We now turn to the control of the third term, which
  requires the most effort.

  We first deal with the case where $\wt \lambda > 1$, which directly reduces to
  Lemma~\ref{lem:tail-residual}.
  Indeed, since $\wh \lambda = D_n /d \leq 1$, one must have $\wt \lambda = \wt \lambda_\delta %
  = \log (1/\delta)/d$, and $\delta \leq e^{-d}$.
  Hence, 
  inequality~\eqref{eq:tail-residual-conf} from Lemma~\ref{lem:tail-residual} gives: with probability $1-2\delta$,
  \begin{equation}
    \label{eq:proof-sparse-third-large-lambda}
    \sum_{j \pp p_j \geq 4 \wt \lambda/n} p_j \log \Big( \frac{2 n p_j}{\wt \lambda} \Big) \bm 1 \Big( N_j \leq \frac{n p_j}{4} \Big)
    = R_{\lambda_\delta}
    \leq \frac{74000 \log (d) \log (1/\delta)}{n}
    \, .
  \end{equation}

  From now on, we assume that $\wt \lambda \leq 1$.
  In this case, we further decompose the third term (denoted $R_{\wt \lambda}$) as follows:
  \begin{equation}
    \label{eq:proof-sparse-dec-3}
    R_{\wt \lambda} %
    = \sum_{j \pp p_j \geq 4 \wt \lambda/n} p_j \log ( {2 n p_j} ) \bm 1 \Big( N_j \leq \frac{n p_j}{4} \Big)
      + \sum_{j \pp p_j \geq 4 \wt \lambda/n} p_j \bm 1 \Big( N_j \leq \frac{n p_j}{4} \Big) \cdot \log \bigg( \frac{1}{\wt \lambda} \bigg)
      \, .
  \end{equation}
  We then bound the first term above as
  \begin{align*}
    &\sum_{j \pp p_j \geq 4 \wt \lambda/n} p_j \log ( {2 n p_j} ) \bm 1 \Big( N_j \leq \frac{n p_j}{4} \Big) \\
    &= \sum_{j \pp 4 \wt \lambda/n \leq p_j < 4/n} p_j \log ( {2 n p_j} ) \bm 1 \Big( N_j \leq \frac{n p_j}{4} \Big) + \sum_{j \pp p_j \geq 4/n} p_j \log ( {2 n p_j} ) \bm 1 \Big( N_j \leq \frac{n p_j}{4} \Big) \\
    &\leq \log (8) \sum_{j \pp 4 \wt \lambda/n \leq p_j < 4/n} p_j \bm 1 \Big( N_j \leq \frac{n p_j}{4} \Big) + R_1 
    \leq \log (8) \sum_{j \pp p_j < 4/n} p_j %
      + R_1 \\
      &\leq \frac{4 \log (8) s_n}{n} + R_1
      \, .
  \end{align*}
  Hence, bounding $R_1$ via the bound~\eqref{eq:tail-residual-laplace} of Lemma~\ref{lem:tail-residual} and using that $d_1 = |\set{j : p_j \geq 4/n}| \leq s_n$, we obtain: with probability $1-2\delta$,
  \begin{align*}
    \sum_{j \pp p_j \geq 4 \wt \lambda/n} p_j \log ( {2 n p_j} ) \bm 1 \Big( N_j \leq \frac{n p_j}{4} \Big)
    &\leq \frac{4 \log (8) s_n}{n} + \frac{62000 \, s_n + 106000 \log (1/\delta) \log \log (1/\delta)}{n} \\
    &\leq \frac{63000 \, s_n + 106000 \log (1/\delta) \log \log (1/\delta)}{n}
      \, .
  \end{align*}

  It remains to upper bound the second term in~\eqref{eq:proof-sparse-dec-3}.
  Since $\wt \lambda \geq \wh \lambda = D_n/d$, and recalling the definition of the underestimated mass $U_n$ (Definition~\ref{def:missing-underestimated-mass}), we have
  \begin{equation*}
    \sum_{j \pp p_j \geq 4 \wt \lambda/n} p_j \bm 1 \Big( N_j \leq \frac{n p_j}{4} \Big) \cdot \log \bigg( \frac{1}{\wt \lambda} \bigg)
    \leq \sum_{j =1}^d p_j \bm 1 \Big( N_j \leq \frac{n p_j}{4} \Big) \cdot \log \bigg( \frac{d}{D_n} \bigg)
    = U_n \log \bigg( \frac{d}{D_n} \bigg)
    \, .
  \end{equation*}

  Now, Theorem~\ref{thm:deviation-missing-mass} shows that, with probability at least $1-8\delta$, one has
  \begin{equation*}
    U_n
    \leq \frac{336 \, s_{n/112}^\circ (P) + 2500 e \log (1/\delta)}{n}
    \, .
  \end{equation*}
  Thus, %
  under the same event (using that $D_n \geq 1$),
  \begin{equation}
    \label{eq:proof-sparse-u-log-d-k}
    U_n \log \bigg( \frac{d}{D_n} \bigg)
    \leq \frac{336 \, s_{n/112}^\circ (P) \log (d/D_n) + 2500 e \log (d) \log (1/\delta)}{n}
    \, .
  \end{equation}

  We now consider two cases, depending on whether $\log (1/\delta)$ is larger or smaller than $s_n$.
  First, by deviation lower bound~\eqref{eq:deviation-occupancy-lower} from Lemma~\ref{lem:deviation-occupancy}, letting $s_n' = \E_P [D_n]$ one has
  \begin{equation*}
    \P \Big( D_n \leq \frac{s_n'}{2} \Big)
    \leq \exp \Big\{ - D \Big( \frac{s_n'}{2}, s_n' \Big) \Big\}
    = \exp \Big\{ - \frac{1- \log 2}{2} s_n' \Big\}
    \, .
  \end{equation*}
  But since $s_n' \geq (1-e^{-1}) s_n$ (Fact~\ref{fac:expected-support}), after bounding constants we deduce that
  \begin{equation*}
    \P \big( D_n \leq 0.3 s_n \big)
    \leq e^{-s_n/21}
    \, .
  \end{equation*}
  Thus, if $\log (1/\delta) \leq s_n/21$, then with probability at least $1-\delta$ one has
  \begin{equation*}
    D_n
    \geq 0.3 s_n
    \, ,
  \end{equation*}
  which combined with~\eqref{eq:proof-sparse-u-log-d-k} gives, with probability $1-9\delta$,
  \begin{equation*}
    U_n \log \bigg( \frac{d}{D_n} \bigg)
    \leq \frac{336 \, s_{n/112}^\circ (P) \log (2 e d/s_{n}) + 2500 e \log (d) \log (1/\delta)}{n}
    \, .
  \end{equation*}
  On the other hand, if $\log (1/\delta) > s_n / 21$, then by lower-bounding $D_n \geq 1$ in~\eqref{eq:proof-sparse-u-log-d-k} and using that
  \begin{equation*}
    s_{n/112}^\circ
    \leq s_{n/112}
    \leq s_n
    \leq 21 \log (1/\delta)
    \, ,
  \end{equation*}
  we get with probability at least $1-\delta$ that
  \begin{equation*}
    U_n \log \bigg( \frac{d}{D_n} \bigg)
    \leq \frac{336 \times 21 \log (1/\delta) \times \log (d) + 2500 e \log (d) \log (1/\delta)}{n}
    \leq \frac{14000 \log (d) \log (1/\delta)}{n}
    \, .
  \end{equation*}
  Summarizing, we get that regardless of $\delta$, we have with probability at least $1-9\delta$ that
  \begin{equation*}
    \sum_{j \pp p_j \geq 4 \wt \lambda/n} p_j \bm 1 \Big( N_j \leq \frac{n p_j}{4} \Big) \cdot \log \bigg( \frac{1}{\wt \lambda} \bigg)
    \leq \frac{336 \, s_{n/112}^\circ (P) \log (2 e d/s_{n}) + 14000 \log (d) \log (1/\delta)}{n}
    \, .
  \end{equation*}
  Putting this inequality into the decomposition~\eqref{eq:proof-sparse-dec-3} of the third term, we get that whenever $\wt \lambda \leq 1$, we have with probability at least $1-11\delta$,
  \begin{align}
    \label{eq:proof-sparse-third-term}
    R_{\wt \lambda}
    &\leq \frac{63000 \, s_n + 106000 \log (1/\delta) \log \log (1/\delta)}{n} + \nonumber \\
    & \quad + \frac{336 \, s_{n/112}^\circ (P) \log (2 e d/s_{n}) + 14000 \log (d) \log (1/\delta)}{n} \nonumber \\
    &\leq \frac{64000 \, s_n + 336 \, s_{n/112}^\circ (P) \log (e d/s_{n}) + 120000 \max \{ \log d, \log \log (1/\delta) \} \log(1/\delta)}{n}
      \, .
  \end{align}

  \paragraph{Conclusion.}

  We first establish the bound~\eqref{eq:upper-bound-sparse} for the estimator $\adp$.
  In this case, one has $\wh \lambda = D_n/d \leq 1$, thus the bound~\eqref{eq:proof-sparse-third-term} applies.
  Injecting this inequality, together with the bounds~\eqref{eq:proof-sparse-first-term} and~\eqref{eq:proof-sparse-second-term} on the first two terms, into the decomposition~\eqref{eq:proof-sparse-decomp}, we obtain: with probability at least at $1-14 \delta$,
  \begin{align*}
    \kll{P}{\adp}
    &< \frac{24 s_n + 42 \log (1/\delta)}{n} + \frac{14 s_n + 14 \log (1/\delta)}{n} + \\
    &\ + \frac{64000 \, s_n + 336 \, s_{n/112}^\circ (P) \log (e d/s_{n}) + 120000 \max \{ \log d, \log \log (1/\delta) \} \log(1/\delta)}{n}
      ,
  \end{align*}
  which implies the claimed bound.

  We now conclude with the estimator $\adpdelta$.
  Then, either $\log (1/\delta)\leq d$, in which case $\wh \lambda_\delta \leq 1$ and again the bound on the third term for $\wt \lambda \leq 1$ applies, so that the same bound as for $\adp$ holds.
  In addition, one has $\max \{ \log d, \log \log (1/\delta) \} = \log d$ in this case.
  On the other hand, if $\log (1/\delta)> d$, then $\wh \lambda_\delta > 1$, thus the third term is bounded using~\eqref{eq:proof-sparse-third-large-lambda}.
  Combining with the bounds~\eqref{eq:proof-sparse-first-term} and~\eqref{eq:proof-sparse-second-term} on the first two terms gives the desired inequality.

\section{Proofs of lower bounds}
\label{sec:proof-lower-bounds}

In this section, we provide the proofs of the lower bounds stated in previous sections, specifically the tail (low-probability) lower bounds of Theorem~\ref{thm:lower-bound-conf-indep-tail}, Lemma~\ref{lem:lower-minimax-tail}, Theorem~\ref{thm:lower-bound-minimax} and Corollary~\ref{cor:minimax-sparse-lower}, as well as the high-probability minimax lower bound of Proposition~\ref{prop:lower-bound-minimax-sparse}.

\subsection{Proof of Theorem~\ref{thm:lower-bound-conf-indep-tail}}
\label{sec:proof-theorem}

We first prove Theorem~\ref{thm:lower-bound-conf-indep-tail} on confidence-independent estimators.
The proof rests on the following lemma.

\begin{lemma}
  \label{lem:lower-bound-conf-indep-tail}
  Let $n \geq d \geq 2$ and $\kappa \geq 1$, and $\wh P_n = \Phi (X_1, \dots, X_n)$ be an estimator as in Theorem~\ref{thm:lower-bound-conf-indep-tail}.
  Then, for any $\delta \in (e^{- n}, e^{-16 \kappa^2})$, there exists a distribution $P \in \probas_d$ such that
  \begin{equation}
    \label{eq:lower-bound-conf-indep-tail}
    \P_P \bigg( \kll{P}{\wh P_n} \geq \frac{\log (1/\delta) \log \log (1/\delta)}{10 n} \bigg)
    \geq \delta
    \, .
  \end{equation}
\end{lemma}

\begin{proof}[Proof of Lemma~\ref{lem:lower-bound-conf-indep-tail}]  
  Let $Q = \Phi (1, \dots, 1) = (q_1, \dots, q_d) \in \probas_d$ be the value of the estimator when only the first class is observed.
  Clearly, if $P = \delta_1$, then $\wh P_n = Q$ almost surely, and thus condition~\eqref{eq:assumption-in-expectation-opt} writes:
  \begin{equation*}
    \kll{\delta_1}{Q}
    = \log (1/q_1)
    \leq \frac{\kappa d}{n}
    \, .
  \end{equation*}
  Since $1 - q_1 \leq - \log (1 - (1-q_1)) = \log (1/q_1)$, this implies that
  \begin{equation*}
    \sum_{j=2}^d q_j
    = 1 - q_1
    \leq \frac{\kappa d}{n}
    \, ,
  \end{equation*}
  and thus there exists $j \in \set{2, \dots, d}$ such that $q_j \leq (\kappa d /n)/(d-1) \leq 2 \kappa/n$.
  Now, for $\delta \in (e^{-n}, e^{-16 \kappa^2})$, consider the distribution $P = P_{\delta, n} = (1 - \rho) \delta_1 + \rho \delta_j$, where $\rho = 1 - \delta^{1/n}$.
  Then, the event $E = \{ X_1 = \dots = X_n = 1 \}$ is such that 
  \begin{equation*}
    \P_P (E)
    = (1-\rho)^n
    = \delta
    \, .
  \end{equation*}
  In addition, under $E$ one has $\wh P_n = Q$, thus
  \begin{equation*}
    \kll{P}{\wh P_n}
    = \kll{P}{Q}
    \geq D (\rho, q_j)
    = q_j h \Big( \frac{\rho}{q_j} \Big)
    \, .
  \end{equation*}
  By convexity of the exponential function, one has $1- e^{-x} \geq (1 - e^{-1}) x$ for $x \in [0,1]$; since $\frac{\log (1/\delta)}{n} \leq 1$, this implies that
  $\rho = 1 - \exp \big(- \frac{\log (1/\delta)}{n} \big) \geq (1-e^{-1}) \frac{\log (1/\delta)}{n}$.
  Since $\delta \leq e^{-16 \kappa^2}$, we therefore have
  \begin{equation*}
    \frac{\rho}{q_j}
    \geq \frac{(1-e^{-1}) 16 \kappa^2 / n}{2 \kappa/n}
    = 8 (1 - e^{-1}) \kappa
    \geq e
    \, .
  \end{equation*}
  Hence, by Lemma~\ref{lem:entropy-function} we have
  \begin{align*}
    q_j h \Big( \frac{\rho}{q_j} \Big)
    &\geq q_j \times e^{-1} \frac{\rho}{q_j} \log \Big( \frac{\rho}{q_j} \Big)
    \geq e^{-1} (1-e^{-1}) \frac{\log (1/\delta)}{n} \log \Big( \frac{(1-e^{-1}) \log (1/\delta)/n}{2 \kappa/n} \Big) \\
    &\geq \frac{\log (1/\delta)}{5 n} \log \Big( \frac{\log (1/\delta)}{4 \kappa} \Big)
      \geq \frac{\log (1/\delta) \log \log (1/\delta)}{10 n}
      \, .
  \end{align*}
  This concludes the proof of Lemma~\ref{lem:lower-bound-conf-indep-tail}.
\end{proof}

We can now conclude the proof of Theorem~\ref{thm:lower-bound-conf-indep-tail}.

\begin{proof}[Proof of Theorem~\ref{thm:lower-bound-conf-indep-tail}]
  The statement follows from a combination of the lower bounds of Lemma~\ref{lem:lower-bound-conf-indep-tail} and of the consequence~\eqref{eq:dense-lower-high-prob} of Proposition~\ref{prop:lower-bound-minimax-sparse}.
  Specifically, one the one hand, since $d \geq 3300 \geq 3000 \log \big( \frac{3}{1-e^{-16}} \big)$, it follows from~\eqref{eq:dense-lower-high-prob} that there exists $P \in \probas_d$ such that
  \begin{equation*}
    \P_P \bigg( \kll{P}{\wh P_n} \geq \frac{d}{4600 n} \bigg)
    \geq 1 - 3 \exp \Big( - \frac{d}{3000} \Big)
    \geq e^{-16}
    \geq \delta
    \, .
  \end{equation*}
  On the other hand, by Lemma~\ref{lem:lower-bound-conf-indep-tail}, there exists $P \in \probas_d$ such that
  \begin{equation*}
    \P_P \bigg( \kll{P}{\wh P_n} \geq \frac{\log (1/\delta) \log \log (1/\delta)}{10 n} \bigg)
    \geq \delta
    \, .
  \end{equation*}
  By taking the best of these two lower bounds (depending on $d,\delta$), we deduce that there exists $P \in \probas_d$ such that, with probability at least $\delta$,
  \begin{align*}
    &\kll{P}{\wh P_n}
    \geq \max \bigg\{ \frac{d}{4600 n}, \frac{\log (1/\delta) \log \log (1/\delta)}{10 n} \bigg\} \\
    &\geq \frac{99}{100} \cdot \frac{d}{4600 n} + \frac{1}{100} \cdot \frac{\log (1/\delta) \log \log (1/\delta)}{10 n}
      \geq \frac{d + \log (1/\delta) \log \log (1/\delta)}{5000 n}
      \, ,
  \end{align*}
  which establishes the claim.
\end{proof}

\subsection{Proof of Lemma~\ref{lem:lower-minimax-tail} and Theorem~\ref{thm:lower-bound-minimax}}
\label{sec:proof-lemma-}

In this section, we establish Lemma~\ref{lem:lower-minimax-tail} and then deduce Theorem~\ref{thm:lower-bound-minimax}.

\begin{proof}[Proof of Lemma~\ref{lem:lower-minimax-tail}]
  Fix $n,d,\delta$ as in Lemma~\ref{lem:lower-minimax-tail}.
  We define the class $\F = \F_{n,d,\delta}$ as
  \begin{equation*}
    \F
    = \Big\{ P^{(j)} = \delta^{1/n} \delta_1 + \big( 1 - \delta^{1/n} \big) \delta_j : 1 \leq j \leq d \Big\}
    = \big\{ \delta_1 \big\} \cup \big\{ P^{(j)} : 2 \leq j \leq d \big\}
    \, .
  \end{equation*}
  Let $\Phi : [d]^n \to \probas_d$ be an estimator, and let $Q = (q_1, \dots, q_d) = \Phi (1, \dots, 1)$ denote the value of this estimator when only the first class is observed.
  Let $\alpha \in \R^+$ such that $1 - q_1 = \alpha d / n$.

  First, assume that $\alpha \geq \log (1/\delta)/(7\sqrt{d})$.
  In this case, if $P = P^{(1)} = \delta_1$, then $\wh P_n = Q$ almost surely, hence
  \begin{equation*}
    \kll{P}{\wh P_n}
    = \kll{\delta_1}{Q}
    = \log (1/q_1)
    \geq 1 - q_1
    = \frac{\alpha d}{n}
    \, .
  \end{equation*}
  Since $\alpha \geq \log (1/\delta)/(7\sqrt{d})$, we deduce that
  \begin{equation}
    \label{eq:lower-minimax-large-alpha}
    \kll{P}{\wh P_n}
    \geq \frac{\sqrt{d} \log (1/\delta)}{7 n}
    \geq \frac{\log (d) \log (1/\delta)}{14 n}
    \, .
  \end{equation}

  Now, assume that $\alpha < \log (1/\delta)/(7\sqrt{d})$.
  Since
  $\sum_{j=2}^d q_j = 1 - q_1 = \alpha d / n$, there exists $2 \leq j \leq d$ such that
  \begin{equation*}
    q_j
    \leq \frac{\alpha d}{n (d-1)}
    \leq \frac{2 \alpha}{n}
    < \frac{2 \log (1/\delta)}{7 n \sqrt{d}}
    \, .
  \end{equation*}
  Let $P = P^{(j)}$, so that letting $E = \{ X_1=1, \dots, X_n = 1 \}$, we have $\P_{P^{(j)}} (E) = (\delta^{1/n})^n = \delta$.
  Under $E$, one has $\wh P_n = Q$, thus denoting $\rho = 1 - \delta^{1/n}$ we have
  \begin{equation*}
    \kll{P}{\wh P_n}
    = \kll{P^{(j)}}{Q}
    \geq D (\rho, q_j)
    = q_j h \Big( \frac{\rho}{q_j} \Big)
    \, .
  \end{equation*}
  By convexity of the exponential function, one has $1- e^{-x} \geq (1 - e^{-1}) x$ for $x \in [0,1]$; since $\frac{\log (1/\delta)}{n} \leq 1$, this implies that
  $\rho = 1 - \exp \big(- \frac{\log (1/\delta)}{n} \big) \geq (1-e^{-1}) \frac{\log (1/\delta)}{n}$.
  We therefore have
  \begin{equation*}
    \frac{\rho}{q_j}
    \geq \frac{(1-e^{-1}) \log (\delta^{-1}) / n}{2 \log (\delta^{-1})/(7 n \sqrt{d})}
    = \frac{7 (1-e^{-1}) \sqrt{d}}{2}
    \geq 2 \sqrt{d}
    \geq e
    \, .
  \end{equation*}
  Hence, by Lemma~\ref{lem:entropy-function} we have
  \begin{align*}
    \kll{P}{\wh P_n}
    &\geq q_j h \Big( \frac{\rho}{q_j} \Big)
    \geq q_j \times e^{-1} \frac{\rho}{q_j} \log \Big( \frac{\rho}{q_j} \Big)
      \geq e^{-1} (1-e^{-1}) \frac{\log (1/\delta)}{n} \log (2 \sqrt{d}) \\
    &\geq \frac{\log (1/\delta) \log (\sqrt{d})}{5 n}
      = \frac{\log (d) \log (1/\delta)}{10 n}
      \, .
  \end{align*}
  This concludes the proof of Lemma~\ref{lem:lower-minimax-tail}.
\end{proof}

\begin{proof}[Proof of Theorem~\ref{thm:lower-bound-minimax}]
  We again use the consequence~\eqref{eq:dense-lower-high-prob} of Proposition~\ref{prop:lower-bound-minimax-sparse}, this time combined with Lemma~\ref{lem:lower-minimax-tail}.
  Specifically, since $d \geq 5000 \geq 3000 \log (\frac{3}{1-e^{-1}})$, it follows from~\eqref{eq:dense-lower-high-prob} that there exists $P \in \probas_d$ such that
  \begin{equation*}
    \P_P \bigg( \kll{P}{\wh P_n} \geq \frac{d}{4600 n} \bigg)
    \geq 1 - 3 \exp \Big( - \frac{d}{3000} \Big)
    \geq e^{-1}
    \geq \delta
    \, .
  \end{equation*}
  On the other hand, by Lemma~\ref{lem:lower-minimax-tail}, there exists $P \in \probas_d$ such that
  \begin{equation*}
    \P_P \bigg( \kll{P}{\wh P_n} \geq \frac{\log (d) \log (1/\delta)}{14 \, n} \bigg)
    \geq \delta
    \, .
  \end{equation*}
  As before, taking the best of these two lower bounds shows that there exists $P \in \probas_d$ under which, with probability at least $\delta$,
  \begin{equation*}
    \kll{P}{\wh P_n}
    \geq \frac{99}{100} \cdot \frac{d}{4600 n} + \frac{1}{100} \cdot \frac{\log (d) \log (1/\delta)}{14 \, n}
    \geq \frac{d + \log (d) \log (1/\delta)}{5000 n}
    \, .
  \end{equation*}
\end{proof}

\subsection{Proof of Proposition~\ref{prop:lower-bound-minimax-sparse}}
\label{sec:proof-lower-minimax-sparse}

In this section, we turn to the proof of the high-probability lower bound of Proposition~\ref{prop:lower-bound-minimax-sparse}.

  Note that if $s \leq 35$, then $1 - 3 e^{-s/35} < 0$ and the inequality is trivial.
  From now on, we therefore assume that $s \geq 36$.

  \paragraph{Random support.}
  Fix $n,d,s$ and an estimator $\wh P_n = \Phi (X_1, \dots, X_n)$.
  Let $\Subsets$ denote the class of subsets $\sigma$ of $\set{2, \dots, d}$ with cardinality $|\sigma| = s-1$.
  For any $\sigma \in \Subsets$, we let $P_\sigma = (p_1^\sigma, \dots, p_d^\sigma)$ be the element of $\probas_{s,d}$ defined by 
  \begin{equation}
    \label{eq:def-P-sigma}
    P_\sigma = \Big( 1 - \frac{s-1}{2 e n} \Big) \delta_1 + \frac{1}{2 e n} \sum_{j \in \sigma} \delta_j
    \, .
  \end{equation}
  In addition, we let $\pi$ denote the uniform distribution on $\Subsets$, which induces a ``prior'' distribution on $\psd$.
  We can then define a joint distribution for $(\sigma, (X_1, \dots, X_n))$ on $\Subsets \times [d]^n$ as follows: $\sigma \sim \pi$, and conditionally on $\sigma$, the variables $X_1, \dots, X_n$ are \iid with distribution $P_\sigma$.
  We denote by $\pi P$ the marginal distribution of $X_1, \dots, X_n$ under this distribution.
  As before, for $1 \leq j \leq d$ we denote by $N_j = \sum_{i=1}^n \indic{X_i = j}$ the number of occurrences of the class $j$.

  We start by establishing an upper bound on the number $D_n = \sum_{j=1}^d \indic{N_j \geq 1}$ of distinct classes.
  For any $\sigma \in \Subsets$, using that $\E [D_n | \sigma] = \E_{P_\sigma} [D_n] \leq 1 + (s-1) n /(2en) = 1 + (s-1)/(2e)$ (where we used that $\P_{P_\sigma} (N_j \geq 1) \leq \sum_{i=1}^n \P_{P_\sigma} (X_i= j)$) and applying Lemma~\ref{lem:deviation-occupancy}, we obtain that
  \begin{equation*}
    \P \bigg( D_n \geq e + \frac{s-1}{2} \bigg)
    = \P \bigg( D_n \geq e \cdot \bigg[ 1 + \frac{s-1}{2 e} \bigg] \bigg)
    \leq \exp \bigg( - 1 - \frac{s-1}{2 e} \bigg)
    \leq \exp \bigg( - \frac{s}{2 e} \bigg)    
    \, .
  \end{equation*}
  In what follows, we define the event $E = \{ D_n < e + (s-1)/2, \, N_j \geq 1 \}$, so that $\P (E^c) \leq \exp (-s/(2e)) + (\frac{s-1}{2 e n})^n \leq \exp (-s/(2e)) + (2e)^{-n}$.
  We also let $\wh \sigma = \{ 2 \leq j \leq d : N_j \geq 1 \}$, so that $|\wh \sigma| = D_n - 1$ under $E$.
  
  We now proceed to the lower bound on the estimation error.
  For any $t > 0$, we have
  \begin{equation}
    \label{eq:fubini-conditioning}
    \E_{\sigma \sim \pi} [ \P (\kll{P_\sigma}{\wh P_n} \geq t | \sigma) ]
    = \E [ \P (\kll{P_\sigma}{\wh P_n} \geq t | X_1, \dots, X_n) ]
      \, .
  \end{equation}
  We thus aim at establishing a lower bound of the following form, for some constant $C > 0$:
  \begin{equation*}
    \P \bigg( \kll{P_\sigma}{\wh P_n} \geq \frac{s \log (e d/s)}{C n} \bigg| X_1, \dots, X_n \bigg)
    \geq e^{- s/C}
    \, .
  \end{equation*}
  Below, we reason conditionally on $X_1, \dots, X_n$.
  First, note that the (posterior) conditional distribution of $\sigma$ given $X_1, \dots, X_n$, denoted $\wh \pi$, is the uniform distribution on the set
  \begin{equation*}
    \wh \Subsets
    = \set{\sigma \in \Subsets : \wh \sigma \subset \sigma}
    \, .
  \end{equation*}
  (Indeed, for any $\sigma \in \Subsets \setminus \wh \Subsets$, the distribution $P_\sigma^{\otimes n}$ puts a mass of $0$ to the sequence $(X_1, \dots, X_n)$; while all measures $P_\sigma^{\otimes n}$ with $\sigma \in \wh \Subsets$ put the same positive mass to such a sequence.)
  In other words, $\sigma = \wh \sigma \cup \wt \sigma$, where (conditionally on $X_1, \dots, X_n$) $\wt \sigma$ is uniformly distributed over all subsets of $\set{2, \dots, d} \setminus \wh \sigma$ with $s - D_n$ elements.
  Write $\wh P_n = (\wh p_1, \dots, \wh p_d)$, and 
  let $\wh \alpha = n (1-\wh p_1)$.

  \paragraph{Large $\wh \alpha$.}
  Assume first that $\wh \alpha \geq \sqrt{s d}/20$.
  In this case, applying Lemma~\ref{lem:kl-lower-subset} with $J = \{ 2, \dots, d \}$ gives, for any $\sigma \in \Subsets$,
  \begin{equation*}
    \kll{P_\sigma}{\wh P_n}
    \geq D \bigg( \sum_{j = 2}^d p_j^\sigma, \sum_{j=2}^d \wh p_j \bigg)
    = D \bigg( \frac{s-1}{2 e n}, \frac{\wh \alpha}{n} \bigg)
    = \frac{\wh \alpha}{n} \cdot h \Big( \frac{s-1}{2 e \wh \alpha} \Big)    
    \, .
  \end{equation*}
  Now, if $\wh \alpha \geq \sqrt{s d}/20$, we have
  $\frac{s-1}{2 e \wh \alpha} \leq 10 e^{-1}{\sqrt{s/d}} \leq 1/2$
  as $d/s \geq 55$,
  thus $h (\frac{s-1}{2 e \wh \alpha}) \geq h (1/2) = (1-\log 2)/2 \geq 1/7$.
  The previous inequality then writes, for any $\sigma \in \wh \Subsets$ (and thus with probability $1$ over $\sigma \sim \wh \pi$),
  \begin{equation}
    \label{eq:lower-large-alpha}
    \kll{P_\sigma}{\wh P_n}
    \geq \frac{1}{7} \frac{\sqrt{s d}}{20 n}
    = \frac{s}{140 n} \sqrt{\frac{d}{s}}
    \geq \frac{s}{140 n} \log \bigg( e \sqrt{\frac{d}{s}} \bigg)    
    \geq \frac{s \log (e d/s)}{280 n}
    \, .
  \end{equation}

  \paragraph{Small $\wh \alpha$.}
  Assume from now on that $\wh \alpha \leq \sqrt{s d}/20$.
  We start by noting that
  \begin{align}
    \label{eq:proof-lower-sparse-decomp-kl}
    \kll{P_\sigma}{\wh P_n}
    &= \sum_{j=1}^d D (p_j^\sigma, \wh p_j) 
    \geq \sum_{2 \leq j \leq d, \, j \not\in \wh \sigma} D \big( p_j^\sigma, \wh p_j \big) 
    \geq \sum_{2 \leq j \leq d, \, j \not\in \wh \sigma} D \Big( \frac{1}{2 e n}, \wh p_j \Big) \indic{j \in \wt \sigma}
      \, .
  \end{align}
  Observe that in the right-hand side of~\eqref{eq:proof-lower-sparse-decomp-kl}, conditionally on $X_1, \dots, X_n$, the only randomness comes from the presence of $\wt \sigma$.
  Now since $\sum_{j=2}^d \wh p_j = \wh \alpha/n \leq \sqrt{s d} / (20 n)$, we have:
  \begin{equation*}
    \Big| \Big\{ 2 \leq j \leq d : \wh p_j \geq \frac{\sqrt{s/d}}{10 n} \Big\} \Big|
    = \Big| \Big\{ 2 \leq j \leq d : \wh p_j \geq \frac{\sqrt{s d}/(20 n)}{d/2} \Big\} \Big|
    \leq d/2
    \, .
  \end{equation*}
  Recall that, under the event $E$, one has $|\wh \sigma| = D_n - 1 \leq e-1 + (s-1)/2$.
  It follows that, letting
  \begin{equation*}
    \wh J
    = \bigg\{ 2 \leq j \leq d : \wh p_j < \frac{\sqrt{s/d}}{10 n}, \ N_j = 0 \bigg\}
    \, ,
  \end{equation*}
  we have, recalling that
  $d \geq 55 s \geq 1980$,
  \begin{align}
    \label{eq:card-J-lower-bound-sparse}
    |\wh J|
    &\geq (d-1) - \frac{d}{2} - | \wh \sigma |
      \geq \frac{d}{2} - 1 - (e-1) - \frac{s-1}{2} %
      = \frac{d - s - (2e-1)}{2}
      \nonumber \\
    &\geq \frac{d}{2}
      \Big( 1 - \frac{1}{55} - \frac{2e-1}{1980} \Big)
      \geq 0.48 d
    \, .
  \end{align}
  In addition,
  since $D (\frac{1}{2 e n}, \cdot)$ decreases on $(0, \frac{1}{2 e n})$, we have for every $j \in \wh J$:
  \begin{align*}
    D \Big( \frac{1}{2 e n}, \wh p_j \Big)
    &\geq D \Big( \frac{1}{2 e n}, \frac{\sqrt{s/d}}{10 n} \Big)
      = \frac{\sqrt{s/d}}{10 n} \cdot h \Big( 5 e^{-1} \sqrt{d/s} \Big)
    \\
    &\geq \frac{\sqrt{s/d}}{10 n} \cdot e^{-1} 5 e^{-1} {\sqrt{d/s}} \log \Big( 5 e^{-1}{\sqrt{d/s}} \Big)
      \geq \frac{1}{4 e^2 n} \log \bigg( \frac{e d}{s} \bigg)
  \end{align*}
  where we used that $5 e^{-1} \sqrt{d/s}
  \geq e$ as $d \geq 55 s %
  $, %
  and that $h (t) \geq e^{-1} t \log t$ for $t \geq e$ (Lemma~\ref{lem:entropy-function}).
  Plugging this lower bound into~\eqref{eq:proof-lower-sparse-decomp-kl}, we obtain
  \begin{equation}
    \label{eq:kl-card-intersect}
    \kll{P_\sigma}{\wh P_n}
    \geq \frac{1}{4 e^2 n} \log \bigg( \frac{e d}{s} \bigg) \sum_{j \in \wh J} \indic{j \in \wt \sigma}
    = \frac{1}{4 e^2 n} \log \bigg( \frac{e d}{s} \bigg) |\wh J \cap \wt \sigma |    
    \, .
  \end{equation}
  We will now show that $| \wh J \cap \wt \sigma | \gtrsim s$ with high probability over the draw of $\wt \sigma$, conditionally on $X_1, \dots, X_n$.
  Although one could in principle show this by purely combinatorial means, we will instead resort to the notion of negative association~\cite{joagdev1983negative}, which provides a particularly convenient way to handle the dependence that arises here.

  Denote $\wh \sigma^c = \{ 2, \dots, d \} \setminus \wh \sigma$, so that $|\wh \sigma^c | = d - D_n$ under $E$.
  Since $\wt \sigma$ is uniformly distributed on subsets of $\wh \sigma^c$ with $s - D_n$ elements, conditionally on $X_1, \dots, X_n$ the vector $( \indic{j \in \wt \sigma})_{j \in \wh \sigma^c}$ is uniformly distributed over all permutations of the vector $( \indic{j \in \wt \sigma_0})_{j \in \wh \sigma^c}$ for some fixed $\wt \sigma_0 \subset \wh \sigma^c$ with $s-D_n$ elements.
  In other words, this distribution is a permutation distribution, which by~\cite[Theorem~2.11]{joagdev1983negative} is negatively associated, in the sense of~\cite[Definition~2.1]{joagdev1983negative}.
  By the restriction property of negatively associated random variables~\cite[Property~P4]{joagdev1983negative}, this implies that the variables $( \indic{j \in \wt \sigma})_{j \in \wh J}$ are negatively associated.
  This implies in particular that, for every $\lambda \in \R^+$,
  \begin{align*}
    &\E \big[ \exp \big( - \lambda | \wh J \cap \wt \sigma | \big) \big| X_1, \dots, X_n \big]
    = \E \bigg[ \prod_{j \in \wh J} \exp \big( - \lambda \indic{j \in \wt \sigma} \big) \bigg| X_1, \dots, X_n \bigg] \\
    &\leq \prod_{j \in \wh J} \E \big[  \exp \big( - \lambda \indic{j \in \wt \sigma} \big) \big| X_1, \dots, X_n \big]
      = \bigg( \frac{s - D_n}{d - D_n} e^{-\lambda} + \frac{d - s}{d - D_n} \bigg)^{|\wh J|}
      \, .
  \end{align*}
  Denoting $\wh \eta = (s-D_n)/(d-D_n) \in [0,1]$ and recalling that $|\wh J| \geq 0.48 d$ by~\eqref{eq:card-J-lower-bound-sparse}, we deduce that, for every $\lambda \in \R^+$ (using that $1 - \wh \eta + \wh \eta e^{-\lambda} \in [0,1]$),
  \begin{equation*}
    \log \E \big[ \exp \big( - \lambda | \wh J \cap \wt \sigma | \big) \big| X_1, \dots, X_n \big]
    \leq {0.48 d} \log \big( 1 - \wh \eta + \wh \eta e^{-\lambda} \big)
    \leq 0.48 \, d\, \wh \eta\, (e^{-\lambda} - 1)    
    \, .
  \end{equation*}
  Now under $E$ and recalling that $s \geq 36$, %
  we have
  \begin{equation*}
    d \, \wh \eta
    = d \cdot \frac{s-D_n}{d-D_n}
    \geq s - \frac{s + 2 e -1}{2}
    = \frac{s - (2e - 1)}{2}
    \geq 0.43 s
    \, ,
  \end{equation*}
  so that $\log \E [ \exp ( - \lambda |\wh J \cap \wt \sigma|)] \leq
  0.2 s
  (e^{-\lambda} - 1)$.  
  By a standard argument~\cite[p.~23]{boucheron2013concentration}, this implies the following tail bound: for every $s' \in (0, {s}/{10})$,
  \begin{equation*}
    \P \big( | \wh J \cap \wt \sigma | \leq s' \big| X_1, \dots, X_n \big)
    \leq \exp \big( - D (s', s/5) \big)
    \, .
  \end{equation*}
  In particular,
  \begin{equation*}
    \P \Big( | \wh J \cap \wt \sigma | \leq \frac{s}{10} \, \Big| \, X_1, \dots, X_n \Big)
    \leq \exp \Big( - \frac{1 - \log 2}{10} s \Big)
    \leq e^{-s/35}
    \, .
  \end{equation*}
  Plugging this into the lower bound~\eqref{eq:kl-card-intersect} shows that, under the event $E$,
  \begin{equation}
    \label{eq:sparse-conditional-lower}
    \P \bigg( \kll{P_\sigma}{\wh P_n} \geq \frac{s}{40 e^2 n} \log \bigg( \frac{e d}{s} \bigg) \,\bigg|\, X_1, \dots, X_n \bigg)
    \geq 1 - e^{-s/35}
    \, .
  \end{equation}
  Note that this lower bound also holds in the case where $\wh \alpha \geq \sqrt{s d}/20$ due to~\eqref{eq:lower-large-alpha}.

  \paragraph{Conclusion of the proof.}

  Using the identity~\eqref{eq:fubini-conditioning} together with the conditional lower bound~\eqref{eq:sparse-conditional-lower} under $E$, as well as the bound $\P (E^c) \leq e^{-s/2e} + (2e)^{-n} \leq 2 e^{-s/2e}$ established above, we get
  \begin{align}
    \label{eq:lower-average-sigma}
    &\E_{\sigma \sim \pi} \bigg[ \P \bigg(\kll{P_\sigma}{\wh P_n} \geq \frac{s \log (e d/s)}{40 e^2 n} \, \bigg| \, \sigma \bigg) \bigg] \nonumber \\
    &\geq \E \bigg[ \P \bigg(\kll{P_\sigma}{\wh P_n} \geq \frac{s \log (e d/s)}{40 e^2 n} \, \bigg| \, X_1, \dots, X_n \bigg) \bm 1_E \bigg] \nonumber \\
    &\geq \P (E) \cdot \big( 1 - e^{-s/35} \big)
      \geq \big( 1 - 2 e^{-s/2e} \big) \big( 1 - e^{-s/35} \big) \nonumber \\
    &\geq 1 - 3 e^{-s/35}
      \, .
  \end{align}
  Since $\max_{\sigma \in \Subsets} \set{\cdots} \geq \E_{\sigma \sim \pi} [\cdots]$, it follows from~\eqref{eq:lower-average-sigma} that there exists $\sigma \in \Subsets$ such that
  \begin{equation*}
    \P_{P_\sigma} \bigg( \kll{P_\sigma}{\wh P_n} \geq \frac{s \log (e d/s)}{40 e^2 n} \bigg)
    \geq 1 - 3 e^{-s/35}
    \, .
  \end{equation*}
  The lower bound of Proposition~\ref{prop:lower-bound-minimax-sparse} follows by further bounding $40 e^2 \leq 300$.

\subsection{Proof of Corollary~\ref{cor:minimax-sparse-lower}}
\label{sec:proof-lower-minimax-sparse-tail}

  Fix $\Phi = \Phi_{s,\delta}$.  
  By Lemma~\ref{lem:lower-minimax-tail}, 
  there exists $P \in \F \subset \probas_{2,d} \subset \probas_{s,d}$ such that
  \begin{equation}
    \label{eq:proof-corollary-sparse-dev}
    \P_P \bigg( \kll{P}{\wh P_n} \geq \frac{\log (d) \log (1/\delta)}{14 \,n} \bigg)
    \geq \delta
    \, .
  \end{equation}
  On the other hand, by Proposition~\ref{prop:lower-bound-minimax-sparse}, there exists a distribution $P \in \probas_{s,d}$ such that
  \begin{equation*}
    \P_P \bigg( \kll{P}{\wh P_n} \geq \frac{s \log (e d/s)}{300 \, n} \bigg)
    \geq 1 - 3 e^{-s/35}
    \, .
  \end{equation*}
  Now, if $s \geq 42$, then $1-3e^{-s/35} > e^{-2} \geq \delta$; on the other hand, if $2 \leq s \leq 41$ then (using that $d \geq 55 s \geq 41$)
  \begin{equation*}
    \frac{\log (d) \log(1/\delta)}{14 n}
    \geq \frac{3 \log (d)}{14 n}
    \geq \frac{2 \cdot 41 \log (e d/41)}{14\cdot 41 n}
    \geq \frac{2 \cdot s \log (e d/s)}{574 n}
    \geq \frac{s \log (e d/s)}{300 n}
    \, .
  \end{equation*}
  Hence, using the previous inequalities, we deduce that regardless of $s \geq 2$, there exists $P \in \probas_{s,d}$ such that
  \begin{equation}
    \label{eq:proof-sparse-complexity}
    \P_P \bigg( \kll{P}{\wh P_n} \geq \frac{s \log (e d/s)}{300 n} \bigg)
    \geq \delta
    \, .
  \end{equation}
  Taking the best of the two lower bounds~\eqref{eq:proof-corollary-sparse-dev} and~\eqref{eq:proof-sparse-complexity}, we obtain that under some $P \in \probas_{s,d}$, with probability at least $\delta$ one has
  \begin{align*}
    \kll{P}{\wh P_n}
    &\geq \max \bigg\{ \frac{s \log (e d/s)}{300 n}, \frac{\log (d) \log (1/\delta)}{14 n} \bigg\} \\
    &\geq \frac{20}{21} \frac{s \log (e d/s)}{300 n} + \frac{1}{21} \frac{\log (d) \log (1/\delta)}{14 n}
      \geq \frac{s \log (e d/s) + \log (d) \log (1/\delta)}{320 \, n}
      \, .
  \end{align*}

\section{Proof of Theorem~\ref{thm:deviation-missing-mass}}
\label{sec:proof-deviation-missing-mass}

We now provide the proof of Theorem~\ref{thm:deviation-missing-mass}.
First, note that if $\delta \leq e^{-n/6}$, then the right-hand side of~\eqref{eq:deviation-missing-mass} is greater than $1$, hence the inequality holds.
We now assume that $\delta \in (e^{-n/6}, 1)$.

\paragraph{Poisson sampling.}

As before, we let $(X_i)_{i \geq 1}$ denote an \iid sequence from $P$, and $N$ be an independent random variable with distribution $\poissondist (n/2)$.
In addition, we let $\wt N_j = \sum_{1 \leq i \leq N} \indic{X_i = j}$.
We work under the event $E= \{ N \leq n \}$, such that $\P (E) \geq 1-e^{-n/6} \geq 1-\delta$ (inequality~\eqref{eq:proba-poisson-sample}), and under which $\wt N_j \leq N_j$ for $1 \leq j \leq d$.
Thus, letting $\lambda_j = n p_j$ 
we have under $E$ that
\begin{equation}
  \label{eq:poissonized-missing-mass}
  U_n
  \leq \sum_{j \pp p_j < 1/n} p_j + \sum_{j \pp p_j \geq 1/n} p_j \ind \Big( N_j \leq \frac{n p_j}{4} \Big)
  \leq \sum_{j \pp p_j < 1/n} p_j + %
  \frac{1}{n} \sum_{j \pp \lambda_j \geq 1} \lambda_j \ind \Big( \wt N_j \leq \frac{\lambda_j}{4} \Big)
  \, .
\end{equation}
Also, recall that $\wt N_1, \dots, \wt N_d$ are independent, with $\wt N_j \sim \poissondist (\lambda_j/2)$ for $j= 1, \dots, d$.
We therefore need to control the last term in~\eqref{eq:poissonized-missing-mass}.

\paragraph{Multi-scale decomposition and domination.}
We collect terms in the sum as follows:
\begin{align*}
  \sum_{j \pp \lambda_j \geq 1} \lambda_j \ind \Big( \wt N_j \leq \frac{\lambda_j}{4} \Big)
  &= \sum_{k\geq 0} \sum_{j \pp 2^k \leq \lambda_j < 2^{k+1}} \lambda_j \ind \Big( \wt N_j \leq \frac{\lambda_j}{4} \Big) \\
  &\leq \sum_{k\geq 0} 2^{k+1} \sum_{j \pp 2^k \leq \lambda_j < 2^{k+1}} \ind \Big( \wt N_j \leq \frac{\lambda_j}{4} \Big)
    \, .
\end{align*}
Now, let $d_k' = \abs{\set{1 \leq j \leq d : 2^k \leq \lambda_j < 2^{k+1}}}$ for each $k \geq 0$, and let $(P_k)_{k \geq 0}$ be independent random variables with $P_k$ following the binomial distribution $\binomdist (d_k', e^{-2^k/14})$.
Since $\P (\wt N_j \leq \lambda_j /4) \leq e^{-\lambda_j/14} \leq e^{-2^k/14}$ for each $k \geq 0$ and $j$ such that $2^{k} \leq \lambda_j < 2^{k+1}$, and since the random variables $\wt N_j$ are independent, it follows from Lemma~\ref{lem:stoch-domin} that
\begin{equation*}
  \sum_{j \pp 2^k \leq \lambda_j < 2^{k+1}} \ind \Big( \wt N_j \leq \frac{\lambda_j}{4} \Big)
  \mleq P_k
  \, .
\end{equation*}
The previous inequalities together with another application of Lemma~\ref{lem:stoch-domin} imply that
\begin{equation}
  \label{eq:multi-scale-decomposition}
  \sum_{j \pp \lambda_j \geq 1} \lambda_j \ind \Big( \wt N_j \leq \frac{\lambda_j}{4} \Big)
  \mleq \sum_{k \geq 0} 2^{k+1} P_k
  \, .
\end{equation}

\paragraph{Control for individual scales.}

The following key lemma controls the tails of individual scales in the sum~\eqref{eq:multi-scale-decomposition}.

\begin{lemma}
  \label{lem:underestimated-single-scale}
  For every $k \in \N$ and $t \geq 0$, one has
  \begin{equation}
    \label{eq:tail-individual-scale}
    \P \Big( 2^k P_k \geq 56 d_k' e^{-2^k/56} + 28 e t \Big)
    \leq e^{-t}
    \, .
  \end{equation}  
\end{lemma}

\begin{proof}
  Since $P_k \sim \binomdist (d_k', e^{-2^k/14})$, Bennett's inequality~\cite[Theorem~2.9 p.~35]{boucheron2013concentration}
  implies that for any $u \geq d_k' e^{-2^k/28}$,
  \begin{equation*}
    \P \big( P_k \geq e u \big)
    \leq \exp \bigg\{ - d_k' e^{-2^k/14} h \bigg( \frac{e u}{d_k' e^{-2^k/14}} \bigg) \bigg\}
      \, .
  \end{equation*}
  Now since $e u/ (d_k' e^{-2^k/14})  \geq e (d_k' e^{-2^k/28})/(d_k' e^{-2^k/14}) \geq e \cdot e^{2^k/28}$, and since
  $h (t) \geq e^{-1} t \log t$ when $t \geq e$ (Lemma~\ref{lem:entropy-function}),
  the previous bound implies that
  \begin{equation*}
    \P \big( P_k \geq e u \big)
    \leq \exp \Big\{ - e^{-1} \cdot e u \log \Big( e \cdot e^{2^k/28} \Big) \Big\}
    \leq \exp \Big\{ - \frac{2^k u}{28} \Big\}
    \, .
  \end{equation*}
  Hence, letting $u = 4 t/2^k$, the inequality
  \begin{equation}
    \label{eq:proof-underestimated-poisson-1}
    \P \big( 2^k P_k \geq 4 e t \big)
    = \P \big( P_k \geq e u %
    \big)
    \leq \exp \Big\{ - \frac{2^k u}{28} \Big\}
    \leq %
    e^{-t/7}
  \end{equation}
  holds for every $t \geq 2^k d_k' e^{-2^k/28}/4$.
  Since
  \begin{equation*}
    \frac{2^k \cdot d_k' e^{-2^k/28}}{4}
    \leq \frac{56 e^{-1} e^{2^k/56} \cdot d_k' e^{-2^k/28}}{4}
    = 14 e^{-1} d_k' e^{-2^k/56} = t_0
    \, ,
  \end{equation*}
  the bound~\eqref{eq:proof-underestimated-poisson-1} holds for any $t \geq t_0$.
  Thus, for any $t \geq 0$,
  \begin{equation*}
    \P \big( 2^k P_k \geq 4e (t_0 + 7 t) \big)
    \leq e^{-(t_0 + 7t)/7}
    \leq e^{-t}
    \, ,
  \end{equation*}
  which is precisely the claimed inequality~\eqref{eq:tail-individual-scale}.
\end{proof}

Lemma~\ref{lem:underestimated-single-scale} states that $2^k P_k$ is stochastically dominated by $56 d_k' e^{-2^k/56} + 28 e E_k$, where $E_k \sim \expdist (1)$ is an exponential random variable.
In addition, since $\lambda_j = np_j \leq n$ for $j = 1, \dots, d$, we have $d_k' = 0$ (and thus $P_k = 0$) for $k> \log_2 n$.
Thus, letting $(E_k)_{k \in \N}$ be independent exponential random variables, inequality~\eqref{eq:multi-scale-decomposition} implies that
\begin{equation}
  \label{eq:underestimated-dominated-exponential}
  \sum_{j \pp \lambda_j \geq 1} \lambda_j \ind \Big( \wt N_j \leq \frac{\lambda_j}{4} \Big)
  \mleq 112 \sum_{0 \leq k \leq \log_2 n} d_k' e^{-2^k/56} + 56 e \sum_{0 \leq k \leq \log_2 n} E_k
  \, .
\end{equation}
In addition, the first term in the r.h.s.\ of~\eqref{eq:underestimated-dominated-exponential} is controlled by $c s_{n/112}^{\circ}$ for some constant $c$ (see below).
The second term may also be controlled using a deviation bound for sums of exponential random variables (that is, gamma random variables)~\cite[p.~28]{boucheron2013concentration}, which gives, with probability $1-\delta$,
\begin{equation*}
  \sum_{0 \leq k \leq \log_2 n} E_k
  \lesssim \log n + \log (1/\delta)
  \, .
\end{equation*}
While already non-trivial and in fact near-optimal, the resulting bound on the missing mass features an additional $\log n$ term, which is suboptimal in some (rather extreme) cases.

In order to address this sub-optimality, we need to account more carefully for the contribution of certain scales to the sum $\sum_{k \geq 0} 2^k P_k$.

\paragraph{Accounting for ``rarely contributing'' scales.}

We now define two types of ``scale indices'' $k\in \N$.
Specifically, we let
\begin{equation}
  \label{eq:typical-scales}
  A
  = \Big\{ k \in \N : d_k' e^{-2^k/56} \geq e \Big\}
  \, .
\end{equation}
Intuitively speaking, since $\E [P_k] = d_k' e^{-2^k/14}$, the scale indices $k \in A$ are those for which $P_k$ is often larger than $1$, \ie non-zero---with the technical caveat that we have changed the exponent in~\eqref{eq:typical-scales}.
On the contrary, indices $k \in \N \setminus A$ are those for which typically $P_k = 0$.
However, given that there may be several such terms and that we account for low-probability events, some of these terms may be positive with small probability.
Hence, they may contribute to the sum, especially if their coefficient $2^k$ is large.
Thus, they should also be accounted for.

In what follows, we separately account for the contribution of indices $k \in A$ and $k \not\in A$.
First recall that, by Lemma~\ref{lem:underestimated-single-scale}, for any $k \in \N$ one has
\begin{equation}
  \label{eq:dominated-scale}
  2^k P_k
  \mleq 56 d_k' e^{-2^k/56} + 28 e E_k
\end{equation}
where $(E_k)_{k \in \N}$ are \iid exponential random variables.

We also recall (see~\cite[p.~28]{boucheron2013concentration}) that for any $\delta \in (0, 1)$ and finite subset $B \subset \N$,
with probability $1-\delta$ one has
\begin{equation}
  \label{eq:sum-exponentials}
  \sum_{k \in B} E_k
  < |B| + \sqrt{2 |B| \log (1/\delta)} + \log (1/\delta)
  \leq 2 |B| + 2\log (1/\delta)
  \, .
\end{equation}

When $k \in A$, one has $d_k' e^{-2^k/56} \geq e$, hence~\eqref{eq:dominated-scale} gives
\begin{equation}
  \label{eq:dominated-scale-typical}
  2^k P_k
  \mleq 56 d_k' e^{-2^k/56} + 56 e + 28 e (E_k - 2)
  \leq 112 d_k' e^{-2^k/56} + 28 e (E_k - 2)
  \, .
\end{equation}
Now applying the tail bound~\eqref{eq:sum-exponentials} to $B = A$ (which is finite since $|A| \leq \log_2 n$), we obtain with probability $1-\delta$,
\begin{equation*}
  \sum_{k \in A} (E_k - 2)
  < 2 |A| + 2 \log (1/\delta) - 2 |A|
  = 2 \log (1/\delta)
  \, ,
\end{equation*}
which combined with~\eqref{eq:dominated-scale-typical} (and Lemma~\ref{lem:stoch-domin}) gives
\begin{equation}
  \label{eq:proof-sum-scale-A}
  \P \bigg( \sum_{k \in A} 2^k P_k \geq 112 \sum_{k \in A} d_k' e^{-2^k/56} + 56 e \log (1/\delta) \bigg)
  \leq \delta
  \, .
\end{equation}

We now turn to the case $k \not\in A$, in which case $d'_ke^{-2^k/56} < e$.
Using that if $P \sim \poissondist (\lambda)$ then $\P (P \neq 0) = 1-e^{-\lambda} \leq \lambda$, this implies that
\begin{equation*}
  \P (P_k \neq 0)
  \leq d_k' e^{-2^k/14}
  = (d_k' e^{-2^k/56}) e^{-2^k 3/56}
  \leq e \cdot e^{-2^k/19}
  \, .
\end{equation*}
Now, let $k^* = \ceil{19 \log (1/\delta)} \leq 20 \log (1/\delta)$ (as $\delta \leq e^{-1}$), so that $e^{-2^{k^*}/19} \leq \delta$.
Using the previous inequality, we may bound
\begin{align*}  
  \P \bigg( \sum_{k \not\in A, \, k \geq k^*} 2^k P_k \neq 0 \bigg)
  &\leq \sum_{k \not\in A, \, k \geq k^*} \P (P_k \neq 0)
    \leq \sum_{k \geq k^*} e \cdot e^{-2^k/19} \\
  &= e \sum_{l \geq 0} e^{- 2^l 2^{k^*}/19}
    \leq e \sum_{l \geq 0} \big( e^{-2^{k^*}/19} \big)^{l+1} \\
  &\leq e \sum_{l \geq 0} \delta^{l+1}
    = \frac{e \delta}{1-\delta}
    \leq \frac{e}{1-e^{-1}} \delta
    \, ,
\end{align*}
so that
\begin{equation}
  \label{eq:proof-nonzero-notA}
  \P \bigg( \sum_{k \not\in A, \, k \geq k^*} 2^k P_k \neq 0 \bigg)
  \leq 5 \delta
  \, .
\end{equation}
Finally, it remains to control indices $0 \leq k < k^*$ such that $k \not\in A$.
For this, we combine the domination~\eqref{eq:dominated-scale} with the tail bound~\eqref{eq:sum-exponentials} to obtain:
\begin{align*}
  \P \bigg( \sum_{k \not\in A,\, k < k^*} 2^k P_k \geq 56 \sum_{k \not\in A,\, k < k^*} d_k' e^{-2^k/56} + 56 e \{k^* + \log (1/\delta)\} \bigg)
  \leq \delta
  \, .
\end{align*}
Since $k^* \leq 20 \log (1/\delta)$, we conclude that
\begin{equation}
  \label{eq:proof-notA-small}
  \P \bigg( \sum_{k \not\in A,\, k < k^*} 2^k P_k \geq 56 \sum_{k<k^*,\, k \not\in A} d_k' e^{-2^k/56} +
  1176 e
  \log (1/\delta) \bigg)
  \leq 7\delta
  \, .
\end{equation}
Combining inequalities~\eqref{eq:proof-sum-scale-A},~\eqref{eq:proof-nonzero-notA} and~\eqref{eq:proof-notA-small}  through a union bound to control the sum of the three terms, we obtain
\begin{equation*}
  \P \bigg( \sum_{k \geq 0} 2^k P_k
  \geq 168 \sum_{k \geq 0} d_k' e^{-2^k/56} +
  1232 e \log (1/\delta) \bigg)
  \leq 7 \delta
  \, .
\end{equation*}

\paragraph{Conclusion.}

Together with inequality~\eqref{eq:multi-scale-decomposition}, the previous bound implies that
\begin{equation*}
  \P \bigg( \sum_{j \pp \lambda_j \geq 1} \lambda_j \ind \Big( \wt N_j \leq \frac{\lambda_j}{4} \Big)
  \geq 336 \sum_{k \geq 0} d_k' e^{-2^k/56} +
  2464 e \log (1/\delta) \bigg)
  \leq 7 \delta
  \, .
\end{equation*}
Now, by definition of $d_k'$ one has
\begin{align*}
  &\sum_{k \geq 0} d_k' e^{-2^k/56}
  = \sum_{k \geq 0} \sum_{j \pp 2^k \leq n p_j < 2^{k+1}} e^{-2^k/56}  \\
  &\leq \sum_{k \geq 0} \sum_{j \pp 2^k \leq n p_j < 2^{k+1}} e^{-np_j/112} 
  = \sum_{j \pp p_j \geq 1/n} e^{-np_j/112}
    \, .
\end{align*}
Plugging the previous inequalities into the decomposition~\eqref{eq:poissonized-missing-mass} (which holds except for an event of probability at most $\delta$), we get that with probability at least $1-8\delta$,
\begin{align*}
  U_n
  &< \frac{1}{n} \bigg[ \sum_{j\pp p_j < 1/n} (np_j) + 336 \sum_{j \pp p_j \geq 1/n} e^{-np_j/112} + 2464 e \log (1/\delta) \bigg] \\
  &\leq \frac{336 \, s_{n/112}^\circ (P) + 2500 e \log (1/\delta)}{n}
  \, .
\end{align*}

\section{Proof of Proposition~\ref{prop:expected-sparse-loo}}
\label{sec:proof-expected-sparse-loo}

  Let $X_{n+1}$ be a new observation from $P$, independent from $X_1, \dots, X_n$.
For $j = 1, \dots, d$, denote by $\wt N_j = \sum_{i=1}^{n+1} \indic{X_i = j}$ the count of class $j$ and by $\wt D = \sum_{j=1}^d \indic{\wt N_j \geq 1}$ the number of distinct classes in the extended sample $(X_1, \dots, X_{n+1})$.
Also, let $\wt N = (\wt N_j)_{1 \leq j \leq d}$.
Since the joint distribution of $(X_1, \dots, X_{n+1})$ is exchangeable (and permuting indices does not change $\wt N$), for every $j = 1, \dots, d$ and $i=1, \dots, n$ one has $\P (X_{n+1} = j | \wt N) = \P (X_i = j | \wt N)$; hence,
\begin{align}
  \label{eq:loo-cond-proba}
  \P (X_{n+1} = j | \wt N)
  &= \frac{1}{n+1} \sum_{i=1}^n \P (X_i = j | \wt N)
  = \frac{1}{n+1} \E \bigg[ \sum_{i=1}^{n+1} \indic{X_i = j} \Big| \wt N \bigg] \nonumber \\
  &= \frac{\E [\wt N_j | \wt N]}{n+1}
  = \frac{\wt N_j}{n+1}
  \, .
\end{align}
Now, observe that for any probability distribution $Q \in \probas_d$, one has 
$\kll{P}{Q} = L (Q) - L (P)$, where $L (Q) = \E [ \ell (Q, X) ]$ with $X \sim P$, and where $\ell$ stands for the logarithmic loss $\ell (Q, x) = - \log Q (\{x\})$.
In addition, %
since $\adp$ is independent of $X_{n+1}$ one has 
\begin{equation*}
  \E [ \ell (\adp, X_{n+1}) ]
  = \E \big[ \E [ \ell (\adp, X_{n+1}) | \adp ] \big]
  = \E [ L (\adp) ]
  \, .
\end{equation*}
On the other hand, let $\wt P = (\wt p_j)_{1 \leq j \leq d}$ denote the maximum likelihood distribution on the extended sample $(X_1, \dots, X_{n+1})$, defined by
\begin{equation*}
  \wt P
  = \argmin_{P' \in \probas_d} \bigg\{ \frac{1}{n+1} \sum_{i=1}^{n+1} \ell (P', X_i) \bigg\}
  \, .
\end{equation*}
It is a standard fact that $\wt P$ is the empirical distribution, namely $\wt p_j = \wt N_j / (n+1)$.
In addition, by definition of $\wt P$, one has
\begin{equation*}
  L (P)
  = \E \bigg[ \frac{1}{n+1} \sum_{i=1}^{n+1} \ell (P, X_i) \bigg]
  \geq \E \bigg[ \frac{1}{n+1} \sum_{i=1}^{n+1} \ell (\wt P, X_i) \bigg]
  = \E \big[ \ell (\wt P, X_{n+1}) \big]
  \, ,
\end{equation*}
where the last step used the fact that the distribution of the vector $(X_1, \dots, X_{n+1})$ is invariant under permutation, and that $\wt P$ is also invariant under permutation.
Putting the previous inequalities together gives:
\begin{equation}
  \label{eq:expect-kl-loo}
  \E \big[ \kll{P}{\adp} \big]
  = \E \big[ L (\adp) \big] - L (P)
  \leq \E \big[ \ell (\adp, X_{n+1}) - \ell (\wt P, X_{n+1}) \big]
  \, .
\end{equation}
Recall that $\wh p_j = (N_j + D_n/d)/(n+D_n)$ while $\wt p_j = \wt N_j/(n+1)$.
It then follows from~\eqref{eq:expect-kl-loo} that
\begin{align}
  \label{eq:ineq-kl-loo-count}
  \E \big[ \kll{P}{\adp} \big]
  &\leq \E \bigg[ \sum_{j=1}^d \big\{ \ell (\adp, X_{n+1}) - \ell (\wt P, X_{n+1}) \big\} \indic{X_{n+1} = j} \bigg] \nonumber \\
    &= \E \bigg[ \sum_{j=1}^d \log \bigg( \frac{\wt N_j / (n+1)}{(N_j + D_n/d)/(n+D_n)} \bigg) \indic{X_{n+1} = j} \bigg] \nonumber \\
    &= \E \bigg[ \sum_{j=1}^d \log \bigg( \frac{\wt N_j}{N_j + D_n/d} \bigg) \indic{X_{n+1} = j} + \log \Big( \frac{n + D_n}{n + 1} \Big) \bigg] \nonumber \\
    &\leq \E \bigg[ \sum_{j=1}^d \E \bigg[ \log \bigg( \frac{\wt N_j}{\wt N_j - 1 + D_n/d} \bigg) \indic{X_{n+1} = j} \Big| \wt N \bigg] \bigg] + \frac{\E [D_n] - 1}{n+1}
    \, ,
\end{align}
where in the last inequality we used that if $X_{n+1} = j$, then $N_j = \wt N_j - 1$, and that $\log [ (n+D_n)/(n+1) ] = \log [ 1 + (D_n-1)/(n+1)] \leq (D_n - 1)/(n+1)$.

Consider first the case where $\wt N_j = 1$.
In this case and if $X_{n+1} = j$, then $D_n = \wt D - 1$ (since $X_{n+1} = j$ does not appear in the first $n$ observations) and hence:
\begin{align}
  \label{eq:loo-bound-missing}
  \E \bigg[ \log \bigg( \frac{\wt N_j}{\wt N_j - 1 + D_n/d} \bigg) \indic{X_{n+1} = j} \Big| \wt N \bigg]
  &= \E \bigg[ \log \bigg( \frac{1}{(\wt D - 1)/d} \bigg) \indic{X_{n+1} = j} \Big| \wt N \bigg] \nonumber \\
  &= \log \Big( \frac{d}{\wt D - 1} \Big) \cdot \P \big( X_{n+1} = j | \wt N \big) \nonumber \\
  &= \frac{1}{n+1} \log \Big( \frac{d}{\wt D - 1} \Big)
\end{align}
where we used~\eqref{eq:loo-cond-proba} (and $\wt N_j = 1$) in the last equality.

Consider now the case $\wt N_j \geq 2$.
We then bound
\begin{align}
  \label{eq:loo-bound-not-missing}
  \E \bigg[ \log \bigg( \frac{\wt N_j}{\wt N_j - 1 + D_n/d} \bigg) \indic{X_{n+1} = j} \Big| \wt N \bigg] 
  &\leq \E \bigg[ \log \bigg( \frac{\wt N_j}{\wt N_j - 1} \bigg) \indic{X_{n+1} = j} \Big| \wt N \bigg] \nonumber \\
  &= \log \bigg( \frac{\wt N_j}{\wt N_j - 1} \bigg) \cdot \P (X_{n+1} = j | \wt N) \nonumber \\
  &= \frac{\wt N_j}{n+1} \cdot \log \bigg( \frac{\wt N_j}{\wt N_j - 1} \bigg) \nonumber \\
  &\leq \frac{\log 4}{n+1}
    \, ;
\end{align}
in the last inequality, we used that the function $\phi : t \mapsto t \log [t/(t-1)]$, such that
\begin{equation*}
  \phi'(t) = \log [t/(t-1)] + t [ 1/t - 1/(t-1) ] = \log [ 1 + 1/(t-1)] - 1/(t-1) \leq 0
\end{equation*}
for $t> 1$, is decreasing, so that $\phi (\wt N_j) \leq \phi (2) = 2 \log 2 =\log 4$.

Denote now by $\wt C_1 = \sum_{j=1}^d \indic{\wt N_j = 1}$ and $\wt C_2 = \sum_{j=1}^d \indic{\wt N_j \geq 2}$, so that $\wt D = \wt C_1 + \wt C_2$.
Plugging the bounds~\eqref{eq:loo-bound-missing} and~\eqref{eq:loo-bound-not-missing} into inequality~\eqref{eq:ineq-kl-loo-count}, we obtain (using that if $\wt N_j = 0$, then $\indic{X_{n+1} = j} = 0$):
\begin{align*}
  &\E \big[ \kll{P}{\adp} \big] \nonumber \\
  &\leq \E \bigg[ \sum_{j=1}^d \E \bigg[ \log \bigg( \frac{\wt N_j}{\wt N_j - 1 + D_n/d} \bigg) \indic{X_{n+1} = j} \Big| \wt N \bigg] \big\{ \indic{\wt N_j = 1} + \indic{\wt N_j \geq 2} \big\} \bigg] + \frac{\E [D_n] - 1}{n+1} \nonumber \\
  &\leq \E \bigg[ \sum_{j=1}^d \frac{1}{n+1} \log \Big( \frac{d}{\wt D - 1} \Big) \indic{\wt N_j = 1} + \sum_{j=1}^d \frac{\log 4}{n+1} \indic{\wt N_j \geq 2} \bigg] + \frac{\E [D_n] - 1}{n+1} \nonumber \\
  &\leq \frac{1}{n+1} \E \bigg[ \wt C_1 \log \Big( \frac{d}{\wt D - 1} \Big) + \wt C_2 \log 4 + D_n - 1 \bigg] %
    \\
  &\leq \frac{1}{n+1} \E \bigg[ \wt C_1 \log \Big( \frac{e d}{\wt C_1} \Big) + {\wt C_2} \log 4 + {D_n - 1} \bigg]
    \, .
\end{align*}
(In the last bound, we have used that if $\wt D = 1$, then only one class appears $n+1 > 1$ times, hence $\wt C_1 = 0$ and the first term vanishes; on the other hand, if $\wt D\geq 2$, then $\wt D - 1
\geq \wt D/e
\geq \wt C_1/e$.)

By concavity of the map $x \mapsto - x \log x$ on $\R_+^*$, we deduce that
\begin{equation}
  \label{eq:proof-expected-KL-loo}
  \E_P \big[ \kll{P}{\adp} \big]
  \leq %
  \frac{1}{n+1} \bigg\{ \E_P [\wt C_1] \log \Big( \frac{e d}{\E_P [\wt C_1]} \Big) + \E_P [\wt C_2] \log 4 + {\E_P [D_n] - 1} \bigg\}
  \, .
\end{equation}
Now, note that (using Lemma~\ref{lem:expected-missing-mass} for the last inequality)
\begin{equation*}
  \E_P [\wt C_1]
  = \sum_{j=1}^d \P (\wt N_j = 1)
  = \sum_{j=1}^d n p_j (1-p_j)^n
  \leq \sum_{j=1}^d n p_j e^{-np_j}
  = s_n^\bullet (P)
  \leq 2 s_{n/2}^\circ (P)  
  \, ,
\end{equation*}
and in addition $\E_P [\wt C_2] \leq \E_P [D_n] \leq s_n = s_n (P)$.
Using that the map $x \mapsto x \log (e d/x)$ is increasing over $[0,d]$, we deduce that
\begin{align*}
  &\E_P [\wt C_1] \log \Big( \frac{e d}{\E_P [\wt C_1]} \Big) + \E_P [\wt C_2] \log 4 + {\E_P [D_n] } \\
  &\leq 2 s_{n/2}^\circ \log \Big( \frac{e d}{s_{n/2}^\circ} \Big) + \E_P [\wt C_1 + \wt C_2] \log 4 + \E_P [ D_n]  \\
  &\leq 2 s_{n/2}^\circ \log \Big( \frac{e d}{s_{n/2}^\circ} \Big) + (1 + \log 4) s_n
    \, .
\end{align*}
Plugging this inequality into~\eqref{eq:proof-expected-KL-loo} and bounding $1 + \log 4 \leq 2.4$ leads to the desired bound~\eqref{eq:expected-sparse-loo}.

\section{Technical lemmata}
\label{sec:technical-lemmata}

In this section, we gather simple technical results that are used at various places in the proofs.

\begin{lemma}
  \label{lem:entropy-function}
  Define the function $h : \R^+ \to \R^+$ by $h (t) = t \log t - t + 1$ for $t > 0$, and $h (0) = 1$.
  Then $h$ is continuous and convex on $\R^+$. 
  In addition, $h (t) \leq t \log t$ for $t \geq 1$, while $h(t) \geq e^{-1} t \log t$ for $t \geq e$.
  Finally, $h (t) \leq (t-1)^2$ for any $t \in \R^+$.
\end{lemma}

\begin{proof}
  Continuity of $h$ is straightforward, while convexity comes from the fact that $h''(t) = 1/t > 0$ for $t > 0$.
  The inequality $h (t) \leq t \log t$ for $t \geq 1$ is immediate.
  We now turn to the lower bound $h (t) \geq e^{-1} t \log t$ for $t \geq e$.
  By convexity of the map $t \mapsto t \log t$ over $\R_+^*$, the quantity $t \log (t)/(t-1)$ increases in $t$, thus for $t \geq e$ one has $t \log (t) / (t-1) \geq e / (e-1)$, hence
  \begin{equation*}
    h (t)
    = t \log t - (t-1)
    \geq \Big( 1 - \frac{e-1}{e} \Big) t \log t
    = e^{-1} t \log t
    \, .
  \end{equation*}

  We conclude with the proof of the bound $h(t) \leq (t-1)^2$.
  To this end, let $f (t) = h (t)/(t-1)^2$ for $t \in \R^+ \setminus \{ 1 \}$, and $f(1) = 1/2$.
  It is easily verified that $f$ is continuous on $\R^+$ and differentiable on $\R^+ \setminus \{1\}$.
  In addition,  $f'(t) = g(t)/(t-1)^3$ where $g(t) = 2 (t-1) - (t+1) \log t$ for $t > 0$.
  One has $g'(t) = - (t^{-1} - 1 - \log (t^{-1})) \leq 0$, and since $g (1) = 0$, this implies that $g \leq 0$ on $(0, 1]$ while $g \geq 0$ on $[1, + \infty)$.
  Therefore $f'(t) = g(t)/(t-1)^3 \leq 0$ for any $t \in \R^+ \setminus \{1\}$ and thus $f$ is decreasing on $\R^+$.
  We conclude by noting that $\lim_{t \to 0^+} f(t) = 1$, so that $f \leq 1$ on $\R_+^*$.
\end{proof}

\begin{lemma}
  \label{lem:D-function}
  For any $p \in \R_+^*$, the function $q \mapsto D (p, q)$ is strictly convex on $\R_+^*$, and reaches its minimum \textup(equal to $0$\textup) at $q = p$.
  Hence, it is decreasing on $(0, p]$ and increasing on $[p, + \infty)$.
  In addition, $D (p, q) \leq p \log (p/q)$ when $q \leq p$, and $D (p, q) \geq e^{-1} p \log (p/q)$ when $q \leq p/e$.
\end{lemma}

\begin{proof}
  The claims on convexity and monotonicity follow from the fact that $\frac{\partial D}{\partial q} (p, q) = - \frac{p}{q} + 1$ (which cancels out at $q = p$) and $\frac{\partial D}{\partial q} (p, q) = \frac{p}{q^2} > 0$.
  The inequalities on $D$ follow from the expression $D (p, q) = q h (p/q)$ and from Lemma~\ref{lem:entropy-function}.
\end{proof}

\begin{lemma}
  \label{lem:kl-lower-subset}
  Let $P = (p_1, \dots, p_d) \in \probas_d$ and $Q = (q_1, \dots, q_d) \in \probas_d$.
  For any subset $J \subset [d]$, we have
  \begin{equation}
    \label{eq:kl-lower-subset}
    \kll{P}{Q}
    \geq D \bigg( \sum_{j \in J} p_j, \sum_{j \in J} q_j \bigg)
    \, .
  \end{equation}
\end{lemma}

\begin{proof}
  We may assume, without loss of generality, that $q_j > 0$ for any $j \in J$.
  Indeed, either there exists a $j \in J$ such that $q_j = 0$ and $p_j > 0$, in which case the left-hand side of~\eqref{eq:kl-lower-subset} is $+ \infty$ and the inequality holds; or, for any $j \in J$ such that $q_j = 0$, one has $p_j = 0$: but in this case replacing $J$ by $J' = \set{j \in J : q_j > 0}$ does not affect the right-hand side of~\eqref{eq:kl-lower-subset}.

  Next, if $J = \varnothing$, the right-hand side of~\eqref{eq:kl-lower-subset} is $0$, thus the inequality holds.
  We now also assume that $J \neq \varnothing$.
  Then, by non-negativity and convexity of the function $h$ (Lemma~\ref{lem:entropy-function}), we have
  \begin{align*}
    \kll{P}{Q}
    &= \sum_{j=1}^d D (p_j, q_j)
      \geq \sum_{j \in J} q_j h \Big( \frac{p_j}{q_j} \Big)
      = \bigg( \sum_{j \in J} q_j \bigg) \sum_{j\in J} \frac{q_j}{\sum_{k \in J} q_k} h \Big( \frac{p_j}{q_j} \Big)
    \\
    &\geq \bigg( \sum_{j \in J} q_j \bigg) h \bigg( \sum_{j\in J} \frac{q_j}{\sum_{k \in J} q_k} \cdot \frac{p_j}{q_j} \bigg)
      = D \bigg( \sum_{j \in J} p_j, \sum_{j \in J} q_j \bigg)
      \, .
      \qedhere
  \end{align*}
\end{proof}

We also recall the following standard Poisson tail bound (\eg,~\cite[p.~23]{boucheron2013concentration}):
\begin{lemma}
  \label{lem:poisson-tail}
  Let $\lambda \in \R^+$ and $N \sim \poissondist (\lambda)$.
  For any $\mu \in \R^+$ such that $\mu \geq \lambda$, one has
  \begin{equation}
    \label{eq:poisson-upper-tail}
    \P (N \geq \mu)
    \leq \exp (- D (\mu, \lambda))
    \, .
  \end{equation}
  In addition, for any $\mu \in \R^+$ such that $\mu \leq \lambda$, one has
  \begin{equation}
    \label{eq:poisson-lower-tail}
    \P (N \leq \mu)
    \leq \exp ( - D (\mu, \lambda))
    \, .
  \end{equation}
\end{lemma}

\begin{definition}
  \label{def:stoch-domin}
  Let $X,Y$ be real random variables.
  We say that $X$ is \emph{stochastically dominated} by $Y$, denoted $X \mleq Y$, if $\P (X \geq t) \leq \P (Y \geq t)$ for every $t \in \R$.
\end{definition}

\begin{lemma}
  \label{lem:stoch-domin}
  Let $X_1, \dots, X_n$ and $Y_1, \dots, Y_n$ be independent real random variables, such that $X_i \mleq Y_i$ for $i=1, \dots, n$.
  Then $\sum_{i=1}^n X_i \mleq \sum_{i=1}^n Y_i$.
\end{lemma}

\begin{proof}
  For $i=1, \dots, n$, let $F_i (t) = \P (X_i \leq t)$ be the cumulative distribution function (c.d.f.) of $X_i$, and $F_i^+ (u) = \inf \{ t \in \R : F_i (t) \geq u \}$ for $u \in (0,1)$ be its right-continuous inverse.
  Likewise, for $i=1, \dots, n$, let $G_i$ be the c.d.f.~of $Y_i$ and $G_i^+$ its right-continuous inverse.
  Since $X_i \mleq Y_i$, we have $F_i \geq G_i$ and thus $F_i^+ \leq G_i^+$.

  Now, let $U_1, \dots, U_n$ be independent random variables that are uniformly distributed on $[0,1]$.
  Then $(F_i^+ (U_i))_{1 \leq i \leq n}$ has the same distribution as $(X_i)_{1\leq i \leq n}$, while $(G_i^+ (U_i))_{1 \leq i \leq n}$ has the same distribution as $(Y_i)_{1\leq i \leq n}$, thus for any $t \in \R$,
  \begin{align*}
    \P \big( X_1 + \dots + X_n \geq t \big)
    &= \P \big( F_1^+ (U_1) + \dots + F_n^+ (U_n) \geq t \big) \\
    &\leq \P \big( G_1^+ (U_1) + \dots + G_n^+ (U_n) \geq t \big)
      = \P \big( Y_1 + \dots + Y_n \geq t \big)
      \, .
  \end{align*}
  Hence $X_1 + \dots + X_n \mleq Y_1 + \dots + Y_n$.
\end{proof}

We will also use the following concentration inequality, due to Ben-Hamou, Boucheron, and Ohannessian \cite{benhamou2017concentration}, for the number $D_n = \sum_{j=1}^d \indic{N_j \geq 1}$ of distinct classes in the sample.

\begin{lemma}
  \label{lem:deviation-occupancy}
  Let $s_n' = s_n' (P) = \E_P [D_n] = \sum_{j=1}^d \{ 1 - (1-p_j)^n \}$.
  For any $s \in \R^+$, the following holds: if $s > s_n'$, then
  \begin{equation}
    \label{eq:deviation-occupancy-upper}
    \P_P (D_n \geq s)
    \leq \exp (- D (s, s_n'))
    \, ;
  \end{equation}
  in addition, if $s < s_n'$, then
  \begin{equation}
    \label{eq:deviation-occupancy-lower}
    \P_P (D_n \leq s)
    \leq \exp (- D (s, s_n'))
    \, .
  \end{equation}
  Finally, for every $\delta \in (0,1)$,
  \begin{equation}
    \label{eq:bernstein-occupancy}
    \P_P \Big( D_n \geq 2 \big\{ s_n' + \log (1/\delta) \big\} \Big)
    \leq \delta
    \, .
  \end{equation}
\end{lemma}

\begin{proof}
  Applying~\cite[Proposition~3.4]{benhamou2017concentration} with $r = 1$ gives, for any $\theta \in \R$,
  \begin{equation*}
    \log \E \big[ e^{\theta (D_n - s_n')} \big]
    \leq s_n' (e^\theta - \theta - 1)
    \, .
  \end{equation*}
  Applying the standard Chernoff method to this Poisson-type moment generating function gives the tail bounds~\eqref{eq:deviation-occupancy-upper} and~\eqref{eq:deviation-occupancy-lower} (see, \eg,~\cite[p.~23]{boucheron2013concentration}).
  To obtain the bound~\eqref{eq:bernstein-occupancy}, further relax the m.g.f.~bound above by noting that for $\theta \in (0, 1)$, one has $e^{\theta} - \theta -1 = \sum_{k \geq 2} \frac{\theta^k}{k!} \leq \sum_{k\geq 2} \frac{\theta^k}{2} = \frac{\theta^2}{2(1-\theta)}$, and apply the sub-gamma tail bound~\cite[p.~29]{boucheron2013concentration} to conclude that, with probability at least $1-\delta$,
  \begin{equation*}
    D_n
    < s_n' + \sqrt{2 s_n' \log(1/\delta)} + \log (1/\delta)
    \leq 2 \big( s_n' + \log (1/\delta) \big)
    \, .
  \end{equation*}
  This concludes the proof.
\end{proof}

Finally, the following lemma was used in the proof of the decomposition of Lemma~\ref{lem:decomp-risk}.

\begin{lemma}
  \label{lem:kl-hellinger-bounded}
  Let $C \geq 4$, and define $\phi (t) = (t \log t -t +1)/(\sqrt{t}-1)^2$ for $t \in \R_+ \setminus \{1\}$, extended by continuity by $\phi(1)=2$.
  For every $p, q \in \R^+$ such that $q \geq p/C$, one has
  \begin{equation}
    \label{eq:kl-hellinger}
    (\sqrt{p} - \sqrt{q})^2
    \leq D (p, q)
    \leq \phi (C) (\sqrt{p} - \sqrt{q})^2
    \leq 4 \log (C) (\sqrt{p} - \sqrt{q})^2
    \, .
  \end{equation}
\end{lemma}

\begin{proof}
  If $q = 0$, then by assumption $p = 0$ as well, hence all terms in~\eqref{eq:kl-hellinger} are equal to $0$ and the inequalities hold; the same holds when $p=q$.
  We thus assume $q > 0$ and $p \neq q$.
  In this case, we have $D (p, q) / (\sqrt{p} - \sqrt{q})^2 = \phi (p/q)$.
  Now, a direct computation shows that the map $t \mapsto \phi (t^2)$ is continuously differentiable on $\R_+^*$, with derivative $t \mapsto 2 (t-1)^{-3} \psi (t)$ where $\psi (t) = t^2 - 2t\log t -1$.
  Since $\psi'(t) = 2 (t - 1 - \log t) \geq 0$ for any $t > 0$, the function $\psi$ is non-decreasing, and since $\psi(1) = 0$ we deduce that $\psi \leq 0$ on $(0, 1]$ and $\psi \geq 0$ on $[1, +\infty)$.
  Hence $\frac{\di}{\di t} \phi (t^2) = 2 (t-1)^{-3} \psi(t) \geq 0$ for any $t > 0$, $t \neq 1$, thus by continuity of $\phi$ at $0$ and $1$, the function $\phi$ is non-decreasing on $\R^+$.

  Since $\phi (0) = 1$ and $0 \leq p/q \leq C$, we deduce that
  $(\sqrt{p} - \sqrt{q})^2 \leq D (p, q) \leq \phi (C) (\sqrt{p} - \sqrt{q})^2$.
  Finally, since $C \geq 4$ we have
  \begin{equation*}
    \phi (C)
    = \frac{C \log C - C + 1}{C (1 - 1/\sqrt{C})^2}
    \leq \frac{C \log C}{C (1-1/\sqrt{4})^2}
    = 4 \log C
    \, ,
  \end{equation*}
  which concludes the proof.
\end{proof}

\begingroup
\small
\newcommand{\etalchar}[1]{$^{#1}$}

\endgroup


\end{document}